\documentclass[11pt]{amsart}
\usepackage[utf8]{inputenc}
\usepackage[T1]{fontenc}
\usepackage{aeguill}
\usepackage{xspace} 
\usepackage{mathrsfs} 
\usepackage{amsfonts}
\usepackage{url}
\usepackage{relsize}
\usepackage{accents}
\usepackage{bbm}
\usepackage[makeroom]{cancel}
\usepackage{amsmath,amssymb,amsthm,calligra,mathrsfs}
\usepackage{epsfig}
\usepackage[matrix,arrow]{xy}

\usepackage{mathbbol}
\usepackage{longtable}
\usepackage{stmaryrd}
\usepackage[matrix,arrow]{xy}
\usepackage{multicol}
\usepackage{array}
\usepackage{multirow}
\usepackage{tikz-cd}
\usepackage{rank-2-roots}
\usepackage{fancybox} 
\usepackage{hyperref}
\definecolor{puces}{cmyk}{0,0,0,.5} 

\theoremstyle{plain}

\newtheorem{thm}{Theorem}[section]
\newtheorem{lem}[thm]{Lemma}
\newtheorem{coro}[thm]{Corollary}
\newtheorem{prop}[thm]{Proposition}

\theoremstyle{remark} 

\usepackage{xcolor}
\newtheorem{rem}[thm]{Remark} 
\newtheorem{rems}[thm]{Remarks}

\usepackage{xcolor}

\DeclareMathOperator{\Hom}{\mathscr{H}\text{\kern -3pt {\calligra\large om}}\,}

\newcommand{\car}{\mbox{\relsize{-5}$\Box$}}

\usetikzlibrary{shapes.geometric}
\numberwithin{equation}{section}
    \author{Emilien Zabeth}
\address{Université Clermont Auvergne, CNRS, LMBP\\
F-63000 Clermont-Ferrand, France}
\email{emilien.zabeth@uca.fr}
\begin{document}
\title{Block decomposition via the geometric Satake equivalence}
\begin{abstract}
We give a new proof for the description of the blocks in the category of representations of a reductive algebraic group $\mathbf{G}$ over a field of positive characteristic $\ell$ (originally due to Donkin), by working in the Satake category of the Langlands dual group and applying Smith-Treumann theory as developed by Riche and Williamson. On the representation theoretic side, our methods enable us to give a bound for the length of a minimum chain linking two weights in the same block, and to give a new proof for the block decomposition of a quantum group at an $\ell$-th root of unity.
\end{abstract}

\setcounter{tocdepth}{1}

\maketitle
	\tableofcontents
	\section{Introduction}
	In this paper, we give a proof for the description of the block decomposition of a reductive algebraic group over a field of prime characteristic $\ell$ using the geometry of the affine Grassmannian. We apply results from \cite{RW22}, where Riche and Williamson used and developed Smith-Treumann theory for sheaves to give a geometric proof of the Linkage principle, which is the first step towards the block decomposition. Most of our paper is dedicated to the study of equivalence classes on some subsets of the affine Weyl group. These equivalence classes will be defined using homomorphisms between indecomposable parity complexes on some partial affine flag varieties, and we will show that their description implies the desired description for the block decomposition. Our methods will moreover allow us to give a bound for the length of a minimum chain linking two weights in the same block. Finally, using the fact that some of our techniques work in any characteristic (for the field of coefficients of our sheaves), we will give a proof for the description of the block decomposition of a quantum group at an $\ell$-th root of unity.
	
	In the rest of this introduction, we start by recalling the description of the blocks for a simply-connected simple algebraic group (due to Donkin), before giving a brief summary of some key results from \cite{RW22} and an overview of our proof.
	\subsection{The block decomposition of a simply-connected simple algebraic group}
	Let $\mathbf{G}$ be a simply-connected simple\footnote{By simple we mean that the root system of $\mathbf{G}$ is irreducible.} algebraic group over a field $\mathbf{k}$ of prime characteristic $\ell$, with a split maximal torus and a Borel subgroup $\mathbf{T}\subset \mathbf{B}$. Let also $\mathfrak{R}_+\subset\mathfrak{R}\subset\mathbf{X}$ denote the set of positive roots (with respect to the Borel subgroup $B^+$ satisfying $B^+\cap B=T$) and roots inside the character lattice of $\mathbf{T}$, and $W_0$ be the Weyl group associated with $(G,T)$. The category $\mathrm{Rep}(\mathbf{G})$ of finite dimensional algebraic $\mathbf{G}$-modules is a highest weight category, admitting the set of dominant characters $\mathbf{X}_+$ of $\mathbf{T}$ as a weight poset. For each $\lambda\in\mathbf{X}_+$, denote by $L_\lambda$ the associated simple $\mathbf{G}$-module (which is the simple socle of the induced $\mathbf{G}$-module of highest weight $\lambda$), and consider the equivalence relation $\sim$ on $\mathbf{X}_+$ generated by the relation $\mathscr{R}_1$:
	$$\lambda\mathscr{R}_1\mu\Longleftrightarrow \mathrm{Ext}_{\mathrm{Rep}(\mathbf{G})}^1(L_\lambda,L_\mu)\neq 0. $$
	For any $\lambda\in\mathbf{X}_+$, let $\overline{\lambda}\in \mathbf{X}_+/\sim$ be the associated equivalence class and define $\mathrm{Rep}_{\overline{\lambda}}(\mathbf{G})$ as the Serre subcategory generated by the family $(L_\mu,~\mu\sim\lambda)$ (this coincides with the full subcategory of $\mathrm{Rep}(\mathbf{G})$ consisting of $\mathbf{G}$-modules whose composition factors are of the form $L_\mu$, with $\mu\sim\lambda$). The so-called block decomposition (cf. \cite[Lemma 7.1, Part II]{jantzen2003representations}) of $\mathrm{Rep}(\mathbf{G})$ is then: 
	$$\mathrm{Rep}(\mathbf{G})= \bigoplus_{\overline{\lambda}\in \mathbf{X}_+/\sim}\mathrm{Rep}_{\overline{\lambda}}(\mathbf{G}).$$
	\begin{rem}
	    This formalism also makes sense when $\mathrm{char}(\mathbf{k})=0$, in which case the semi-simplicity of the category $\mathrm{Rep}(\mathbf{G})$ (cf. \cite[Theorem 22.42]{milne2017algebraic}) tells us that $\overline{\lambda}=\{\lambda\}$ for every $\lambda\in\mathbf{X}_+$, so that the block $\mathrm{Rep}_{\overline{\lambda}}(\mathbf{G})$ is just the additive subcategory generated by $L_\lambda$.
	\end{rem}  
	Fix $\lambda\in\mathbf{X}_+$ and denote by $\rho\in\mathbf{X}$ the half-sum of the positive roots (this belongs to $\mathbf{X}$ thanks to our assumption that $\mathbf{G}$ is simply connected). The linkage principle (cf. \cite[Corollary 6.17]{jantzen2003representations}) tells us that \begin{equation}\label{linkage inclusion}
	    \overline{\lambda}\subset W\bullet_\ell\lambda\cap\mathbf{X}_+,
	\end{equation} where $W:=W_{0}\ltimes \mathbb{Z}\mathfrak{R}$ is the affine Weyl group associated with $G$ and $\bullet_\ell$ is the usual $\ell$-dilated dot action on $\mathbf{X}$, i.e. we have
	$$wt_\mu\bullet_{\ell}\lambda:=w(\lambda+\ell\mu+\rho)-\rho$$ for any $\lambda\in \mathbf{X}$, $w\in W_{0}$ and $\mu\in \mathbb{Z}\mathfrak{R}$. A non-obvious result is that the inclusion (\ref{linkage inclusion}) turns out to be an equality when $\lambda$ is not contained in a special facet of $\mathbf{X}\otimes_\mathbb{Z}\mathbb{R}$ (which concerns the majority of weights), i.e. when $\lambda+\rho\in\mathbf{X}\backslash~\ell\cdot\mathbf{X}$. An easy case, which has been treated in \cite{humphreys1978blocks}, is when $\lambda$ is contained inside of an alcove, which means that $\langle\lambda+\rho,\alpha\rangle\notin\ell\mathbb{Z}$ for all $\alpha$ in the dual root system $\mathfrak{R}^\vee$ (where $\langle\cdot,\cdot\rangle$ denotes the usual perfect pairing between $\mathbf{X}$ and the cocharacter lattice $\mathbf{X}^\vee$). The case where $\langle\lambda+\rho,\alpha\rangle\in\ell\mathbb{Z}$ for some $\alpha\in \mathfrak{R}^\vee$ (but still with $\lambda+\rho\in\mathbf{X}\backslash~\ell\cdot\mathbf{X}$) is however much more involved, and was treated by Donkin in \cite{donkin1980blocks}.
	
	The inclusion (\ref{linkage inclusion}) is not an equality in general, and the exact description of the blocks is still due to Donkin (cf. \cite{donkin1980blocks}; as we will recall in Remark \ref{rem regular singular} below, the proof uses the case where $\lambda$ is not contained in a special facet): let $r(\lambda+\rho)$ be the smallest integer satisfying $\lambda+\rho\in\ell^{r(\lambda+\rho)}\cdot\mathbf{X}\backslash~\ell^{r(\lambda+\rho)+1}\cdot\mathbf{X}$ and $W^{(r(\lambda+\rho))}$ be the affine Weyl group with translation part dilated by $\ell^{r(\lambda+\rho)}$ (cf. subsetion \ref{section special} for the precise definition), then we have:
	$$\overline{\lambda}=W^{(r(\lambda+\rho))}\bullet_\ell\lambda\cap\mathbf{X}_+ .$$
	\begin{rem}
	More recently, M. De Visscher gave a shorter proof for the block decomposition \cite{article}. Several of her ideas play a key role in our proof. Her proof requires some restrictions on $\ell$ and $\mathfrak{R}$; we circumvent these restrictions in this paper.
	\end{rem}
	The main goal of this article is to give a proof of this result by working in the setting of constructible sheaves via the geometric Satake equivalence. As we will see in the end (cf. subsection \ref{A new proof of Donkin's Theorem}), one can deduce from this case the block decomposition for a general reductive group (this description was known, but not explicitly written down in \cite{donkin1980blocks}).
	\subsection{A geometric proof of the linkage principle}\label{A geometric proof of the linkage principle}
	Let $\mathbb{F}$ be an algebraically closed field of prime characteristic $p\neq\ell$, and $G$ be the Langlands dual group of $\mathbf{G}$ over $\mathbb{F}$ (in the main body of the text, we will change the notation and replace $\mathbf{G}$ with $G^\vee$). The affine grassmannian $\mathrm{Gr}$ is an ind-scheme over $\mathbb{F}$, which can be defined as the fppf-quotient of the loop group ind-scheme $LG$ (representing the functor $R\mapsto G(R((z)))$, where $z$ is an indeterminate) by the positive loop group scheme $L^+G$ (representing the functor $R\mapsto G(R[[z]])$). The geometric Satake equivalence (cf. \cite{Mirkovic2004GeometricLD}), asserts that there is an equivalence of monoidal categories
	$$(\mathrm{Perv}_{L^+G}(\mathrm{Gr},\mathbf{k}),\star)\xrightarrow{\sim} (\mathrm{Rep}(\mathbf{G}),\otimes_\mathbf{k}),$$
	where the left-hand side denotes the \textit{Satake category} (equipped with a convolution product), consisting of perverse (\'etale) sheaves on the affine Grassmannian $\mathrm{Gr}$, with coefficients in $\mathbf{k}$ and which respect an equivariance condition for the left action of the positive loop group $L^+G$ on $\mathrm{Gr}$.
	
	In \cite{RW22}, Riche and Williamson managed to give a new proof of the linkage principle by working in $\mathrm{Perv}_{L^+G}(\mathrm{Gr},\mathbf{k})$. Moreover, their methods allowed them to give a new character formula for tilting objects (valid in all characteristics), which involves certain $\ell$-Kazhdan-Lusztig polynomials. This seems to be the first instance of the geometric Satake equivalence being able to bring us some knowledge on the combinatorics of the category $\mathrm{Rep}(\mathbf{G})$ in positive characteristic. Their proof applies Treumann's ``Smith theory for sheaves" to the Iwahori-Whittaker incarnation of the Satake category, which is a highest weight category $\mathrm{Perv}_{\mathcal{IW}}(\mathrm{Gr},\mathbf{k})$ with\footnote{One needs to assume that there exists a primitive $p$-th roots of $1$ in $\mathbf{k}$  to define $\mathrm{Perv}_{\mathcal{IW}}(\mathrm{Gr},\mathbf{k})$.} weight poset $\mathbf{X}_{++}:=\rho+\mathbf{X}_+$ admitting an equivalence of highest weight categories  
	\begin{equation}\label{iwgeosat}
	    \mathrm{Perv}_{L^+G}(\mathrm{Gr},\mathbf{k})\xrightarrow{\sim} \mathrm{Perv}_{\mathcal{IW}}(\mathrm{Gr},\mathbf{k}).
	\end{equation}
	This equivalence, which comes from \cite{JEP_2019__6__707_0}, sends an indecomposable tilting object associated with $\lambda\in\mathbf{X}_+$ to the indecomposable tilting object associated with $\lambda+\rho$, which we denote by $\mathscr{T}^{\mathcal{IW}}_{\lambda+\rho}$.
 \begin{rem}
     Smith-Treumann theory originates in the work of Smith in the 1930's concerning the cohomology of topological spaces with coefficients in $\mathbb{Z}/\ell\mathbb{Z}$, and was more recently revisited by Treumann in the setting of constructible sheaves (\cite{Treu19}). The fact that the latter results can be applied to the theory of parity sheaves was first pointed out by Leslie-Lonergan in \cite{Les21}. See \cite[§1.5]{RW22} for more comments.
 \end{rem}
	
	Let $\mu_\ell$ denote the $\mathbb{F}$-group scheme of $\ell$-th root of unity, $(\mathrm{Gr})^{\mu_\ell}$ be the fixed points for the action of $\mu_\ell\subset\mathbb{G}_\mathrm{m}$ on $\mathrm{Gr}$ by rescaling the indeterminate, and put $$\mathbf{a}_\ell:=\{\mu\in\mathbf{X}\otimes_\mathbb{Z}\mathbb{R}~:~0<\langle \mu,\alpha^\vee\rangle<\ell~\forall\alpha\in\mathfrak{R}_+\}.$$ The set $\mathbf{a}_\ell$ is called the fundamental alcove, and its closure is a fundamental domain for the $\ell$-dilated ``box" action $\car_\ell$ of $W$ on $\mathbf{X}\otimes_\mathbb{Z}\mathbb{R}$, defined by
	$$wt_\mu\car_{\ell}\lambda:=w(\lambda+\ell\mu)$$ 
	for any $\lambda\in \mathbf{X}\otimes_\mathbb{Z}\mathbb{R}$, $w\in W_{0}$ and $\mu\in \mathbb{Z}\mathfrak{R}$. Notice that, for any $\lambda\in \mathbf{X}\otimes_\mathbb{Z}\mathbb{R}$ and $w\in W$, we have
	\begin{equation}\label{boxvsbullet}
	    w\bullet_\ell(\lambda-\rho)=w\car_\ell\lambda-\rho. 
	\end{equation}
	One of the main ingredients used in \cite{RW22} is the decomposition into connected components
	$$(\mathrm{Gr})^{\mu_\ell}=\bigsqcup_{\lambda\in \overline{\mathbf{a}_\ell}\cap\mathbf{X}}\mathrm{Fl}^{\ell,\circ}_{\mathbf{g}_\lambda}, $$
	where $\mathrm{Fl}^{\ell,\circ}_{\mathbf{g}_\lambda}$ denotes the identity component in the partial affine flag variety associated with the facet $\mathbf{g}_\lambda\subset \overline{\mathbf{a}_\ell}$ containing $\lambda$ (cf. \cite[Proposition 4.7]{RW22}). This partial affine flag variety is an ind-scheme defined as the fppf-quotient of the loop group ind-scheme $L_\ell G$ representing the functor $R\mapsto G(R((z^\ell)))$ by a positive loop group scheme $L_\ell^+P_{\mathbf{g}_\lambda}$, representing $R\mapsto P_{\mathbf{g}_\lambda}(R[[z^\ell]])$ and associated with the ``parahoric" group scheme $P_{\mathbf{g}_\lambda}$ arising from Bruhat-Tits theory. When $\lambda\in\mathbf{a}_\ell$, one recovers for instance the full affine flag variety for $G$ and if $\lambda=0$, then we get a copy of the affine Grassmannian.  
	
	The other main result is the construction of a fully-faithful functor 
	$$\Phi:\mathrm{Tilt}_{\mathcal{IW}}(\mathrm{Gr},\mathbf{k})\to \mathrm{Sm}_{\mathcal{IW}}((\mathrm{Gr})^{\mu_\ell},\mathbf{k}), $$
	from the subcategory of tilting objects of $\mathrm{Perv}_{\mathcal{IW}}(\mathrm{Gr},\mathbf{k})$ to the so-called \textit{Smith category} on $(\mathrm{Gr})^{\mu_\ell}$, which involves a pull-back along the immersion $(\mathrm{Gr})^{\mu_\ell}\hookrightarrow \mathrm{Gr}$ followed by passing to a certain Verdier quotient. 
	
	As a consequence of the full-faithfulness (and of the fact that both categories are Krull-Schmidt), $\Phi$ sends indecomposable objects to indecomposable objects. Thus, one can see that for every $\lambda,\mu\in\mathbf{X}_{++}$, the space $\mathrm{Hom}_{\mathrm{Tilt}_{\mathcal{IW}}(\mathrm{Gr},\mathbf{k})}(\mathscr{T}^{\mathcal{IW}}_{\lambda},\mathscr{T}^{\mathcal{IW}}_{\mu})$ is non-zero only if the supports of $\Phi(\mathscr{T}^{\mathcal{IW}}_{\lambda})$ and $\Phi(\mathscr{T}^{\mathcal{IW}}_{\mu})$ lie in the same connected component $\mathrm{Fl}^{\ell,\circ}_{\mathbf{g}_{\gamma}}$ of $(\mathrm{Gr})^{\mu_\ell}$ for some $\gamma\in \overline{\mathbf{a}_\ell}\cap\mathbf{X}$, and observe that this happens only if $\lambda,\mu\in W\car_\ell\gamma$. This means that we must have $W\car_\ell\lambda= W\car_\ell\mu$, which is equivalent to $W\bullet_\ell(\lambda-\rho)= W\bullet_\ell(\mu-\rho)$ thanks to (\ref{boxvsbullet}). In view of the equivalence (\ref{iwgeosat}) and of the geometric Satake equivalence, this means that we have
	$$\mathrm{Hom}_{\mathrm{Rep}(\mathbf{G})}(T_{\lambda-\rho},T_{\mu-\rho})\neq0\Longrightarrow W\bullet_\ell(\lambda-\rho)=W\bullet_\ell(\mu-\rho), $$
	where $T_{\lambda-\rho}$ (resp. $T_{\mu-\rho}$) denotes the indecomposable tilting $\mathbf{G}$-module associated with $\lambda-\rho$ (resp. $\mu-\rho$). Standard arguments on highest weight categories (see Theorem \ref{relations thm} below) then show that this statement is equivalent to the linkage principle (\ref{linkage inclusion}). Moreover, those same standard arguments and equivalences of categories show that, if we denote by $\overline{\lambda}\subset\mathbf{X}_{++}$ the equivalence class of $\lambda$ for the equivalence relation on $\mathbf{X}_{++}$ induced by the relation $\mathscr{R}_2$:
	$$\gamma\mathscr{R}_2\gamma'\Longleftrightarrow  \mathrm{Hom}_{\mathrm{Tilt}_{\mathcal{IW}}(\mathrm{Gr},\mathbf{k})}(\mathscr{T}^{\mathcal{IW}}_{\gamma},\mathscr{T}^{\mathcal{IW}}_{\gamma'})\neq 0,$$
	then Donkin's theorem on blocks is equivalent to
	$$\overline{\lambda}=W^{(r(\lambda))}\car_\ell\lambda\cap\mathbf{X}_{++}. $$
	\subsection{Summary of the proof} In the sequel, we will push this study further to get the full description of the blocks. Fix $\lambda,\mu\in\mathbf{X}_{++}$. A first step in this direction was actually made in an earlier unpublished version of \cite{RW22}; namely, we have an isomorphism (cf. Proposition \ref{red to non special})
		$$\mathrm{Hom}_{\mathrm{Tilt}_\mathcal{IW}(\mathrm{Gr},\mathbf{k})}(\mathscr{T}^\mathcal{IW}_{\ell\cdot\lambda},\mathscr{T}^\mathcal{IW}_{\ell\cdot\mu})\simeq \mathrm{Hom}_{\mathrm{Tilt}_\mathcal{IW}(\mathrm{Gr},\mathbf{k})}(\mathscr{T}^\mathcal{IW}_{\lambda},\mathscr{T}^\mathcal{IW}_{\mu}).$$
	This isomorphism enables us to only focus on the case where $\lambda\in\mathbf{X}\backslash~\ell\cdot\mathbf{X}$, for which we want to show that the inclusion $\overline{\lambda}\subset W\car_\ell\lambda\cap\mathbf{X}_{++}$, provided by the linkage principle, is an equality. So from now on, let us assume that $r(\lambda)=r(\mu)=0$, with $W\car_\ell\lambda=W\car_\ell\mu$; we are thus reduced to showing that $\lambda\sim\mu$. We let $\gamma$ denote the unique element of $W\car_\ell\lambda\cap \overline{\mathbf{a}_\ell}$.
	\begin{rem}\label{rem regular singular}
	    For any $r\in\mathbb{Z}_{\geq 1}$ and $M\in\mathrm{Rep}(\mathbf{G})$, let $M^{[r]}$ denote the twist of $M$ by the $r$-th power of the Frobenius endomorphism. This first step is the analogue of the second step in \cite{article}, which uses the fact that the functor $M\mapsto M^{[r]}\otimes L_{(\ell^r-1)\cdot\rho}$ induces an equivalence of categories from $\mathrm{Rep}_{\overline{\lambda}}(\mathbf{G})$ to $\mathrm{Rep}_{\overline{\lambda'}}(\mathbf{G})$, where $\lambda':=(\ell^r-1)\cdot\rho+\ell^r\cdot\lambda$.
	\end{rem}
	Let $\mathbf{g}\subset\overline{\mathbf{a}_\ell}$ be a facet, $W_\mathbf{g}\subset W$ denote the stabilizer of $\mathbf{g}$ for $\car_\ell$, and $\mathrm{Par}_{\mathcal{IW}_\ell}(\mathrm{Fl}^{\ell,\circ}_{\mathbf{g}},\mathbf{k})$ denote the additive category of Iwahori-Whittaker-equivariant parity complexes\footnote{In the body of the paper, we will mostly work in the equivalent context where $\mathbf{g}$ is a facet for the box action $\car_1$ included in the closure of $\mathbf{a}_1$ and $\mathrm{Fl}^{1,\circ}_{\mathbf{g}}$ replaces the isomorphic ind-variety $\mathrm{Fl}^{\ell,\circ}_{\ell\cdot\mathbf{g}}$.} on $\mathrm{Fl}^{\ell,\circ}_{\mathbf{g}}$ (cf. subsection \ref{Fixed points of the affine Grassmannian and connected components}). The isomorphism classes of indecomposable objects of this category are labelled (up to a shift) by the set 
	$${_\mathrm{f}W}^\mathbf{g}=\{w\in W:w~\text{is minimal in}~W_0w~\text{and maximal in}~wW_\mathbf{g}\}, $$
	and we will denote by $\mathcal{E}_{\ell,w}^\mathbf{g}$ the indecomposable parity complex associated with $w$. Moreover, the assignment $w\mapsto w\car_\ell\lambda$ induces a bijection ${_\mathrm{f}W}^{\mathbf{g}_\gamma}\xrightarrow{~}W\car_\ell\gamma\cap\mathbf{X}_{++}$. In this article, we will extensively study the equivalence relation $\sim_\mathbf{g}$ on ${_\mathrm{f}W}^\mathbf{g}$ generated by the relation $\mathscr{R}_\mathbf{g}$: 
	\begin{equation}
	    w\mathscr{R}_{\mathbf{g}}w'\Longleftrightarrow \mathrm{Hom}^\bullet_{\mathrm{Par}_{\mathcal{IW}_\ell}(\mathrm{Fl}^{\ell,\circ}_{\mathbf{g}},\mathbf{k})}(\mathcal{E}^{\mathbf{g}}_{\ell,w},\mathcal{E}^{\mathbf{g}}_{\ell,w'})\neq 0~\forall w,w'\in {_\mathrm{f}W}^{\mathbf{g}}.
	\end{equation}
	The reason for this interest is that we have an isomorphism
	\begin{equation}\label{eqintro}
	    \mathrm{Hom}_{\mathrm{Tilt}_\mathcal{IW}(\mathrm{Gr},\mathbf{k})}(\mathscr{T}^\mathcal{IW}_{u\car_\ell\gamma},\mathscr{T}^\mathcal{IW}_{u'\car_\ell\gamma})\simeq \mathrm{Hom}^\bullet_{\mathrm{Par}_{\mathcal{IW}_\ell}(\mathrm{Fl}^{\ell,\circ}_{\mathbf{g}_\gamma},\mathbf{k})}(\mathcal{E}_{\ell,u}^{\mathbf{g}_\gamma},\mathcal{E}_{\ell,u'}^{\mathbf{g}_\gamma})
	\end{equation}
	 for all $u,u'\in {_\mathrm{f}W}^{\mathbf{g}_\gamma}$, proved in \cite{RW22} and which arises once again from Smith-Treumann theory (cf. Proposition \ref{iso parity}).
	 Thus, if we let $v$ and $v'$ be the elements of ${_\mathrm{f}W}^{\mathbf{g}_\gamma}$ such that $v\car_\ell\gamma=\lambda$, $v'\car_\ell\gamma=\mu$, we have that
	 $$\lambda\sim\mu\Longleftrightarrow v\sim_{\mathbf{g}_\gamma}v'. $$
	 So we are reduced to proving that the set ${_\mathrm{f}W}^{\mathbf{g}_\gamma}$ consists of a single equivalence class for $\sim_{\mathbf{g}_\gamma}$. The following statement (which is Theorem \ref{cas general}) is the main result of this article.
	 \begin{thm}\label{thm 1.3}
	     If $\mathbf{g}\subset\overline{\mathbf{a}_\ell}$ is a non-special facet, then ${_\mathrm{f}W}^{\mathbf{g}}$ consists of a single equivalence class for $\sim_\mathbf{g}$.
	 \end{thm}
	 A facet $\mathbf{g}$ is called \textit{special} if it is of the form $\mathbf{g}=\{\ell\cdot\delta\}$ for some $\delta\in\mathbf{X}$. Since $r(\lambda)=0$, the facet $\mathbf{g}_\gamma$ is non-special, so this theorem implies the desired result. There might also be some non-special facets  $\mathbf{g}\subset\overline{\mathbf{a}_\ell}$ which do not contain an element of $\mathbf{X}$; this is for instance the case of $\mathbf{a}_\ell$ when $\ell$ is less than or equal to the Coxeter number of $\mathbf{G}$. So this theorem gives a more general result than the block decomposition. We will also describe completely the equivalence classes in ${_\mathrm{f}W}^{\mathbf{g}}$ in the case where $\mathbf{g}$ is a special facet and $\mathfrak{R}^\vee$ is not irreducible (see Theorem \ref{cas final}). 
	 
	 Fix a non-special facet $\mathbf{g}\subset\overline{\mathbf{a}_\ell}$. By standard considerations on parity complexes, computing the dimension of the $\mathrm{Hom}$-space on the right-hand side of (\ref{eqintro}) boils down to computing the dimension of the stalks of $\mathcal{E}_{\ell,w}^{\mathbf{g}}$ and $\mathcal{E}_{\ell,w'}^{\mathbf{g}}$, which are given by evaluating some anti-spherical $\ell$-Kazhdan-Lusztig polynomials at $1$ (cf. \cite[Part III]{riche2018tilting}). As the dimension of the $\mathrm{Hom}$-space can only increase when passing from $\mathrm{char}(\mathbf{k})=0$ to $\mathrm{char}(\mathbf{k})=\ell>0$, we will get the following crucial implication (which is our Corollary \ref{anti spherical implies R}):
	 $$\forall w,w'\in {_\mathrm{f}W}^{\mathbf{g}},~n_{w',w}(1)\neq0\Rightarrow w\mathscr{R}_{\mathbf{g}}w',$$
	 where $(n_{x,y},~x,y\in {_\mathrm{f}W})$ denotes the usual anti-spherical Kazhdan-Lusztig polynomials studied in \cite[Theorem 3.1]{Soergel1997KazhdanLusztigPA}. This implication allows us to view our problem as a question of combinatorics for the anti-spherical Kazhdan-Lusztig polynomials. The proof of Theorem \ref{thm 1.3} then roughly goes into two steps:
	 \begin{enumerate}
	     \item \textit{Step 1: moving away from the walls of the dominant cone}. We show that, for any $u\in {_\mathrm{f}W}^{\mathbf{g}}$, there exists some $\hat{u}\in {_\mathrm{f}W}^{\mathbf{g}}$ satisfying $$n_{u,\hat{u}}(1)\neq 0~\text{and}~\hat{u}\car_\ell\mathbf{g}\subset\ell\cdot\rho+\overline{\mathscr{C}^+_0},$$ where $\overline{\mathscr{C}^+_0}$ denotes the closure of the dominant cone (see Proposition \ref{away}). This will allow us to only focus on polynomials $n_{w',w}$ with $w,w'\in {_\mathrm{f}W}^{\mathbf{g}}$ satisfying $$w'\car_\ell\mathbf{g},w\car_\ell\mathbf{g}\subset\ell\cdot\rho+\overline{\mathscr{C}^+_0},$$
	     over which we have a much better control.
	     \item\textit{Step 2: linking close elements}. Let $v\in {_\mathrm{f}W}^{\mathbf{g}}$ be such that $ v\car_\ell\mathbf{g}\subset \ell\cdot\rho+ \overline{\mathbf{a}_\ell}$ and pick $w\in {_\mathrm{f}W}^{\mathbf{g}}$ such that $w\car_\ell\mathbf{g}\subset\ell\cdot\rho+\overline{\mathscr{C}^+_0}$. We will prove that we have $w\sim_\mathbf{g} v$ (in view of the first step, this will conclude the proof, since any element of ${_\mathrm{f}W}^{\mathbf{g}}$ will then be in relation with $v$). In order to do that, we will let $A_0$ be an alcove such that $w\car_\ell\mathbf{g}\subset\overline{A_0}$,
	     $$A_r:=\ell\cdot\rho+\mathbf{a}_\ell,A_{r-1},\cdots,A_0 $$
	     be a sequence of alcoves included in $\ell\cdot\rho+\mathscr{C}^+_0$, where $A_{i+1}$ is obtained by reflecting $A_{i}$ along one of its walls for every $i\in\llbracket0,r-1\rrbracket$, and denote by $w_i$ the element of ${_\mathrm{f}W}^{\mathbf{g}}$ satisfying $w_i\car_\ell\mathbf{g}\subset\overline{A_i}$ for all $i$ (so $w_0=w$ and $w_r=v$)\footnote{Notice that we don't necessarily have $w_i\car_\ell \mathbf{a}_i=A_i$.}. See Figure \ref{Affine hyperplanes C2} for a representation of the situation in types $C_2$ and $A_2$, where $\mathbf{g}$ is a wall. We will show that we have
	     $$w_i\sim_\mathbf{g}w_{i-1}~\forall i\in\llbracket0,r-1\rrbracket. $$
	     More specifically, we will find an element $u_i\in{_\mathrm{f}W}^{\mathbf{g}}$ such that $n_{w_i,u_i}(1)\neq 0$ and $n_{w_{i-1},u_i}(1)\neq 0$ (occasionally we might have $u_i=w_{i-1}$, but not always).
	 \end{enumerate}
	 \begin{figure}[!h]	
	 \begin{tabular}{c c}
	 (a) &
	 \begin{tikzpicture}[>=stealth]
		\foreach \x in {1,...,12}{
			\draw[ultra thin, color=gray, fill opacity=1] (\x,1)--(\x,8);
		}
	\foreach \x in {1,...,8}{
		\draw[ultra thin, color=gray, fill opacity=1] (1,\x)--(12,\x);
	}

\foreach \x in {1,...,4}{
	\draw[ultra thin, color=gray, fill opacity=1] (1,2*\x)--(2*\x,1);
	\draw[ultra thin, color=gray, fill opacity=1] (1,2*\x)--(9-2*\x,8);
	\draw[ultra thin, color=gray, fill opacity=1] (2*\x+2,1)--(11,10-2*\x);
}
\foreach \x in {1,...,3}{
	\draw[ultra thin, color=gray, fill opacity=1] (2*\x+4,1)--(12,9-2*\x);
}
\foreach \x in {1,...,1}{
	\draw[ultra thin, color=gray, fill opacity=1](\x+10,8)--(12,6+\x);
}

 \foreach \x in {1,...,2}{
	\draw[ultra thin, color=gray, fill opacity=1] (2*\x+1,8)--(2*\x+8,1);
	\draw[ultra thin, color=gray, fill opacity=1] (2*\x+5,8)--(12,2*\x+1);
	\draw[ultra thin, color=gray, fill opacity=1] (2*\x+7,8)--(2*\x,1);
}

\draw[very thin, fill=gray!30] (8,5)--(9,5)--(9,4)--(8,5);
\draw[very thin, fill=gray!30] (8,5)--(8,4)--(9,4)--(8,5);
\draw[very thin, fill=gray!30] (8,5)--(8,4)--(7,4)--(8,5);
\draw[very thin, fill=gray!30] (8,5)--(7,5)--(7,4)--(8,5);
\draw[very thin, fill=gray!30] (6,5)--(7,5)--(7,4)--(6,5);
\draw[very thin, fill=gray!30] (6,5)--(6,4)--(7,4)--(6,5);
\draw[very thin, fill=gray!30] (7,4)--(6,4)--(6,3)--(7,4);
\draw[very thin, fill=gray!30] (5,4)--(6,4)--(6,3)--(5,4);
\draw[very thin, fill=gray!30] (5,4)--(5,3)--(6,3)--(5,4);
\draw[very thin, fill=gray!30] (5,4)--(5,3)--(4,3)--(5,4);
\draw (1,2) node[circle,inner sep=2, fill=black] {} node[xshift=-8pt,yshift=-8pt,inner sep=1,fill=white] {$0$};
\draw[very thick, color=red!60, fill opacity=0] (4,3)--(5,4);
\draw (4,3) node[circle,inner sep=2, fill=black] {} node[xshift=-8pt,yshift=-8pt,inner sep=1] {$\ell\cdot\rho$};
\draw (8.75,4.75) node[scale=1,rotate=0]{$w$};
\draw (7.75,4.25) node[scale=1,rotate=0]{$w_2$};
\draw (6.75,4.75) node[scale=1,rotate=0]{$w_4$};
\draw (6.75,3.25) node[scale=1,rotate=0]{$w_6$};
\draw (5.75,3.75) node[scale=1,rotate=0]{$w_7$};
\draw (4.75,3.25) node[scale=1,rotate=0]{$v$};
\draw[black, dashed] (4,3)--(12,3);
\draw[black, dashed] (5,4)--(9,8);
\draw[very thick] (1,2)--(7,8);
\draw[very thick] (1,2)--(12,2);

\draw[very thick, color=red!60] (6,3)--(5,4);
\draw[very thick, color=red!60] (6,3)--(7,4);
\draw[very thick, color=red!60] (6,5)--(7,4);
\draw[very thick, color=red!60] (8,5)--(7,4);
\draw[very thick, color=red!60] (8,5)--(9,4);
	
\end{tikzpicture}
\\
\\
(b) &
\begin{tikzpicture}[scale=2]
\begin{rootSystem}{A}
\draw[black,very thick] \weight{-2}{5} -- \weight{-2}{-2} --
\weight{5}{-2};
\draw[black,dashed] \weight{-1}{5} -- \weight{-1}{-1} --
\weight{5}{-1};
\draw[very thin, fill=gray!30] \weight{0}{3}--\weight{1}{2}--\weight{0}{2}--\weight{0}{3};
\draw[very thin, fill=gray!30] \weight{1}{2}--\weight{1}{1}--\weight{0}{2}--\weight{1}{2};
\draw[very thin, fill=gray!30] \weight{0}{1}--\weight{1}{1}--\weight{0}{2}--\weight{0}{1};
\draw[very thin, fill=gray!30] \weight{0}{1}--\weight{1}{0}--\weight{1}{1}--\weight{0}{1};
\draw[very thin, fill=gray!30] \weight{0}{1}--\weight{0}{0}--\weight{1}{0}--\weight{0}{1};
\draw[very thin, fill=gray!30] \weight{0}{0}--\weight{1}{0}--\weight{1}{-1}--\weight{0}{0};
\draw[very thin, fill=gray!30] \weight{0}{0}--\weight{1}{-1}--\weight{0}{-1}--\weight{0}{0};
\draw[very thin, fill=gray!30] \weight{0}{0}--\weight{-1}{0}--\weight{0}{-1}--\weight{0}{0};
\draw[very thin, fill=gray!30] \weight{-1}{-1}--\weight{-1}{0}--\weight{0}{-1}--\weight{-1}{-1};

\draw[very thick, color=red!60] \weight{0}{2}--\weight{1}{2};
\draw[very thick, color=red!60] \weight{0}{2}--\weight{0}{1};
\draw[very thick, color=red!60] \weight{0}{1}--\weight{1}{0};
\draw[very thick, color=red!60] \weight{1}{0}--\weight{1}{-1};
\draw[very thick, color=red!60] \weight{0}{-1}--\weight{1}{-1};
\draw[very thick, color=red!60] \weight{0}{-1}--\weight{-1}{0};

\draw \weight{0.35}{2.35} node[scale=.8,rotate=0]{$w$};
\draw \weight{0.35}{1.35} node[scale=.8,rotate=0]{$w_2$};
\draw \weight{0.65}{.65} node[scale=.8,rotate=0]{$w_3$};
\draw \weight{1.35}{-.65} node[scale=.8,rotate=0]{$w_5$};
\draw \weight{0.35}{-.65} node[scale=.8,rotate=0]{$w_6$};
\draw \weight{-.35}{-.35} node[scale=.8,rotate=0]{$v$};

\draw \weight{-2}{-2} node[circle,inner sep=1.5, fill=black, fill opacity=1] {} node[xshift=-6pt,yshift=-6pt,inner sep=1] { $0$};
\draw \weight{-1}{-1} node[circle,inner sep=1.5, fill=black] {} node[xshift=-6pt,yshift=-6pt,inner sep=1] { $\ell\cdot\rho$};
\end{rootSystem}
\end{tikzpicture}

\end{tabular}
	\caption{\textit{Hyperplanes arrangement when $\mathfrak{R}^\vee$  is of type $C_2$ for (a) and $A_2$ for (b)}. The walls of the dominant cone $\mathscr{C}_0^+$ (resp. of $\ell\cdot\rho+\mathscr{C}_0^+$) are represented by thick black lines (resp. dashed lines), each alcove $A_i$ for ($i\in\llbracket 0,r\rrbracket)$ is gray, each alcove $w_i\car_\ell\mathbf{a}_\ell$ is labelled with $w_i$, and the facets $\{w_i\car_\ell\mathbf{g},~i\in\llbracket 0,r\rrbracket\}$ are represented by red lines. Notice that on both of these examples, it sometimes happens that there are two alcoves $A_i$ and $A_{i+1}$ containing a same red faced in their closure for some $i$, so that we have $w_i=w_{i+1}$.}  \label{Affine hyperplanes C2}
\end{figure}
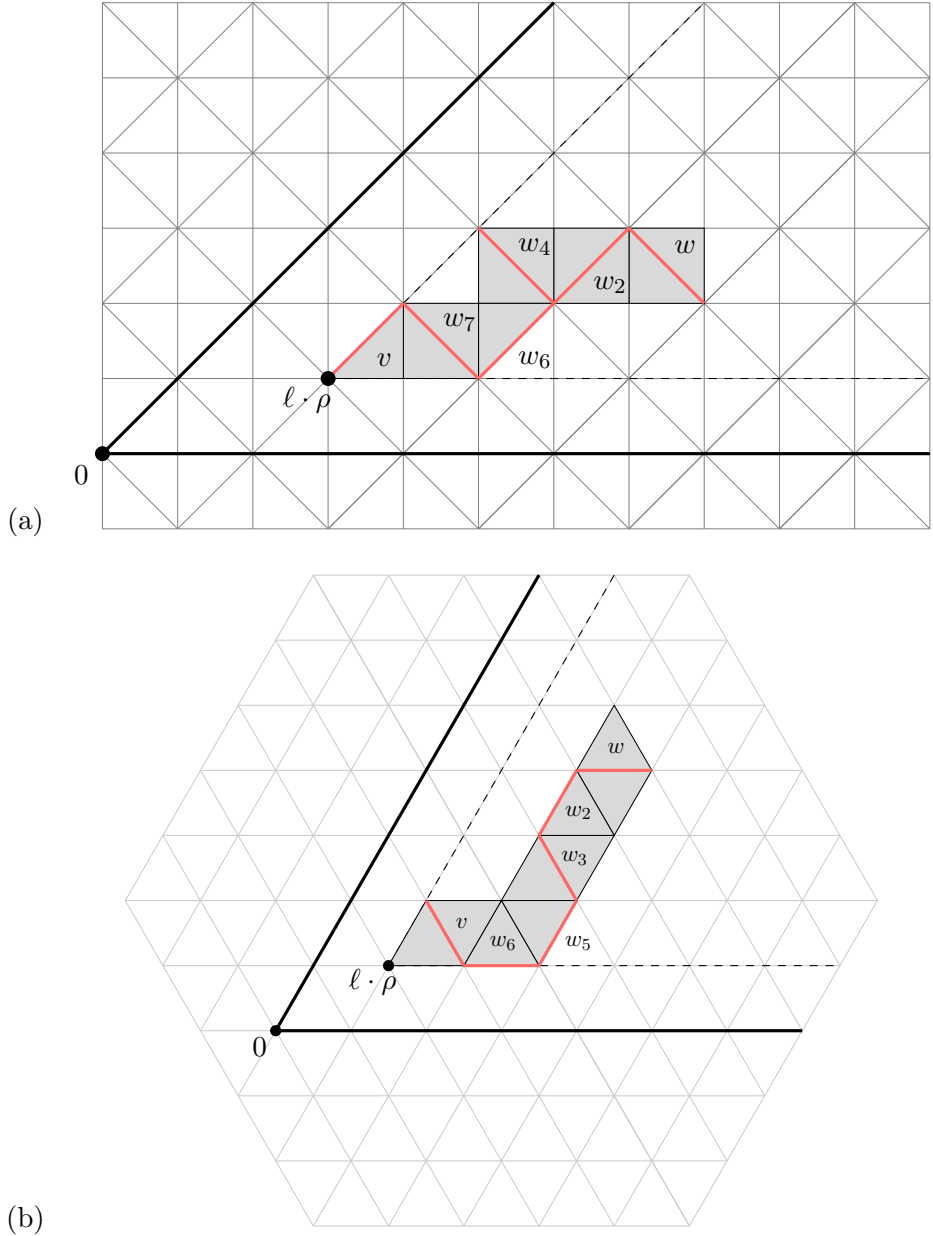

\begin{rems}\begin{enumerate}
    \item The need for \textit{Step 1} is due to the ``cancellation effect'' occurring for the anti-spherical Kazhdan-Lusztig polynomials when we are close to the walls of the dominant chamber, as it was observed in the similar context of \cite{andersen1986inversion} (cf. the introduction and section 9 of this reference).
    \item The second step will actually be split into two parts: the case where $\mathbf{g}$ is not a point, and the case where $\mathbf{g}$ is a point (and still a non-special facet). The latter uses the theory of ``periodic polynomials'' (originally due to Lusztig) as displayed in \cite{Soergel1997KazhdanLusztigPA}, and is the most involved part of the paper (cf. subsection \ref{Facets which are points}).
    \item These steps were inspired to us by the steps used in \cite[§2]{article}.
    \item We will actually use ``geometric" arguments (namely, parity complexes on partial affine flag varieties) to prove some ``combinatorial" facts concerning the anti-spherical $\ell$-Kazhdan-Lusztig polynomials (see Proposition \ref{prop max coset}), which hold in particular for the ordinary anti-spherical Kazhdan-Lusztig polynomials. However, we do not yet have an analogous geometric incarnation for the combinatorics we use concerning the periodic polynomials (cf. Proposition \ref{period and antisph}), but hope to come back to this question in the future.
\end{enumerate}
 \end{rems}	 

\subsection{Consequences}
Our new proof of Donkin's Theorem enables us to give a bound on the length of a minimum chain linking two weights in the same block for $\mathrm{Rep}(\mathbf{G})$, in the same fashion as \cite[Corollary 3.1]{article}, but without any restriction on the root system nor on the characteristic of $\mathbf{k}$, see Proposition \ref{bound group}. Such a result was not available via Donkin's original proof, as his arguments required going arbitrarily far inside of the dominant cone. 

Moreover, our determination of the equivalence classes in ${_\mathrm{f}W}^\mathbf{g}$ for $\sim_\mathbf{g}$ also applies to the case where $\mathrm{char}(\mathbf{k})=0$, from which we are able to deduce the block decomposition for a quantum group at an $\ell$-th root of unity in subsection \ref{Block decomposition of a quantum group} (this result was originally found in \cite{thams1994blocks}).
	 \subsection{Structure of the paper}
	 After some preliminaries on highest weight categories in section \ref{relation section}, we recall in section \ref{Recollections on (Iwahori-Whittaker) equivariant derived categories} the construction of Iwahori-Whittaker equivariant derived categories on partial affine flag varieties of the form $\mathrm{Fl}^{\circ}_{\mathbf{g}}$, for a facet $\mathbf{g}\subset\overline{\mathbf{a}_1}$. In particular, we introduce indecomposable parity complexes, and study the effect of pushforward and pullback of these objects under the canonical proper morphism $\pi_{\mathbf{g}}:\mathrm{Fl}^{\circ}_{\mathbf{a}_1}\to \mathrm{Fl}^{\circ}_{\mathbf{g}}$ (subsection \ref{Pushforward and pullback of indecomposable parity complexes}). This allows us to introduce and study the equivalence relation $\sim_\mathbf{g}$ on the set ${_\mathrm{f}W}^\mathbf{g}$. Section \ref{Determination of the blocks} is the heart of the paper and is dedicated to the study of the equivalence classes for $\sim_\mathbf{g}$. In particular, we apply the plan described above: \textit{Step 1} is dealt with at the end of subsection \ref{geometry}, while \textit{Step 2} is taken care of in  subsection \ref{Facets which are not points} (when $\mathbf{g}$ is not a point) and  subsection \ref{Facets which are points} (when $\mathbf{g}$ is a non-special point). The case where $\mathbf{g}$ is a special facet is done at the end of section \ref{smith section}, as we need Smith-Treumann theory to deduce it from the previous cases. Finally, we harvest the consequences for representation theory in the last section: block decomposition for reductive groups in subsection \ref{A new proof of Donkin's Theorem} and for quantum groups in subsection \ref{Block decomposition of a quantum group}.
	 \subsection{Acknowledgements}
	 I am deeply grateful to Simon Riche for his guidance through many suggestions and countless careful re-readings, and to Geordie Williamson for helpful discussions and comments. I also want to thank Joel Gibson for putting together the interactive visualization tool \textit{Lievis} and making it accessible on his web page (cf. \cite{Lievis}): a lot (if not all) of the intuition behind the results of this paper concerning the anti-spherical Kazhdan-Lusztig polynomials came from there. This work was conducted during PhD studies at the Universit\'e Clermont Auvergne, co-supervised by Simon Riche and Geordie Williamson.

     This project has received funding from the European Research Council (ERC) under the European Union’s Horizon 2020 research and innovation program (grant agreement No. 101002592).
	\section{Blocks for a highest weight category}\label{relation section}
	In this section, $\mathbf{k}$ is any field. The goal of these formal preliminaries is to introduce two equivalence relations on the weight poset of a highest weight category. Both relations allow us to give block decompositions of the category, and we will show that these relations and decompositions are the same. These results will later be applied in sections \ref{smith section} and \ref{Applications to representation theory} to categories of perverse sheaves and representations. We start by briefly recalling some properties of a highest weight category.
	\subsection{Recollections on highest weight categories}
	Let $\mathcal{A}$ be a highest weight category over $\mathbf{k}$, with weight poset $(\Lambda,\leq)$ (cf. \cite[3.7]{riche} for a detailed treatment of highest weight categories). In particular, we assume that each object of $\mathcal{A}$ has finite length and that the vector space $\mathrm{Hom}_\mathcal{A}(M,N)$ is finite dimensional for every $M,N\in \mathcal{A}$. To each $s\in\Lambda$ corresponds a simple object $L_s$, a \textit{standard object} $\Delta_s$ and a \textit{costandard object} $\nabla_s$. By definition, the association $s\mapsto L_s$ induces a bijection between $\Lambda$ and the isomorphism classes of simple objects of $\mathcal{A}$. For any $M\in\mathcal{A}$ and $s\in\Lambda$, we will denote by $[M:L_s]$ the number of times $L_s$ appears in a composition series of $M$.
	
	We denote by $\mathrm{Tilt}(\mathcal{A})$ the full additive subcategory of $\mathcal{A}$ consisting of tilting objects, i.e. those admitting both a filtration with sub-quotients being standard objects and a filtration with sub-quotients being costandard objects. For any $s\in\Lambda$ and $M\in \mathrm{Tilt}(\mathcal{A})$, we will denote by $(M:\nabla_s)$, resp. $(M:\Delta_s)$, the number of times $\nabla_s$ appears in a filtration of $M$ with costandard, resp. standard, sub-quotients. This number does not depend on the choice of such a filtration, as one can show that it is equal to $\mathrm{dim}_\mathbf{k}\mathrm{Hom}_\mathcal{A}(\Delta_s,M)$, resp. $\mathrm{dim}_\mathbf{k}\mathrm{Hom}_\mathcal{A}(M,\nabla_s)$. Indecomposable tilting objects are also parameterized (up to isomorphism) by $\Lambda$; we denote by $T_s$ the unique indecomposable tilting object such that 
	$$ [T_s:L_s]=1\quad\text{and}\quad \forall t\in\Lambda,~ [T_s:L_t]\neq0 \Rightarrow t\leq s.$$
	We have morphisms $L_s\twoheadleftarrow\Delta_s\hookrightarrow  T_s$ for all $s\in\Lambda$. Finally recall that the canonical functor 
	\begin{equation}\label{equivalence0}
		\mathrm{K}^b\mathrm{Tilt}(\mathcal{A})\to D^b(\mathcal{A})
	\end{equation}
	is an equivalence of categories (cf. \cite[Proposition 7.17]{riche}).
	
	\subsection{Equivalence relations on $\Lambda$}\label{Equivalence relations on Lambda}
	We are going to consider two equivalence relations on the set $\Lambda$. The first one, denoted by $\sim_1$, is generated by the relation $\mathscr{R}_1$, defined by 
	$$s\mathscr{R}_1 t\Leftrightarrow \mathrm{Ext}_\mathcal{A}^1(L_s,L_t)\neq 0.$$
	Here, $\mathrm{Ext}_\mathcal{A}^1(B,A)$ denotes the Ext-group of isomorphism classes of extensions of $B$ by $A$, for two objects $A,B$. Recall (cf. \cite[Lemma 13.27.6]{stacks-project}) that this coincides with the $\mathbf{k}$-vector space $\mathrm{Hom}_{D^b(\mathcal{A})}(B,A[1])$. For all $s\in\Lambda$, we denote by  $\overline{s}$ the associated equivalence class and by $\mathcal{A}_{\overline{s}}$ the Serre sub-category generated by the $L_t$'s, for $t\in\overline{s}$.
	\begin{prop}
		The canonical functor $$\bigoplus_{\overline{s}\in\Lambda/\sim_1}\mathcal{A}_{\overline{s}}\to \mathcal{A}$$ is an equivalence of categories.
	\end{prop}
	\begin{proof}
		The full faithfulness comes from the fact that, for each $(M,N)\in\mathcal{A}_{\overline{s}}\times\mathcal{A}_{\overline{t}}$ with $\overline{s}\neq\overline{t}$, we have $\mathrm{Hom}_\mathcal{A}(M,N)=0$. This is obvious when $M$ and $N$ are simple, and can be shown by induction on the length of these objects in the general case. With this fact and another induction on the length, one can also show using long exact sequences of cohomology that 
		\begin{equation}\label{induction lenghth}
		    \mathrm{Ext}^1_\mathcal{A}(M,N)=0.
		\end{equation}
		
		It remains to show that the functor is essentially surjective. We still proceed by induction on the length of $M$, the case of simple objects being obvious. Let $N\subset M$ and $t\in\Lambda$ be such that we have a short exact sequence
		\begin{equation}\label{exact sequence}
		    0 \rightarrow N\xrightarrow{\varphi} M\rightarrow L_t\rightarrow 0.
		\end{equation}
		By the induction hypothesis, we have a canonical isomorphism $$N\simeq \bigoplus_{\overline{s}\in\Lambda/\sim_1}N_{\overline{s}},$$ with $N_{\overline{s}}\in \mathcal{A}_{\overline{s}}$ for all $\overline{s}$.
		Now, let $p:N\to N_{\overline{t}}$ be the canonical projection and $M'\in\mathcal{A}$ be such that we have the following commutative diagram, with exact rows:
		$$\xymatrix{
    0\ar[r] & N\ar[r]^{\varphi} \ar[d]^{p} &  M \ar[r] \ar[d]^{\psi} & L_t \ar[r]\ar[d]^{\mathrm{id}} & 0 \\
     0\ar[r] & N_{\overline{t}}\ar[r] &  M' \ar[r] & L_t \ar[r] & 0.
  }$$
  The morphism $\psi$ is clearly surjective, and restricting $\varphi$ to $M'':=\oplus_{\overline{s}\neq\overline{t}}N_{\overline{s}}$ yields an exact sequence
  $$0 \rightarrow M''\rightarrow M\rightarrow M'\rightarrow 0. $$
  Finally, observe that $M'$, resp. $M''$, belongs to  $\mathcal{A}_{\overline{t}}$, resp. $\oplus_{\overline{s}\neq\overline{t}}\mathcal{A}_{\overline{s}}$, so that $\mathrm{Ext}^1_\mathcal{A}(M',M'')=0$ thanks to (\ref{induction lenghth}). Thus, we have $M\simeq M'\oplus M''$, and this concludes the proof.
	\end{proof}
	The existence of a quasi-inverse functor allows us to define projection functors $Q_{\overline{s}}:\mathcal{A}\to \mathcal{A}_{\overline{s}}$ for every $\overline{s}$. We thus have an isomorphism of functors
	\begin{equation}\label{eq1}
		\mathrm{id}_\mathcal{A}\simeq \bigoplus_{\overline{s}\in\Lambda/\sim_1} Q_{\overline{s}}
	\end{equation}
	where there exists no non-zero morphism between the essential images of any two distinct functors $Q_{\overline{s}}$ and $Q_{\overline{t}}$. In particular, each $Q_{\overline{s}}$ is exact.
	
	The second equivalence relation is denoted $\sim_2$. It is generated by the relation $\mathscr{R}_2$:
	$$s\mathscr{R}_2 t\Leftrightarrow \mathrm{Hom}_\mathcal{A}(T_s,T_t)\neq 0.$$
	For every weight $s$, we denote by $(s)$ the associated equivalence class and by $\mathrm{Tilt}_{(s)}(\mathcal{A})$ the additive sub-category generated by the family $(T_t)_{t\in(s)}$. Since any tilting object is a sum of indecomposable tilting objects, it is easy to see that the canonical functor  $\bigoplus_{(s)\in\Lambda/\sim_2}\mathrm{Tilt}_{(s)}(\mathcal{A})\to \mathrm{Tilt}(\mathcal{A})$ is an equivalence, so that we can define projection functors $\pi_{(s)}:\mathrm{Tilt}(\mathcal{A})\to \mathrm{Tilt}_{(s)}(\mathcal{A})$, inducing an isomorphism of functors 
	$$\mathrm{id}_{\mathrm{Tilt}(\mathcal{A})}\simeq\bigoplus_{(s)\in\Lambda/\sim_2}\pi_{(s)}.$$
	
	The $\pi_{(s)}$'s induce endofunctors of the category $\mathrm{K}^b\mathrm{Tilt}(\mathcal{A})$ (that we still denote by $\pi_{(s)}$), and those still induce a decomposition of the identity functor. Conjugating the $\pi_{(s)}$'s with the equivalence (\ref{equivalence0}), the obtained functors  (which we still denote by $\pi_{(s)}$) provide a decomposition 
	\begin{equation}\label{eq2}
		\mathrm{id}_{D^b(\mathcal{A})}\simeq\bigoplus_{(s)\in\Lambda/\sim_2}\pi_{(s)}.
	\end{equation}	
	This isomorphism implies that each functor $\pi_{(s)}$ preserves the sub-category $\mathcal{A}$, so we get an isomorphism $\mathrm{id}_{\mathcal{A}}\simeq\bigoplus_{(s)}\pi_{(s)}$. It is also clear that each $\pi_{(s)}$ is an exact functor of the category $\mathcal{A}$ and that there is no non-zero morphism between the essential images of two distinct such functors.
	\begin{lem}\label{lem1}
		Let $s\in\Lambda$. For all $t\in\Lambda$, we have isomorphisms
		$$\pi_{(t)}(L_s) \simeq\left\{
		\begin{array}{ll}
			L_s~\text{if}~s\in(t)\\
			0~\text{otherwise}
		\end{array}
		\right.
		\qquad Q_{\overline{t}}(T_s) \simeq\left\{
		\begin{array}{ll}
			T_s~\text{if}~s\in\overline{t}\\
			0~\text{otherwise}.
		\end{array}
		\right.
		$$
	\end{lem}
	\begin{proof}
		By the decompositions (\ref{eq1}) and (\ref{eq2}) and the fact that the objects $L_s$ and $T_s$ are indecomposable, we know that there exists a unique $(t)$, resp. a unique $\overline{r}$, such that $\pi_{(t)}(L_s)\neq0$, resp. $Q_{\overline{r}}(T_s)\neq0$, and we then have $\pi_{(t)}(L_s)\simeq L_s$, resp. $Q_{\overline{r}}(T_s)\simeq T_s$. By exactness of the projection functors, we have morphisms
		$$ Q_{\overline{u}}(L_s) \twoheadleftarrow Q_{\overline{u}}(\Delta_s)\hookrightarrow Q_{\overline{u}}(T_s)~\text{and}~ \pi_{(u)}(L_s)\twoheadleftarrow \pi_{(u)}(\Delta_s) \hookrightarrow\pi_{(u)}(T_s)$$
		for every $u\in\Lambda$. If $Q_{\overline{u}}(T_s)=0$, then $Q_{\overline{u}}(L_s)=0$, so $\overline{u}\neq \overline{s}$. This implies that $\overline{r}= \overline{s}$. Likewise, if $\pi_{(u)}(L_s)\neq0$, then $\pi_{(u)}(T_s)\neq0$, so $(u)=(s)$, whence $(t)= (s)$.
	\end{proof}
	
	We have now all the necessary ingredients to prove the following theorem.
	\begin{thm}\label{relations thm}
		The relations  $\sim_1$ and $\sim_2$ coincide.
	\end{thm}
	\begin{proof}
		Let $s,t\in\Lambda$ be such that $(s)\neq(t)$. By lemma \ref{lem1}, we have 
		$$\mathrm{Ext}_\mathcal{A}^1(L_{t},L_{s})=\mathrm{Hom}_{D^b(\mathcal{A})}(L_{t},L_{s}[1])\simeq \mathrm{Hom}_{D^b(\mathcal{A})}(\pi_{(t)}(L_{t}),\pi_{(s)}(L_{s}[1])).$$
		As there is no non-zero morphism between the essential images of $\pi_{(s)}$ and $\pi_{(t)}$, the right-hand side is zero, so $t\cancel{\mathscr{R}_1} s$. We have thus shown the implication $\overline{s}=\overline{t}\Rightarrow (s)=(t)$.
		
		Conversely, if $s,t\in\Lambda$ are such that $\overline{s}\neq\overline{t}$, then it follows that
		$$\mathrm{Hom}_\mathcal{A}(T_{t},T_{s})\simeq \mathrm{Hom}_\mathcal{A}(Q_{\overline{t}}(T_{t}),Q_{\overline{s}}(T_{s}))=0$$
		so $t\cancel{\mathscr{R}_2} s$, and we have the implication $(s)=(t)\Rightarrow  \overline{s}=\overline{t}$.
	\end{proof}
	In the sequel, we will denote by $\sim$ the equivalence relation on $\Lambda$ considered in the previous theorem.
	\section{Recollections on (Iwahori-Whittaker) equivariant derived categories}\label{Recollections on (Iwahori-Whittaker) equivariant derived categories}
	In this section, we recall the construction of Iwahori-Whittaker-equivariant derived categories on some partial affine flag varieties arising from Bruhat-Tits theory. We fix a prime number $\ell\neq p$. Since our sheaves will be \'etale, we let $\mathbf{k}$ be either a finite field of prime characteristic $\ell$ or a finite extension of $\mathbb{Q}_\ell$.
	\subsection{Notations}\label{Notations} From now on, $G$ will denote a semi-simple algebraic group of adjoint type defined over an algebraically closed field $\mathbb{F}$ of characteristic $p>0$. Choose a Borel subgroup $B\subset G$ and a maximal torus $T\subset B$, and let $\mathfrak{R}\subset \mathbb{X}$ (resp. $\mathfrak{R}_+$) denote the subset of roots (resp. positive roots with respect to the opposite Borel subgroup of $B$ with respect to $T$) inside the group of characters of $T$. Each root $\alpha$ defines a subgroup $U_\alpha\subset G$ isomorphic to the additive group $\mathbb{G}_\mathrm{a}$, and the subgroup generated by the $U_\alpha$'s, for $\alpha\in\mathfrak{R}_+$ (resp. for $\alpha\in -\mathfrak{R}_+$), is a unipotent group which we will denote by $U^+$ (resp. $U$).  We will also denote by $\mathfrak{R}^\vee\subset\mathbb{X}^\vee$ the set of coroots inside the group of cocharacters of $T$ and by $\mathbb{X}^\vee_+$ (resp. $\mathbb{X}^\vee_{++}$) the set of dominant (resp. strictly dominant) cocharacters, i.e. those cocharacters which satisfy $\langle\lambda,\alpha\rangle\geq0$ (resp. $\langle\lambda,\alpha\rangle>0$) for any $\alpha\in \mathfrak{R}_+$. 
	
	Fix an integer $n\geq 1$. We will write $\mathcal{O}_n:=\mathbb{F}[[z^n]]$, $\mathcal{K}_n:=\mathbb{F}((z^n))$, $\mathcal{O}:=\mathcal{O}_1$, $\mathcal{K}:=\mathcal{K}_1$ where $z$ is an indeterminate. For any affine $\mathbb{F}$-group scheme $H$, we define the functors $L_n^+H:=R \mapsto H(R[[z^n]])$ and $L_nH:R \mapsto H(R((z^n)))$ from $\mathbb{F}$-algebras to groups, which are representable by a $\mathbb{F}$-group scheme and a $\mathbb{F}$-group ind-scheme respectively. Note that these definitions also make sense if $H$ is only defined over $\mathcal{O}_n$. We will write $L^+H$ , resp. $LH$, instead of $L_1^+H$, resp. $L_1H$. We will denote by $\mathrm{Iw}_u^+$ the inverse image of $U^+$ under the evaluation map $L^+G\to G,~z\mapsto 0$.
	
	We assume that there exists a primitive $p$-th root of unity $\zeta\in\mathbf{k}$, and consider the Artin-Schreier map $\mathrm{AS}:\mathbb{G}_\mathrm{a}\to \mathbb{G}_\mathrm{a}$ determined by the map of rings $x\mapsto x^p-x$. This morphism is a Galois cover of group $\mathbb{Z}/p\mathbb{Z}$, so determines a continuous group morphism $\pi_1(\mathbb{G}_\mathrm{a},0)\to \mathbb{Z}/p\mathbb{Z}$, where $\pi_1(\mathbb{G}_\mathrm{a},0)$ is the \'etale fundamental group of $\mathbb{G}_\mathrm{a}$ with geometric base point $0$. The composition of this map with the morphism $\mathbb{Z}/p\mathbb{Z}\to \mathbf{k}^\times$ (induced by $\zeta$) yields a continuous representation of the fundamental group, and thus a local system on $\mathbb{G}_\mathrm{a}$ of rank one. We denote this local system by $\mathcal{L}_{\mathrm{AS}}$.
	
	\subsection{The affine Weyl group and some Bruhat-Tits theory}\label{The affine Weyl group and some Bruhat-Tits theory}
	We let $N_G(T)$ be the normalizer of $T$ in $G$. The finite Weyl group associated with $(G,T)$ will be denoted by $W_0:=N_G(T)/T$, and we consider the affine Weyl group
	$$W:=W_{0}\ltimes \mathbb{Z}\mathfrak{R}^\vee,$$
	which acts naturally on $E:=\mathbb{X}^\vee\otimes_\mathbb{Z} \mathbb{R}$ via the $n$-dilated ``box''  action, defined by $$wt_\mu\car_{n}\lambda:=w(\lambda+n\mu)$$ for any $\lambda\in \mathbb{X}^\vee\otimes_\mathbb{Z} \mathbb{R}$, $w\in W_{0}$ and $\mu\in \mathbb{Z}\mathfrak{R}^\vee$, where we have denoted by $wt_\mu$ the element of the affine Weyl group associated with the couple $(w,\mu)$. The closure $\overline{\mathbf{a}_n}$ of the set 
	$$\mathbf{a}_n:=\{\lambda\in E ~:~0<\langle\lambda,\alpha\rangle<n~\forall \alpha\in\mathfrak{R}_+\}$$ 
	is a fundamental domain for this action, which stabilizes $ \mathbb{X}^\vee$. Thus, $\overline{\mathbf{a}_n}\cap \mathbb{X}^\vee$ is a fundamental domain for the action of $W$ on the set of cocharacters. Each root $\alpha\in \mathfrak{R}$ defines a reflection $s_\alpha\in W$, and we will denote by $S_0$ the set of simple reflections (i.e. associated with a simple root) of the finite Weyl group, which is known to generate $W_0$. It is also well known (cf. \cite[II.6.3]{jantzen2003representations}) that the set of simple reflections 
	$$S:=\{(s,0),s\in S_0\}\cup\{(s_\beta,\beta^\vee)\},$$ where $\beta^\vee$ runs through the set of largest short roots of the irreducible components of $\mathfrak{R}^\vee$, generates $W$. Moreover, $W$ is a Coxeter group, with Coxeter generating system $S$, for which we will denote by $l:W\to \mathbb{Z}_{\geq0}$ the associated length function.
	
	It will also be useful for us to note that one can extend the translation action of $\mathbb{Z}\mathfrak{R}^\vee$ to the whole group of cocharacters $\mathbb{X}^\vee$, by putting $t_\mu\car_n\lambda:=\lambda+n\mu$ for every $\mu\in E,\lambda\in\mathbb{X}^\vee$, and that this extends the action of $W$ on $E$ to an action of the extended affine Weyl group $\tilde{W}:=W_0\ltimes\mathbb{X}^\vee$. The subgroup $W\subset \tilde{W}$ is normal and one can extend the length function $l$ to the whole group $\tilde{W}$ (cf. \cite[§2.1]{achar2021geometric} for more details), so that we can define the subgroup
	$$\Omega:=\{w\in \tilde{W}~|~l(w)=0\}, $$
	which acts on $W$ by Coxeter group isomorphisms and induces an isomorphism
	\begin{equation}\label{minuscule}
	    \Omega\ltimes W\simeq \tilde{W}
	\end{equation}
	such that $l(\omega w)=l(w\omega)=l(w)$ for all $(\omega,w)\in \Omega\times W$.
	
	The action of $\tilde{W}$ on $E$ defines a hyperplane arrangement in $E$, and hence a collection of \textit{facets} (cf. \cite[Ch.5, §1.2]{cdi_springer_books_10_1007_978_3_540_34491_9}). To any facet $\mathbf{g}\subset\overline{\mathbf{a}_n}$, Bruhat-Tits theory associates a ``parahoric group scheme'' $P_\mathbf{g}$ defined over $\mathcal{O}_n$, whose generic fiber is isomorphic to $G\times_{\mathrm{Spec}(\mathbb{F})} \mathrm{Spec}(\mathcal{K}_n)$ and whose group of $\mathcal{O}_n$-points coincides with a subgroup of finite index of the pointwise stabilizer of $-\mathbf{g}$\footnote{We follow the conventions of \cite[§4]{RW22} for the facet which $P_\mathbf{g}(\mathcal{O}_n)$ must stabilize. Note that the authors there defined $\mathbf{a}_n$ to be the opposite of our current fundamental alcove.} for the action of $G(\mathcal{K}_n)$ on the Bruhat-Tits building associated with $G\times_{\mathrm{Spec}(\mathbb{F})} \mathrm{Spec}(\mathcal{K}_n)$. We fix such a facet $\mathbf{g}$. The partial affine flag variety associated with $\mathbf{g}$ will be denoted by $\mathrm{Fl}^{n}_{\mathbf{g}}$ and defined as the fppf-quotient $L_nG/L^+_nP_\mathbf{g}$, which is an ind-projective ind-scheme over $\mathbb{F}$ (cf. \cite{PAPPAS2008118} for a detailed exposition on these partial affine flag varieties). The connected components of $\mathrm{Fl}^{n}_{\mathbf{g}}$ are in bijection with the group $\mathbb{X}^\vee/\mathbb{Z}\mathfrak{R}^\vee$ (see  \cite[Theorem 0.1]{PAPPAS2008118}), so we will denote by $\mathrm{Fl}^{n,\circ}_{\mathbf{g}}$ the connected component associated with the neutral element. In the sequel, we will apply these results to the cases where $n=\ell$ or $n=1$; when $n=1$ (which will be the case until section \ref{smith section}), we will write $\mathrm{Fl}^{\circ}_{\mathbf{g}}$ (resp. $\mathrm{Fl}_{\mathbf{g}}$) instead of $\mathrm{Fl}^{1,\circ}_{\mathbf{g}}$ (resp. $\mathrm{Fl}^1_{\mathbf{g}}$).
	\subsection{ Parity sheaves on partial affine flag varieties}\label{section partial affine flag}
	Let $\chi_0:U^+\to\mathbb{G}_a$ be a morphism of $\mathbb{F}$-algebraic groups which restricts to a non-zero morphism on each subgroup $U_\alpha$, for $\alpha$ a simple root, and let $\chi:\mathrm{Iw}_u^+\to \mathbb{G}_a$ denote the composition of $\chi_0$ with the evaluation map. For any $Y\subset \mathrm{Fl}^{\circ}_{\mathbf{g}}$ which is a locally closed finite union of $\mathrm{Iw}_u^+$-orbits, one can show (cf. \cite[Lemma 3.2]{ciappara2021hecke}) that the $\mathrm{Iw}_u^+$-action on $Y$ factors through a quotient group of finite type $J$ such that $\chi$ factors through $\chi_J:J\to \mathbb{G}_\mathrm{a}$, where $\chi_J$ is induced by $\chi$; we can then consider the $(J,\chi_J^*\mathcal{L}_{\mathrm{AS}})$-equivariant derived category of \'etale $\mathbf{k}$-sheaves $D^b_{J,\chi_J^*\mathcal{L}_{\mathrm{AS}}}(Y,\mathbf{k})$. This category is by definition the subcategory of $D_{\mathrm{c}}^b(Y,\mathbf{k})$ consisting of constructible complexes of \'etale $\mathbf{k}$-sheaves $\mathcal{F}$ such that there exists an isomorphism $a^*\mathcal{F}\simeq \chi_J^*\mathcal{L}_{\mathrm{AS}}\boxtimes\mathcal{F}$, where $a:J\times Y\to Y$ is the action map; this definition does not depend on the choice of $J$. We then define the category $D^b_{\mathcal{IW}}(\mathrm{Fl}^{\circ}_{\mathbf{g}},\mathbf{k})$ as a direct limit of the categories $D^b_{J,\chi_J^*\mathcal{L}_{\mathrm{AS}}}(Y,\mathbf{k})$, indexed by finite and closed unions of $\mathrm{Iw}_u^+$-orbits $Y$ which are ordered by inclusion. Notice that, since the transition maps are push-forwards of closed immersions, they are fully faithful functors, and one can see this limit as an increasing union of categories.
	
	We now  recall the definition and some of the properties of parity complexes on partial affine flag varieties, since those will play a major role in the sequel. Let $Y\subset \mathrm{Fl}^{\circ}_{\mathbf{g}}$ be a finite locally closed union of $\mathrm{Iw}_{u}^+$-orbits. A complex $\mathcal{F}\in D^b_{\mathcal{IW}}(Y,\mathbf{k})$ is called $*$-even, resp. $!$-even, if for any inclusion $j:Y'\hookrightarrow Y$ of an $\mathrm{Iw}_{u}^+$-orbit $Y'$, the complex $j^*\mathcal{F}$, resp. $j^!\mathcal{F}$, is concentrated in even degrees. One defines similarly $*$-odd and $!$-odd complexes and say that a complex is even, resp. odd, if it is both $*$-even and $!$-even, resp. $*$-odd and $!$-odd. A complex is called parity if it is a direct sum of even and odd complexes. 
	
	Since the category  $D^b_{\mathcal{IW}}(\mathrm{Fl}^{\circ}_{\mathbf{g}},\mathbf{k})$ is defined as a direct limit, it makes sense to talk about even and odd objects in this category, and we will denote by $\mathrm{Par}_{\mathcal{IW}}(\mathrm{Fl}^{\circ}_{\mathbf{g}},\mathbf{k})$ the additive full subcategory consisting of parity objects. The general theory of parity complexes (from \cite{2014}) allows us to state the following result.
	\begin{prop}\label{par prop}
		Let $\mathbf{g}\subset\overline{\mathbf{a}_1}$ be a facet. For each $\mathrm{Iw}_{u}^+$-orbit $Y$ in $\mathrm{Fl}^{\circ}_{\mathbf{g}}$ which supports a non-zero Iwahori-Whittaker rank one local system $\mathcal{L}$, there exists a unique  indecomposable parity complex in $\mathrm{Par}_{\mathcal{IW}}(\mathrm{Fl}^{\circ}_{\mathbf{g}},\mathbf{k})$ supported on $\overline{Y}$ whose restriction to $Y$ is isomorphic to $\mathcal{L}[\mathrm{dim}(Y)]$, and each indecomposable parity complex is isomorphic to such an object up to a cohomological shift. Moreover, each object in $\mathrm{Par}_{\mathcal{IW}}(\mathrm{Fl}^{\circ}_{\mathbf{g}},\mathbf{k})$ is isomorphic to a sum of indecomposable parity complexes.
	\end{prop}
 \begin{rem}
     See Proposition \ref{prop IW support} below for a characterization of $\mathrm{Iw}_{u}^+$-orbits which supports a non-zero Iwahori-Whittaker rank one local system.
 \end{rem}
	\section{Parity sheaves on partial affine flag varieties and equivalence relations} 
	\subsection{Pushforward and pullback of indecomposable parity complexes}\label{Pushforward and pullback of indecomposable parity complexes} Let $\mathbf{g}\subset\overline{\mathbf{a}_1}$ be a facet and $\pi_{\mathbf{g}}:\mathrm{Fl}^{\circ}_{\mathbf{a}_1}\to \mathrm{Fl}^{\circ}_{\mathbf{g}}$ be the canonical projection. In this section, we will explain what is the effect of applying the functors $\pi_\mathbf{g}^*$ and $\pi_{\mathbf{g} *}$ to the Iwahori-Whittaker-equivariant indecomposable parity complexes that we described in Proposition \ref{par prop}.
	
	Let us first state a lemma which will be needed below. Let $(\Lambda,\leq)$ be a finite set endowed with a partial order $\leq$, admitting a unique minimal element $\lambda_0$ and equipped with a function $l:\Lambda\to \mathbb{Z}_{\geq0}$ compatible with the order $\leq$ (i.e. such that $\lambda\leq\mu\Rightarrow l(\lambda)\leq l(\mu)$). Let also $(X,(X_\lambda)_{\lambda\in\Lambda})$ be a stratified variety over $\mathbb{F}$. We make the following assumptions on the stratification:\begin{itemize}
	    \item for every $\lambda\in \Lambda$, $X_\lambda$ is an affine space of dimension $l(\lambda)$ over $\mathbb{F}$, and we have $\overline{X_\lambda}=\bigsqcup_{\mu\leq\lambda}X_\mu$;
	    \item there exist isomorphisms of $\mathbb{F}$-schemes (which we fix) $X_{\lambda}\simeq X_{\lambda_0}\times_{\mathbb{F}} \mathbb{A}_{\mathbb{F}}^{l(\lambda)-l(\lambda_0)}$ for every $\lambda\in\Lambda$;
	    \item there exists a morphism of $\mathbb{F}$-schemes $q:X\to X_{\lambda_0}$ such that $q|_{X_\lambda}$ is induced by the canonical projection on the first component $X_{\lambda_0}\times_{\mathbb{F}} \mathbb{A}_{\mathbb{F}}^{l(\lambda)-l(\lambda_0)}\to X_{\lambda_0}$ for every $\lambda\in \Lambda$.
	\end{itemize}
	\begin{lem}\label{constant sheaf}
		We have an isomorphism in $D_c^b(X_{\lambda_0},\mathbf{k})$:
		$$q_!\underline{\mathbf{k}}_X\simeq \bigoplus_{\lambda\in\Lambda}\underline{\mathbf{k}}_{X_{\lambda_0}}[-2(l(\lambda)-l(\lambda_0))].$$
	\end{lem}
	\begin{proof}
		The proof is done by induction on the cardinality of $\Lambda$. First note that the result is trivial when $\#\Lambda=1$, so we assume from now on that $\#\Lambda\geq2$. Let $\lambda\in\Lambda$ be a maximal element for $\leq$, so that we have an open immersion $i:X_\lambda\hookrightarrow X$. Denoting by $j$ the closed immersion $X':=X\backslash X_\lambda\hookrightarrow X$, we get a distinguished triangle
		$$i_!i^* \underline{\mathbf{k}}_X\longrightarrow \underline{\mathbf{k}}_X \longrightarrow j_!j^*\underline{\mathbf{k}}_X\xrightarrow{+1}  $$
		to which we can apply the triangulated functor $q_!$:
		\begin{equation}\label{dist triangle}
			q_!i_!i^* \underline{\mathbf{k}}_X\longrightarrow q_!\underline{\mathbf{k}}_X \longrightarrow q_!j_!j^*\underline{\mathbf{k}}_X\xrightarrow{+1}. 
		\end{equation}
		Since $q\circ i$ is the projection induced by $X_{\lambda_0}\times_{\mathbb{F}} \mathbb{A}_{\mathbb{F}}^{l(\lambda)-l(\lambda_0)}\to X_{\lambda_0}$ and $i^* \underline{\mathbf{k}}_X\simeq \underline{\mathbf{k}}_{X_\lambda}$, we have an isomorphism $q_!i_!i^* \underline{\mathbf{k}}_X\simeq \underline{\mathbf{k}}_{X_{\lambda_0}}[-2(l(\lambda)-l(\lambda_0))]$. On the other hand, we have $j^*\underline{\mathbf{k}}_X\simeq\underline{\mathbf{k}}_{X'}$, and the morphism $q':=q\circ j:X'\to X_{\lambda_0}$ satisfies $q'|_{X_\mu}=q|_{X_\mu}$ for every $\mu\neq\lambda$. Thus, we can apply our induction hypothesis to the stratified space $(X',(X_\mu)_{\mu\neq\lambda})$ to find that 
		$$q_!j_!j^*\underline{\mathbf{k}}_X\simeq q'_!\underline{\mathbf{k}}_{X'}\simeq \bigoplus_{\mu\neq\lambda}\underline{\mathbf{k}}_{X_{\lambda_0}}[-2(l(\mu)-l(\lambda_0))].$$
		In particular, we see that $q_!i_!i^* \underline{\mathbf{k}}_X$ and $q_!j_!j^*\underline{\mathbf{k}}_X$ only have cohomologies in even degrees, so that the distinguished triangle (\ref{dist triangle}) is split (recall that $X_{\lambda_0}$ is an affine space, so it only admits cohomology in even degrees). Therefore we have an isomorphism
		$$ q_!\underline{\mathbf{k}}_X\simeq  q_!i_!i^* \underline{\mathbf{k}}_X\oplus q_!j_!j^*\underline{\mathbf{k}}_X.$$
		This concludes the proof.
	\end{proof}
	Fix a facet $\mathbf{g}\subset\overline{\mathbf{a}_1}$. We denote by $W_{\mathbf{g}}\subset W$ the stabilizer of $\mathbf{g}$ for the box action $\car_1$ and put $S_\mathbf{g}:=S\cap W_\mathbf{g}$. It is well known that $(W_\mathbf{g},S_\mathbf{g})$ is a finite Coxeter system\footnote{Note that $W_\mathbf{g}$ coincides with the stabilizer of the facet $\ell\cdot\mathbf{g}\subset\overline{\mathbf{a}_\ell}$ for the dilated box action $\car_\ell$ of $W$.}. We will denote by $w_\mathbf{g}$ the element of maximal length in $W_{\mathbf{g}}$, and by $W^{\mathbf{g}}$ the set of elements $w\in W$ for which $w$ is maximal in the left coset $wW_{\mathbf{g}}$. We put $w_{0}:=w_{\{0\}}$ (notice that $W_{0}=W_{\{0\}}$). The set of elements $w\in W$ which are minimal in $W_0w$ will be denoted by ${_\mathrm{f}W}$, and we put $${_\mathrm{f}W}^\mathbf{g}=W^\mathbf{g}\cap {_\mathrm{f}W}.$$
	The following well-known result (which is a direct consequence of \cite[Lemma 2.2]{achar2021geometric}) will be useful throughout all the rest of this paper.
	\begin{prop}\label{prop min coset}
		Let $w\in {W}^\mathbf{g}$. We have the following equivalences
		$$w\in {_{\mathrm{f}}W}^\mathbf{g}\Longleftrightarrow \forall r\in W_\mathbf{g},~wr\in {_{\mathrm{f}}W}.$$
	\end{prop}
	For any $w\in W^\mathbf{g}$, we choose a lift $\dot{w}$ of $w$ in $N_G(T)(\mathcal{K})$ (see \cite[§4.2]{RW22} for the construction of such a lift), and denote by $\mathscr{X}^\mathbf{g}_{w}$ the $\mathrm{Iw}_{u}^+$-orbit in $\mathrm{Fl}^{\circ}_{\mathbf{g}}$ of the image of $\dot{w}$ under the canonical projection $LG(\mathbb{F})\to \mathrm{Fl}^{\circ}_{\mathbf{g}}(\mathbb{F})$. We then have a stratification
	\begin{equation}\label{startification bruhat}
	    (\mathrm{Fl}^{\circ}_{\mathbf{g}})_\mathrm{red}=\bigsqcup_{w\in {W}^\mathbf{g}} \mathscr{X}^\mathbf{g}_{w},
	\end{equation}
	 of the reduced ind-scheme associated with $\mathrm{Fl}^{\circ}_{\mathbf{g}}$, where each orbit $\mathscr{X}^\mathbf{g}_{w}$ is isomorphic to an affine space over $\mathbb{F}$ whose dimension is equal to the length of the minimal element in $w_0wW_\mathbf{g}$. All the properties concerning this stratification that we use in the sequel follow from the analogous standard facts concerning the opposite Iwahori subgroup $\mathrm{Iw}_{\mathrm{u}}:=\dot{w}_0\mathrm{Iw}^+_{\mathrm{u}}\dot{w}_0$, cf. \cite[§A]{epiga:4984}. 
  \begin{prop}\label{prop IW support}
      For any $w\in W^\mathbf{g}$, the $\mathrm{Iw}_{u}^+$-orbit $\mathscr{X}^\mathbf{g}_{w}$ supports a non-zero Iwahori-Whittaker equivariant local system if and only if $w\in {_\mathrm{f}W}$, which means that $w\in {_\mathrm{f}W}^\mathbf{g}$.
  \end{prop}
	
	For any $w\in{_\mathrm{f}W}^\mathbf{g}$, we will denote by $\mathcal{L}^\mathbf{g}_{w}$, resp.  $\mathcal{E}^{\mathbf{g}}_w\in \mathrm{Par}_{\mathcal{IW}}(\mathrm{Fl}^{\circ}_{\mathbf{g}},\mathbf{k})$, resp. $\nabla^{\mathbf{g}}_w\in D^b_{\mathcal{IW}}(\mathrm{Fl}^{\circ}_{\mathbf{g}},\mathbf{k})$, the rank one Iwahori-Whittaker equivariant local system, resp. the indecomposable parity complex from Proposition \ref{par prop}, resp. the costandard perverse sheaf,  associated with $\mathscr{X}^\mathbf{g}_{w}$. 
	
	Recall that $\pi_{\mathbf{g}}:\mathrm{Fl}^{\circ}_{\mathbf{a}_1}\to \mathrm{Fl}^{\circ}_{\mathbf{g}}$ is the canonical projection, which is a proper morphism of ind-schemes, and denote by $N_{\mathbf{g}}$ the length of $w_\mathbf{g}$. For any $v\in{_\mathrm{f}W}$, we write $\mathscr{X}_v$  (resp. $\mathcal{L}_v$, resp. $\mathcal{E}_v$, resp. $\nabla_v$) instead of $\mathscr{X}_v^{\mathbf{a}_1}$ (resp. $\mathcal{L}_v^{\mathbf{a}_1}$, resp. $\mathcal{E}^{\mathbf{a}_1}_v$, resp. $\nabla^{\mathbf{a}_1}_v$). For every $w\in {_\mathrm{f}W}^\mathbf{g}$ and $v\in {_\mathrm{f}W}$, we fix isomorphisms of $\mathbb{F}$-schemes
	\begin{equation}\label{iso schemes}
	    \mathscr{X}_{v}\simeq\mathbb{A}_\mathbb{F}^{l(w_0v)},\quad \mathscr{X}_{w}^\mathbf{g}\simeq \mathbb{A}_\mathbb{F}^{l(w_0ww_\mathbf{g})}.
	\end{equation}
	Let $w\in {_\mathrm{f}W}^\mathbf{g}$. It is important to note that $\pi_{\mathbf{g}}$ is a locally trivial fibration, with $\pi_\mathbf{g}^{-1}(\mathscr{X}_{w}^\mathbf{g})=\sqcup_{x\in W_\mathbf{g}}\mathscr{X}_{wx}$, and that for any $h\in W_\mathbf{g}$ the morphism $\pi_\mathbf{g}|_{\mathscr{X}_{ww_\mathbf{g}h}}:\mathscr{X}_{ww_\mathbf{g}h}\to \mathscr{X}_{w}^\mathbf{g}$ identifies with the canonical projection on the first component
	$$\mathscr{X}_{ww_\mathbf{g}h}\simeq\mathscr{X}_{ww_\mathbf{g}}\times \mathscr{X}_{w_0h}\to \mathscr{X}_{w}^\mathbf{g}.$$
	The projection on the first component comes from the identification of $\mathscr{X}_{ww_\mathbf{g}}$ and $\mathscr{X}_{w}^\mathbf{g}$ with an affine space of dimension $l(w_0ww_\mathbf{g})$ thanks to (\ref{iso schemes}); fix $h\in W_\mathbf{g}$, and let us explain the first isomorphism above. For any $v\in W$, denote by $\overline{v}$ the image of $\dot{v}$ under the canonical projection $LG(\mathbb{F})\to \mathrm{Fl}^{\circ}_{\mathbf{a}_1}(\mathbb{F})$. By construction we have 
	$$\mathscr{X}_{ww_\mathbf{g}h}=\dot{w}_0\mathrm{Iw}_{\mathrm{u}}\dot{w}_0\cdot\overline{ww_\mathbf{g}h}= \dot{w}_0\mathrm{Iw}_{\mathrm{u}}\cdot\overline{w_0ww_\mathbf{g}h}.$$ 
	But thanks to our hypothesis that $w\in W^\mathbf{g}$, we know that $ww_\mathbf{g}$ is minimal in $ww_\mathbf{g}W_\mathbf{g}$, and our hypothesis that $w\in {_\mathrm{f}W}$ implies that $wr\in {_\mathrm{f}W}$ for all $r\in W_\mathbf{g}$ thanks to Proposition \ref{prop min coset} (in particular this holds when $r\in\{w_\mathbf{g},w_\mathbf{g}h\}$); these facts imply that we have $$l(w_0ww_\mathbf{g}h)=l(w_0)+l(ww_\mathbf{g}h)=l(w_0)+l(ww_\mathbf{g})+l(h)=l(w_0ww_\mathbf{g})+l(h).$$ From these equalities of lengths and the fact that $\mathrm{Iw}_{\mathrm{u}}\cdot\overline{x}\simeq \mathbb{A}_\mathbb{F}^{l(x)}$ for any $x\in W$, we deduce the first isomorphism below
	$$\mathrm{Iw}_{\mathrm{u}}\cdot\overline{w_0ww_\mathbf{g}h}\simeq \mathrm{Iw}_{\mathrm{u}}\cdot\overline{w_0ww_\mathbf{g}}\times \mathrm{Iw}_{\mathrm{u}}\cdot\overline{h}= \mathrm{Iw}_{\mathrm{u}}\dot{w}_0\cdot\overline{ww_\mathbf{g}}\times \mathrm{Iw}_{\mathrm{u}}\dot{w}_0\cdot\overline{w_0h},$$
	and hence finally
		$$\dot{w}_0\mathrm{Iw}_{\mathrm{u}}\cdot\overline{w_0ww_\mathbf{g}h}\simeq \dot{w}_0\mathrm{Iw}_{\mathrm{u}}\dot{w}_0\cdot\overline{ww_\mathbf{g}}\times \dot{w}_0\mathrm{Iw}_{\mathrm{u}}\dot{w}_0\cdot\overline{w_0h}.$$
	The right-hand side corresponds to $\mathscr{X}_{ww_\mathbf{g}}\times \mathscr{X}_{w_0h}$.
	\begin{lem}
		\phantomsection
		\label{lem1}
		\begin{enumerate}
			\item The functors $\pi_{\mathbf{g}}^*$ and $\pi_{{\mathbf{g}}*}$ send parity complexes to parity complexes.
			\item For any $v\in {_\mathrm{f}W}^\mathbf{g}$, we have isomorphisms
			$$\pi_{\mathbf{g}}^*[N_{\mathbf{g}}](\mathcal{E}_v^\mathbf{g})\simeq \mathcal{E}_v\qquad\text{and}\qquad\pi_{\mathbf{g}*}\mathcal{E}_v\simeq \bigoplus_{w\in W_\mathbf{g}}\mathcal{E}_v^\mathbf{g}[-2(l(w_0w)-l(w_0v))].$$
			\item Let $w\in{_\mathrm{f}W}$, and write $w=w^{\mathbf{g}}h$ for some  $w^{\mathbf{g}}\in W^{\mathbf{g}},h\in W_{\mathbf{g}}$. Then we have 
			\begin{align*}
				\pi_{\mathbf{g}*}\nabla_w&\simeq\nabla^{\mathbf{g}}_{w^{\mathbf{g}}}[l(w_\mathbf{g}h)]\quad\text{if}~w^{\mathbf{g}}\in{_\mathrm{f}W},\\
				\pi_{\mathbf{g}*}\nabla_w&=0\quad\text{otherwise.}
			\end{align*}
		\end{enumerate}
	\end{lem}
	\begin{proof}
		The first (resp. third) point can be proven just as in \cite[Proposition A.2]{epiga:4984} (resp. \cite[Lemma A.1]{epiga:4984}). 
		
		Let us prove the second point, for which we will take back most of the arguments of  \cite[Lemma A.5]{epiga:4984}. The second isomorphism will arise while proving the first one. 
		
		We first make the following observation: let us write $X:=\bigsqcup_{w\in W_\mathbf{g}}\mathscr{X}_{vw}$ and put $q:=\pi_\mathbf{g}|_X$, the object $q^*\mathcal{L}^\mathbf{g}_v$ is a rank one Iwahori-Whittaker local system on $X$, whose restriction to each orbit $\mathscr{X}_{vw}$ coincides with $\mathcal{L}_{vw}$ (recall $vw$ does belong to ${_\mathrm{f}W}$ thanks to Proposition \ref{prop min coset}). In particular, $q^*\mathcal{L}^\mathbf{g}_v$ is indecomposable.
		
		The object $\pi_{\mathbf{g}}^*[N_{\mathbf{g}}](\mathcal{E}_v^\mathbf{g})$ is parity by (1), with $\mathscr{X}_{v}$ open in its support and its restriction to this stratum coinciding with $\mathcal{L}_v[l(w_0v)]$ (recall that $\mathscr{X}_{v}$ is of dimension $l(w_0v)$ thanks to (\ref{iso schemes})). Therefore we can write
		\begin{equation}\label{eqpi0}
		    \pi_{\mathbf{g}}^*[N_{\mathbf{g}}](\mathcal{E}_v^\mathbf{g})\simeq \mathcal{E}_v\oplus\mathcal{G}, 
		\end{equation}
		for some object $\mathcal{G}\in \mathrm{Par}_{\mathcal{IW}}(\mathrm{Fl}^{\circ}_{\mathbf{a}_1},\mathbf{k})$. Since the restriction of $\pi_{\mathbf{g}}^*[N_{\mathbf{g}}](\mathcal{E}_v^\mathbf{g})$ to $X$ is the indecomposable object $q^*\mathcal{L}^\mathbf{g}_v[l(w_0v)]$, it must be isomorphic to  $\mathcal{E}_v|_X$, because the latter is non-zero on $\mathscr{X}_{v}$.

		Thus, the base change theorem implies  the first isomorphism below  \begin{equation}\label{res}
		\begin{split}
		    (\pi_{\mathbf{g}*}\mathcal{E}_v)|_{\mathscr{X}^\mathbf{g}_{v}}&\simeq q_*q^*\mathcal{L}^\mathbf{g}_v[l(w_0v)]\simeq q_!\underline{\mathbf{k}}_X\otimes^L\mathcal{L}^\mathbf{g}_v[l(w_0v)] \\&\simeq \bigoplus_{w\in W_\mathbf{g}}\mathcal{L}^\mathbf{g}_v[-2(l(w_0w)-l(w_0v))],
		\end{split}	
		\end{equation}the second isomorphism is implied by the projection formula (and the fact that $q_!=q_*$ because $q$ is proper), and the third isomorphism is a direct application of Lemma \ref{constant sheaf}. Since $\mathscr{X}^\mathbf{g}_{v}$ is open in the support of the parity complex $\pi_{\mathbf{g}*}\mathcal{E}_v$, we get an isomorphism
		\begin{equation}\label{eqpi1}
		    \pi_{\mathbf{g}*}\mathcal{E}_v\simeq \bigoplus_{w\in W_\mathbf{g}}\mathcal{E}_v^\mathbf{g}[-2(l(w_0w)-l(w_0v))]\oplus \mathcal{G}', 
		\end{equation}
		for some $\mathcal{G}'\in \mathrm{Par}_{\mathcal{IW}}(\mathrm{Fl}^{\circ}_{\mathbf{g}},\mathbf{k})$. Applying $\pi_{\mathbf{g}*}$ to (\ref{eqpi0}) and using (\ref{eqpi1}), we deduce the following isomorphism of parity complexes:
		$$\pi_{\mathbf{g}*}\pi_{\mathbf{g}}^*[N_{\mathbf{g}}]\mathcal{E}^\mathbf{g}_v\simeq \bigoplus_{w\in W_\mathbf{g}}\mathcal{E}_v^\mathbf{g}[-2(l(w_0w)-l(w_0v))]\oplus \pi_{\mathbf{g}*}\mathcal{G}\oplus \mathcal{G}'.$$
		Now, for any parity complex $\mathcal{E}\in \mathrm{Par}_{\mathcal{IW}}(\mathrm{Fl}^{\circ}_{\mathbf{g}},\mathbf{k})$ and $w\in {_\mathrm{f}W}^\mathbf{g}$, one may write (following \cite{williamson2012modular}): $$\mathcal{E}|_{\mathscr{X}^\mathbf{g}_{w}}\simeq V(\mathcal{E})_w\otimes_\mathbf{k}\mathcal{L}_w^\mathbf{g},$$ where $V(\mathcal{E})_w$ is a finite-dimensional graded $\mathbf{k}$-vector space (and one can do likewise for any $\mathcal{E}\in \mathrm{Par}_{\mathcal{IW}}(\mathrm{Fl}^{\circ}_{\mathbf{a}_1},\mathbf{k})$ and $w\in {_\mathrm{f}W}$). From the description of $\pi_\mathbf{g}$ as a locally trivial fibration and from the first observation that was made at the beginning of the proof, we see that 
		$$\mathrm{dim}_\mathbf{k}(V(\pi_{\mathbf{g}}^*[N_{\mathbf{g}}]\mathcal{E}^\mathbf{g}_v)_{uw})=\mathrm{dim}_\mathbf{k}(V(\mathcal{E}^\mathbf{g}_v)_{u}),~\forall (u,w)\in {_\mathrm{f}W}^\mathbf{g}\times W_\mathbf{g}, $$
		where we forget the grading of our vector spaces when taking their dimension. Moreover, the same arguments that were used to prove (\ref{res}) imply that
		$$\mathrm{dim}_\mathbf{k}(V(\pi_{\mathbf{g}*}\pi_{\mathbf{g}}^*[N_{\mathbf{g}}]\mathcal{E}^\mathbf{g}_v)_{u})=|W_\mathbf{g}|\cdot\mathrm{dim}_\mathbf{k}(V(\pi_{\mathbf{g}}^*[N_{\mathbf{g}}]\mathcal{E}^\mathbf{g}_v)_{u}),~\forall u\in {_\mathrm{f}W}^\mathbf{g}. $$
		Thus we get  $$\mathrm{dim}_\mathbf{k}(V(\pi_{\mathbf{g}*}\pi_{\mathbf{g}}^*[N_{\mathbf{g}}]\mathcal{E}^\mathbf{g}_v)_{u})=\mathrm{dim}_\mathbf{k}(V(\mathcal{F})_{u}),~ \forall u\in {_\mathrm{f}W}^\mathbf{g},$$
		where $\mathcal{F}:=\bigoplus_{w\in W_\mathbf{g}}\mathcal{E}_v^\mathbf{g}[-2(l(w_0w)-l(w_0v))]$. We deduce that $V(\pi_{\mathbf{g}*}\mathcal{G})_u=V(\mathcal{G}')_u=0$ for all $u\in {_\mathrm{f}W}^\mathbf{g}$, so that $\pi_{\mathbf{g}*}\mathcal{G}\simeq\mathcal{G}'\simeq 0$ and $\pi_{\mathbf{g}*}\pi_{\mathbf{g}}^*[N_{\mathbf{g}}]\mathcal{E}^\mathbf{g}_v\simeq \mathcal{F}$. This also implies that $\mathcal{G}\simeq 0$, and therefore we get the desired isomorphism
		$$\pi_{\mathbf{g}}^*[N_{\mathbf{g}}](\mathcal{E}_v^\mathbf{g})\simeq \mathcal{E}_v.$$	\end{proof}	
	\subsection{The antispherical module} Let $\mathcal{H}$ be the Hecke algebra associated with $(W,S)$, with standard basis $(H_w,~w\in W)$, and denote by $\mathcal{N}$ its antispherical module (we follow the notation of \cite[§3]{Soergel1997KazhdanLusztigPA}, the antispherical module is denoted by $\mathcal{M}^{\mathrm{asph}}$ in \cite{RW22}), with standard basis $(N_w,w\in {_\mathrm{f}W})$ and \textit{$\ell$-canonical basis} $({^\ell\underline{N}}_w,w\in {_\mathrm{f}W})$. We denote by $ [\mathrm{Par}_{\mathcal{IW}}(\mathrm{Fl}^{\circ}_{\mathbf{a}_1},\mathbf{k})]$ the split Grothendieck group of the additive category $\mathrm{Par}_{\mathcal{IW}}(\mathrm{Fl}^{\circ}_{\mathbf{a}_1},\mathbf{k}$. We have a canonical isomorphism of groups
	$$
		\mathrm{ch}  : [\mathrm{Par}_{\mathcal{IW}}(\mathrm{Fl}^{\circ}_{\mathbf{a}_1},\mathbf{k})]  \xrightarrow{\sim} \mathcal{N}$$ 
  defined by 
  $$\mathrm{ch}([\mathcal{F}]):=\sum_{w\in {_\mathrm{f}W}}\left(\sum_{n\in\mathbb{Z}}\mathrm{dim}_\mathbf{k}~\mathrm{Hom}_{D^b_{\mathcal{IW}}(\mathrm{Fl}^{\circ}_{\mathbf{a}_1},\mathbf{k})}(\mathcal{F},\nabla_w[n])v^n\right)N_w.$$

	The morphism $\mathrm{ch}$ sends the indecomposable parity complex $\mathcal{E}_w$ onto ${^\ell\underline{N}}_w$. The $\ell$-Kazhdan-Lusztig polynomials $({^\ell n}_{x,y},~x,y\in {_\mathrm{f}W})$ are defined by the equality 
	\begin{equation}\label{def antispheric}
		{^\ell\underline{N}}_y=\sum_{x\in {_\mathrm{f}W}}{^\ell n}_{x,y}N_x.
	\end{equation}
	
	It makes sense to extend the definition of these polynomials to the whole group $W$, simply by putting ${^\ell n}_{x,y}=0$ whenever $x$ or $y$ does not belong to ${_\mathrm{f}W}$ (this consideration will slightly simplify the statement of the first point of Proposition \ref{prop max coset}). For any objects $\mathcal{E},\mathcal{F}\in \mathrm{Par}_{\mathcal{IW}}(\mathrm{Fl}^{\circ}_{\mathbf{g}},\mathbf{k})$, we set $$\mathrm{Hom}_{D^b_{\mathcal{IW}}(\mathrm{Fl}^{\circ}_{\mathbf{g}},\mathbf{k})}^\bullet(\mathcal{E},\mathcal{F}):=\bigoplus_{n\in\mathbb{Z}}\mathrm{Hom}_{D^b_{\mathcal{IW}}(\mathrm{Fl}^{\circ}_{\mathbf{g}},\mathbf{k})}(\mathcal{E},\mathcal{F}[n]).$$ It follows that we have 
	\begin{equation}\label{l-KL}
		{^\ell n}_{x,y}(1)=\mathrm{dim}_\mathbf{k}~\mathrm{Hom}_{D^b_{\mathcal{IW}}(\mathrm{Fl}^{\circ}_{\mathbf{a}_1},\mathbf{k})}^\bullet(\mathcal{E}_y,\nabla_x)
	\end{equation}
	When $\ell=0$, the $\ell$-canonical basis coincides with the usual Kazhdan-Lusztig basis from \cite[Theorem 3.1]{Soergel1997KazhdanLusztigPA}, which is denoted by $(n_{x,y},~x,y\in {_\mathrm{f}W})$ there, so we have ${^0n}_{x,y}= {n}_{x,y}$ (this is a consequence of the fact that the perversely shifted indecomposable parity complexes coincide with the intersection cohomology complexes when $\ell=0$). We have the following easy and useful observation.
	\begin{prop}\label{car 0}
		Let $x,y\in {_\mathrm{f}W}$. For any prime number $\ell$, we have ${^\ell n}_{x,y}(1)\geq {n}_{x,y}(1)$.
	\end{prop}
	\begin{proof}
	    The arguments are the same as the one used in \cite[Lemma 3.4]{achar2022geometric}, replacing tilting objects by indecomposable parity sheaves.
	\end{proof}

	\subsection{Equivalence relations on ${_\mathrm{f}W}^\mathbf{g}$}
	Let $\mathbf{g}\subset\overline{\mathbf{a}_1}$ be a facet.  In the sequel, we will consider a relation $\mathscr{R}_{\mathbf{g}}$ on the set ${_\mathrm{f}W}^{\mathbf{g}}$, defined by 
	$$w\mathscr{R}_{\mathbf{g}}w'\Longleftrightarrow \mathrm{Hom}^\bullet_{\mathrm{Par}_{\mathcal{IW}}(\mathrm{Fl}^{\circ}_{\mathbf{g}},\mathbf{k})}(\mathcal{E}^{\mathbf{g}}_{w},\mathcal{E}^{\mathbf{g}}_{w'})\neq 0$$ 
	for any $w,w'\in {_\mathrm{f}W}^{\mathbf{g}}$. The equivalence relation on ${_\mathrm{f}W}^{\mathbf{g}}$ generated by $\mathscr{R}_{\mathbf{g}}$ will be denoted by $\sim_{\mathbf{g}}$. The following results will be constantly used.
	\begin{prop}\label{anti-spherical relations}
		Let $u,v\in{_\mathrm{f}W}$ be such that ${^\ell n}_{u,v}(1)\neq 0$. Then $v\mathscr{R}_{\mathbf{a}_1} u$. Moreover, we get the same conclusion if ${n}_{u,v}(1)\neq 0$.
	\end{prop}
	\begin{proof}
		Recall that, thanks to Proposition \ref{car 0}, we have
		$$n_{u,v}(1)\leq {^\ell n}_{u,v}(1)=\mathrm{dim}_\mathbf{k}\mathrm{Hom}_{D^b_{\mathcal{IW}}(\mathrm{Fl}^{\circ}_{\mathbf{a}_1},\mathbf{k})}^\bullet(\mathcal{E}_v,\nabla_{u}),$$
		so that our claim is a direct consequence of \cite[Proposition 2.6]{2014}.
	\end{proof}
	The relations $\mathscr{R}_{\mathbf{a}_1}$ and $\mathscr{R}_{\mathbf{g}}$ are actually equivalent on ${_\mathrm{f}W}^{\mathbf{g}}$.
	\begin{prop}\label{proj form}
		Let $w,w'\in {_\mathrm{f}W}^{\mathbf{g}}$. We have $w\mathscr{R}_{\mathbf{g}}w'$ iff $w\mathscr{R}_{\mathbf{a}_1}w'$.
	\end{prop}
	\begin{proof}
		Using Lemma \ref{lem1}, we have
		\begin{align*}
			\mathrm{Hom}_{D^b_{\mathcal{IW}}(\mathrm{Fl}^{\circ}_{\mathbf{a}_1},\mathbf{k})}^\bullet(\mathcal{E}_w,\mathcal{E}_{w'})&\simeq \mathrm{Hom}_{D^b_{\mathcal{IW}}(\mathrm{Fl}^{\circ}_{\mathbf{a}_1},\mathbf{k})}^\bullet(\pi^*_{\mathbf{g}}[N_\mathbf{g}]\mathcal{E}_w^{\mathbf{g}},\pi^*_{\mathbf{g}}[N_\mathbf{g}]\mathcal{E}_{w'}^{\mathbf{g}})\\
			&\overset{\text{adjunction}}{\simeq} \mathrm{Hom}_{D^b_{\mathcal{IW}}(\mathrm{Fl}^{\circ}_{\mathbf{g}},\mathbf{k})}^\bullet(\mathcal{E}_w^{\mathbf{g}},\pi_{\mathbf{g}*}\pi_{\mathbf{g}}^*\mathcal{E}_{w'}^{\mathbf{g}}).
		\end{align*}
		From Lemma \ref{lem1}, we know that $\pi_{\mathbf{g}*}\pi_{\mathbf{g}}^*\mathcal{E}_{w'}^{\mathbf{g}}$ is isomorphic to a finite direct sum of shifts of $\mathcal{E}_{w'}^{\mathbf{g}}$,
		so the hom-spaces above are non-zero if and only if $\mathrm{Hom}_{D^b_{\mathcal{IW}}(\mathrm{Fl}^{\circ}_{\mathbf{g}},\mathbf{k})}^\bullet(\mathcal{E}_w^{\mathbf{g}},\mathcal{E}_{w'}^{\mathbf{g}})$ is non-zero.
	\end{proof}
	Putting together the two previous propositions yields:
	\begin{coro}\label{anti spherical implies R}
		Let $w,w'\in {_\mathrm{f}W}^{\mathbf{g}}$. We have $$n_{w',w}(1)\neq0\Rightarrow w\mathscr{R}_{\mathbf{g}}w'.$$
	\end{coro}
	\begin{rem}
		Define a relation $\mathscr{R}'_\mathbf{g}$ on ${_\mathrm{f}W}^{\mathbf{g}}$ by 
		$$w \mathscr{R}'_\mathbf{g}w'\Longleftrightarrow {^\ell n}_{w',w}(1)\neq0.$$
		Then, using Proposition \ref{proj form} together with \cite[Proposition 2.6]{2014}, it is not difficult to show that the equivalence relation generated by $\mathscr{R}'_\mathbf{g}$ is equal to $\sim_\mathbf{g}$.
	\end{rem}\section{Determination of the equivalence classes}\label{Determination of the blocks}
	The main goal of this section is to show that, when the root system $\mathfrak{R}^\vee$ (or, equivalently, $\mathfrak{R}$) is indecomposable and $\mathbf{g}\subset\overline{\mathbf{a}_1}$ is a non-special facet, the set ${_\mathrm{f}W}^\mathbf{g}$ consists of a single equivalence class for $\sim_\mathbf{g}$ (we will eventually see that the case where the root system is not indecomposable follows from this first case, see Proposition \ref{reduction to irred}). As we will see in the end (cf. subsection \ref{section special}), the case where $\mathbf{g}$ is a special facet can be deduced from the non-special case with the help of Smith-Treumann theory, see Proposition \ref{red to non special}.
	\subsection{The regular case}
	In this short paragraph, we show that  the set ${_{\mathrm{f}}W}$ is a single equivalence class for $\sim_{\mathbf{a}_1}$, i.e. that all the elements of ${_{\mathrm{f}}W}$ are in relation for $\sim_{\mathbf{a}_1}$. Although this case will be included in the more general statement of Theorem \ref{thm not points} (where $\mathbf{a}_1$ is replaced with a facet $\mathbf{g}\subset\overline{\mathbf{a}_1}$ which is not a point), we give an independent proof here, which is much simpler since many difficulties do not appear yet.
	\begin{rem}
	  The result we get in Proposition \ref{prop block sing} below implies (via the discussion at the end of subsection \ref{A geometric proof of the linkage principle}) that the block associated with a ``regular" dominant weight $\lambda\in \mathbb{X}^\vee_+$ (regular means inside of an alcove) is $W\bullet_\ell\lambda\cap\mathbb{X}^\vee_+$ (here we see $\mathbb{X}^\vee_+$ as the weight poset of $\mathrm{Rep}_\mathbf{k}(G^\vee)$). This last result was already much simpler to get than the general description of the block associated with an arbitrary dominant weight, see \cite[§2.4]{humphreys1978blocks}.
	\end{rem}
	\begin{lem}\label{lem reg bl}
		Let $w\in {_{\mathrm{f}}W}$. If $w\neq e$, then there exists $w'<w$ such that $w'\in{_\mathrm{f}W}$ and $w'\mathscr{R}_{\mathbf{a}_1} w$.
	\end{lem}
	\begin{proof}
		Writing a reduced expression for $w$, it is easy to see that there exists an $s\in S$ such that $ws<w$. Let $\mathbf{g}$ be the wall fixed by $s$. Since $w\in {_{\mathrm{f}}W}^\mathbf{g}$, Proposition \ref{prop min coset} implies that $ws\in {_{\mathrm{f}}W}$. By Lemma \ref{lem1}, we get
		\begin{align*}
			\mathrm{Hom}_{D^b_{\mathcal{IW}}(\mathrm{Fl}^{\circ}_{\mathbf{a}_1},\mathbf{k})}^\bullet(\mathcal{E}_w,\nabla_{ws})&\simeq \mathrm{Hom}_{D^b_{\mathcal{IW}}(\mathrm{Fl}^{\circ}_{\mathbf{a}_1},\mathbf{k})}^\bullet(\pi_{\mathbf{g}}^*[1]\mathcal{E}_w^{\mathbf{g}},\nabla_{ws})\\
			&\simeq \mathrm{Hom}_{D^b_{\mathcal{IW}}(\mathrm{Fl}^{\circ}_{\mathbf{g}},\mathbf{k})}^\bullet(\mathcal{E}_w^{\mathbf{g}},\pi_{\mathbf{g}*}[-1]\nabla_{ws})\\
			&\simeq \mathrm{Hom}_{D^b_{\mathcal{IW}}(\mathrm{Fl}^{\circ}_{\mathbf{g}},\mathbf{k})}^\bullet(\mathcal{E}_w^{\mathbf{g}},\nabla_{w}^{\mathbf{g}})\neq 0
		\end{align*}
		where the second line is obtained by adjunction. So if we put $w':=ws$, we get that ${^\ell n}_{w',w}(1)\neq 0$, and Proposition \ref{anti-spherical relations} allows us to conclude.
	\end{proof}
	It is now straightforward to conclude.
	\begin{prop}\label{prop block sing}
		The set ${_\mathrm{f}W}$ consists of a single class for the equivalence relation $\sim_{\mathbf{a}_1}$.
	\end{prop}
	\begin{proof}
		Let $w\in {_{\mathrm{f}}W}$. By the previous lemma and an induction on the length of $w$, we can see that $w\sim_{\mathbf{a}_1}e$. This concludes the proof.
	\end{proof}
	The proof of the fact that all of the elements of ${_{\mathrm{f}}W}^\mathbf{g}$ are in relation when $\mathbf{g}\subset\overline{\mathbf{a}_1}$ is an arbitrary non-special facet will be much more involved. This is due to the fact that, for an arbitrary $w\in {_{\mathrm{f}}W}^\mathbf{g}$, there is no obvious choice of an element $w'<w$ in ${_{\mathrm{f}}W}^\mathbf{g}$ such that $w'\sim_{\mathbf{{g}}}w$. The goal of the next subsection is to find such relations.
	\subsection{Some invariance properties of the $\ell$-anti-spherical Kazhdan-Lusztig polynomials}
	The following result is a generalization of Lemma \ref{lem reg bl}. 
	\begin{prop}\label{prop max coset}
		Let $\mathbf{q}$ be a facet inside $\overline{\mathbf{a}_1}$ and $w\in {_{\mathrm{f}}W}^\mathbf{q}$. \begin{enumerate}
			\item For all $w'\in {_\mathrm{f}W}$, we have ${^\ell n}_{w'r,w}(1)={^\ell n}_{w',w}(1)$ for all $r \in W_\mathbf{q}$.
			\item Assume moreover that $\mathbf{g}\subset \overline{\mathbf{a}_1}$ is a facet such that $\mathbf{q}\subset\overline{\mathbf{g}}$. For all $w_1,w_2\in wW_\mathbf{q}\cap W^{\mathbf{g}}$, we have $w_1\sim_{\mathbf{{g}}}w_2$.
		\end{enumerate} 
	\end{prop}
	
	\begin{proof}
		\begin{enumerate}
			\item Let $w'\in {_\mathrm{f}W}$ and write $w'=uh$, where $u$ is the maximal element in $w'W_\mathbf{q}$, and $h\in W_\mathbf{q}$. We have the following sequence of isomorphisms\footnote{Notice that we forget the grading of the vector spaces in those isomorphisms.} of $\mathbf{k}$-vector spaces for all $r\in W_\mathbf{q}$:
			\begin{align*}
				\mathrm{Hom}^\bullet_{D^b_{\mathcal{IW}}(\mathrm{Fl}^{\circ}_{\mathbf{a}_1},\mathbf{k})}(\mathcal{E}_{w},\nabla_{w'r})&\simeq \mathrm{Hom}^\bullet_{D^b_{\mathcal{IW}}(\mathrm{Fl}^{\circ}_{\mathbf{a}_1},\mathbf{k})}(\pi_{\mathbf{q}}^*[N_\mathbf{q}]\mathcal{E}_{w}^{\mathbf{q}},\nabla_{w'r})\\
				&\overset{\text{adjunction}}{\simeq} \mathrm{Hom}^\bullet_{D^b_{\mathcal{IW}}(\mathrm{Fl}^{\circ}_{\mathbf{q}},\mathbf{k})}(\mathcal{E}_{w}^{\mathbf{q}},(\pi_\mathbf{q})_*\nabla_{uhr})\\
				&\simeq \begin{cases}\mathrm{Hom}^\bullet_{D^b_{\mathcal{IW}}(\mathrm{Fl}^{\circ}_{\mathbf{q}},\mathbf{k})}(\mathcal{E}_{w}^{\mathbf{q}},\nabla_{u}^{\mathbf{q}})~&\text{if}~u\in {_\mathrm{f}W},\\
					0~&\text{otherwise},
				\end{cases}
			\end{align*}
			where the first and last isomorphisms are consequences of the first and third point of Lemma \ref{lem1} respectively. Thanks to the definition of the $\ell$-anti-spherical polynomials (see (\ref{def antispheric})), this means that ${^\ell n}_{w'r,w}(1)=0$ for all $r\in W_\mathbf{q}$ when $u\notin {_\mathrm{f}W}$ (which trivially implies that ${^\ell n}_{w'r,w}(1)={^\ell n}_{w',w}(1)$), and ${^\ell n}_{w'r,w}(1)={^\ell n}_{u,w}(1)={^\ell n}_{w',w}(1)$ otherwise.
			\item First notice that, since $w\in {_{\mathrm{f}}W}^\mathbf{q}$, we have that $wW_\mathbf{q}\subset {_{\mathrm{f}}W}$ thanks to Proposition \ref{prop min coset}, so that \begin{equation}\label{inclusion min}
			    wW_\mathbf{q}\cap W^{\mathbf{g}}\subset {_{\mathrm{f}}W}\cap W^{\mathbf{g}}={_{\mathrm{f}}W}^\mathbf{g}.
			\end{equation} Applying the previous point to $w'=w$, we get that ${ n}_{wr,w}(1)={ n}_{w,w}(1)=1$ for every $r\in W_\mathbf{q}$, and therefore that $w\mathscr{R}_\mathbf{g}wr$ if  $wr\in W^{\mathbf{g}}$ (which implies that $wr\in {_{\mathrm{f}}W}^\mathbf{g}$ by (\ref{inclusion min})) thanks to Corollary \ref{anti spherical implies R}. Thus, we have $w\mathscr{R}_\mathbf{g}w_i$ for $i\in\{1,2\}$, and we conclude by transitivity that $w_1\sim_{\mathbf{{g}}}w_2$.
		\end{enumerate}
	\end{proof} 
	The goal of the next subsection is to get a precise picture of the geometry of the affine space $\mathbb{X}^\vee\otimes_\mathbb{Z}\mathbb{R}$. By doing so, we will be able to use the previous relations together with the bijection between ${_{\mathrm{f}}W}^\mathbf{g}$ and $\mathcal{A}^+_\mathbf{g}$ (the set of dominant facets which are in $W\car_1\mathbf{g}$) to show that ${_{\mathrm{f}}W}^\mathbf{g}$ consists of a single class for $\sim_\mathbf{g}$ (when $\mathbf{g}$ is not a point).
	\subsection{Geometry of $\mathbb{X}^\vee\otimes_\mathbb{Z}\mathbb{R}$}\label{geometry}
	In this subsection, we follow the terminology and some of the notations\footnote{The general reference, which is used in \cite[§1]{LUSZTIG1980121}, is \cite{cdi_springer_books_10_1007_978_3_540_34491_9}.} of both \cite[§1]{LUSZTIG1980121} and \cite[§4]{Soergel1997KazhdanLusztigPA}. The box action $\car_1$ of $W$ on the affine space $E:=\mathbb{X}^\vee\otimes_\mathbb{Z}\mathbb{R}$ defines a set of hyperplanes $\mathscr{H}$ (one could also work with the dilated box action $\car_n$ for some $n\geq 1$ or with $\bullet_n$, cf. the setting of \cite[§1]{LUSZTIG1980121}). For any facet $\mathbf{p}$ (not necessarily included in $\overline{\mathbf{a}_1}$), we will denote by $W_\mathbf{p}$ its stabilizer in $W$ for the box action. We define the set of strictly dominant elements by 
	$$\mathscr{C}^+_0:=\{\lambda\in E~|~\langle\lambda,\alpha\rangle>0~\forall \alpha\in\mathfrak{R}_+\}. $$Recall that the connected components of $E\backslash\bigcup_{H\in\mathscr{H}}H$ are products of open simplices called alcoves\footnote{When the root system $\mathfrak{R}^\vee$ is irreducible, the alcoves are open simplices. In general, $\mathfrak{R}^\vee$ decomposes into a product of irreducible root systems, from which we deduce the decomposition of any alcove into a product of open simplices (cf.  subsection \ref{Facets which are not points}). }. An example is 
	$$\mathbf{a}_1=\{\mu\in E~|~0<\langle\mu,\alpha\rangle<1~\forall\alpha\in\mathfrak{R}_+\},$$
	which we call the fundamental alcove. The box action on the set of alcoves $\mathcal{A}$ can be extended to an action of $\tilde{W}$, in which case the stabilizer of  $\mathbf{a}_1$ is equal to $\Omega=\{w\in \tilde{W}~|~l(w)=0 \}$. Thus, the assignment $w\mapsto w\car_1\mathbf{a}_1$ yields a bijection $W\simeq \tilde{W}/\Omega\xrightarrow{\sim}\mathcal{A}$, so that  any alcove $A$ can be written as $A=w\car_1\mathbf{a}_1$ for a unique $w\in W$. This bijection $W\xrightarrow{\sim} \mathcal{A}$ allows us to define a right action of $W$ on $\mathcal{A}$, induced by right multiplication on itself. We may also define an order $\leq$ on $\mathcal{A}$, induced by the Bruhat order on $W$. Thus, for any $s\in S$ and $A\in \mathcal{A}$, we will have that $A<As$ if and only if $l(w)<l(ws)$, where $w\in W$ is such that $w\car_1\mathbf{a}_1=A$. Since $As$ is obtained by reflecting $A$ along its wall which is the $W$-conjugate of the wall of $\mathbf{a}_1$ fixed by $s$, this means that $A<As$ if and only if the number of hyperplanes separating $A$ from $\mathbf{a}_1$ (i.e. the number of hyperplanes $H\in\mathscr{H}$ such that $A$ and $\mathbf{a}_1$ are included in different connected components of $E\backslash H$) is smaller than the number of hyperplanes separating $As$ from $\mathbf{a}_1$ (see \cite[§1.4]{LUSZTIG1980121}). In the sequel, we will denote by $d(A)$ the number of hyperplanes separating $A$ from $\mathbf{a}_1$, so that we have $d(A)=l(w)$.
	
	The set of dominant alcoves (i.e. those included in $\mathscr{C}^+_0$) will be denoted by $\mathcal{A}^+$; we have a bijection ${_\mathrm{f}W}\xrightarrow{\sim}\mathcal{A}^+,~w\mapsto w\car_1\mathbf{a}_1$.
	
	For any facet $\mathbf{g}\subset \overline{\mathbf{a}_1}$, we define $\mathcal{A}_\mathbf{g}:=\{w\car_1\mathbf{g},~w\in W\}$ and  $\mathcal{A}_\mathbf{g}^{+}:=\{w\car_1\mathbf{g},~w\in {_\mathrm{f}W}^\mathbf{g}\}$. One can show that $\mathcal{A}_\mathbf{g}^{+}$ is exactly the subset of facets of $\mathcal{A}_\mathbf{g}$ which are included in $\mathscr{C}^+_0$. Note that the map $w \mapsto w\car_1\mathbf{g}$ induces a bijection between $W^\mathbf{g}$, resp. ${_\mathrm{f}W}^\mathbf{g}$, and $\mathcal{A}_\mathbf{g}$, resp. $\mathcal{A}_\mathbf{g}^{+}$. We will consider the left action of $W$ on $\mathcal{A}_\mathbf{g}$ given by $w\mathbf{h}:=w\car_1\mathbf{h}$ for any $\mathbf{h}\in\mathcal{A}_\mathbf{g}$. Thus, if $\mathbf{h}=u\mathbf{g}$ for some $u\in W^\mathbf{g}$, then $w\mathbf{h}=wu\mathbf{g}$, and the element of $W^\mathbf{g}$ corresponding to the facet $w\mathbf{h}$ is the maximal element of $wuW_\mathbf{g}$. Also note that, by continuity of the action of $W$ on $E$, if $A$ is an alcove containing $\mathbf{h}$ in its closure, then $w\mathbf{h}$ is the unique element of $\mathcal{A}_\mathbf{g}$ which is contained in the closure of $wA$. Let us state and prove the following easy result, which will be used a lot in the sequel.
	\begin{prop}\label{dominant facet}
		Let $\mathbf{h}$ be a facet containing a facet $\mathbf{p}$ in its closure such that $\mathbf{p}\subset \mathscr{C}^+_0$. Then we have $\mathbf{h}\subset \mathscr{C}^+_0$. In particular if $\mathbf{g}\subset\overline{\mathbf{a}_1}$ is the facet such that $\mathbf{h}\in\mathcal{A}_\mathbf{g}$, then there exists a unique element $w\in{_\mathrm{f}W}^\mathbf{g}$ such that $w\mathbf{g}=\mathbf{h}$.
	\end{prop}
	\begin{proof}
	By definition of $\mathscr{C}^+_0$, any facet is either included in $\mathscr{C}^+_0$ or in $E\backslash\mathscr{C}^+_0$, the latter being a closed subset of $E$. If $\mathbf{h}\subset E\backslash\mathscr{C}^+_0$, then $\overline{\mathbf{h}}\subset E\backslash\mathscr{C}^+_0$, contradicting our assumption on $\mathbf{p}$. Therefore $\mathbf{h}\subset \mathscr{C}^+_0$. The rest of the proposition follows.
	\end{proof}
	A \textit{special facet} will be a zero dimensional facet by which a maximal number of hyperplanes passes through. An example is the point $\{0\}$, and all the other special facets are translates of this point by elements of $\mathbb{X}^\vee$, i.e. are of the form $\{\lambda\}$ for $\lambda\in\mathbb{X}^\vee$ (cf.\cite[chap. VI, §2, Proposition 3]{cdi_springer_books_10_1007_978_3_540_34491_9}). For any special facet $\mathbf{v}$, we will denote by $W_\mathbf{v}$ the stabilizer of $\mathbf{v}$, which is the subgroup of $W$ generated by the reflections with respect to the hyperplanes passing through $\mathbf{v}$. If $\mathbf{v}=\{\lambda\}$, then $W_\mathbf{v}=t_\lambda W_0 t_{-\lambda}$. We will write $w_\mathbf{v}:=t_\lambda w_0t_{-\lambda}$, which is an element of $W_\mathbf{v}$.
	\begin{lem}\label{translation}
		Let $\mathbf{u},\mathbf{v}$ be special facets. Then $w_\mathbf{u}w_{\mathbf{v}}$ is a translation.
	\end{lem} 
	\begin{proof}
		Write $\mathbf{v}=\{\lambda\},~\mathbf{u}=\{\mu\}$ for some $\lambda,\mu\in\mathbb{X}^\vee$, then we have
		$$w_\mathbf{u}w_\mathbf{v}=t_{\mu}w_0t_{-\mu}t_{\lambda}w_0t_{-\lambda}=w_0t_{w_0\mu}t_{-\mu}t_{\lambda}w_0t_{-\lambda}=w_0^2t_{\mu}t_{-w_0\mu}t_{w_0\lambda}t_{-\lambda}=t_{\mu-\lambda+w_0(\lambda-\mu)}.$$
	\end{proof}
	
	Let $\mathbf{v}$ be a special facet, a quarter with vertex $\mathbf{v}$ is a connected component of 
	$$E\backslash\bigcup_{\substack{H\in\mathscr{H}\\\mathbf{v}\subset H}}H.$$
	When $\mathbf{v}={0}$, one such quarter is the dominant cone $\mathscr{C}^+_{0}$. For a general $\mathbf{v}$, we denote by $\mathscr{C}^+_\mathbf{v}$ the quarter with vertex $\mathbf{v}$ which is a translate of $\mathscr{C}^+_{0}$, and we put $\mathscr{C}^-_\mathbf{v}:=w_\mathbf{v}\mathscr{C}^+_\mathbf{v}$. We will also denote by $A^+_\mathbf{v}$, resp. $A^-_\mathbf{v}$, the unique alcove contained in $\mathscr{C}^+_\mathbf{v}$, resp. in $\mathscr{C}^-_\mathbf{v}$, containing $\mathbf{v}$ in its closure. We clearly have $A^-_\mathbf{v}=w_\mathbf{v}A^+_\mathbf{v}$.  Furthermore, let $\mathscr{H}^*$ be the set of hyperplanes $H$ of $\mathscr{H}$ such that $H$ is a wall of some $\mathscr{C}^+_\mathbf{v}$ for a special facet $\mathbf{v}$, we define \textit{boxes} to be the connected components of $E\backslash\bigcup_{H\in\mathscr{H}^*}H$. Each alcove is contained in a unique box, and we denote by $\Pi_\mathbf{v}$ the box containing $A^+_\mathbf{v}$ . In our context, if $\mathbf{v}=\{\lambda\}$ for some $\lambda\in\mathbb{X}^\vee$, we put $\Pi_{\lambda}:=\Pi_\mathbf{v}$, and we get 
	\begin{equation}\label{box}
		\Pi_{\lambda}=\{\mu\in E~|~0<\langle\mu-\lambda,\alpha\rangle<1~\forall \alpha\in S_0\}.
	\end{equation} 
	\begin{rem}\label{action compat}
		In \cite{LUSZTIG1980121}, Lusztig defines $\Omega$ to be the group generated by the orthogonal reflections through the hyperplanes of $\mathscr{H}$, which is seen as acting on the right on $E$. In our context, this right action is the left action of $W=\Omega$ on $E$, and the left action of $W$ on the set of alcoves from \cite{LUSZTIG1980121} is our right action on $\mathcal{A}$.
	\end{rem}
	\begin{prop}\label{conncomp}
		Let $H\in\mathscr{H}$. Then, there exists a unique connected component $E_1$ of $E\backslash H$ that has a non-empty intersection with $\mathscr{C}^+_\mathbf{v}$ for any special facet $\mathbf{v}$. Moreover, there exists at least one special facet $\mathbf{u}$ such that $\mathscr{C}^-_\mathbf{u}\cap E_1=\emptyset$.
	\end{prop}
	\begin{proof}
		Let $f:E\to \mathbb{R}$ be a linear form (where $E$ is seen as a real vector space with origin $0$) and $x\in \mathbb{X}^\vee$ such that $H=x+f^{-1}(\{0\})$. Then the connected components of $E\backslash H$ are $E_1:=x+f^{-1}(\mathbb{R}^*_+)$ and $E_2:=x+f^{-1}(\mathbb{R}^*_-)$. Since $f^{-1}(\{0\})\in\mathscr{H}$, $\mathscr{C}^+_0$ must be inside $f^{-1}(\mathbb{R}^*_+)$ or $f^{-1}(\mathbb{R}^*_-)$. We may assume that $\mathscr{C}^+_0\subset f^{-1}(\mathbb{R}^*_+)$. Let $y\in \mathbb{X}$ and put $\mathbf{v}:=\{y\}$, so that $\mathscr{C}^+_\mathbf{v}=y+\mathscr{C}^+_0$. For any $u\in\mathscr{C}^+_0$, we can find $r>0$ big enough so that $f(y-x+ru)>0$ (because $f(u)>0$), which implies that 
		$$y+ru\in (y+\mathscr{C}^+_0)\cap(x+f^{-1}(\mathbb{R}^*_+))=\mathscr{C}^+_\mathbf{v}\cap E_1.$$
		
		So $E_1$ has a non-empty intersection with any $\mathscr{C}^+_\mathbf{v}$. On the other hand, we have seen that $f^{-1}(\mathbb{R}^*_-)\cap \mathscr{C}^+_0=\emptyset$, which means that $E_2\cap \mathscr{C}^+_{x}=\emptyset$. This shows that $E_1$ is the only connected component of $E\backslash H$ having the desired property.
		
		Finally, we want to show that there exists a special facet $\mathbf{u}$ such that $\mathscr{C}^-_\mathbf{u}\cap E_1=\emptyset$. Recall that we have 
		$$\mathscr{C}^-_0=w_0\mathscr{C}^+_0=-\mathscr{C}^+_0.$$
		Since $\mathscr{C}^+_0\subset f^{-1}(\mathbb{R}^*_+)$, we must have $\mathscr{C}^-_0\cap f^{-1}(\mathbb{R}^*_+)=\emptyset$, and therefore $\mathscr{C}^-_x\cap E_1=\emptyset$.
	\end{proof}
	Let $H\in\mathscr{H}$, then $E\backslash H$ consists of two connected components $E^-_H$ and $E^+_H$. Thanks to Proposition \ref{conncomp}, we can set $E^+_H$ to be the one connected component that has a non-empty intersection with $\mathscr{C}^+_\mathbf{v}$ for any special facet $\mathbf{v}$. Following \cite[§4]{Soergel1997KazhdanLusztigPA} (see also \cite[§1.5]{LUSZTIG1980121}), we define a partial order $\preceq$ on $\mathcal{A}$ generated by the relations
	$$A\preceq s_HA\qquad\text{if}~ A\in\mathcal{A},~H\in \mathscr{H},~s_HA\subset E^+_H,$$
	where $s_H\in W$ denotes the affine reflection associated with the hyperplane $H$. One can show (cf. the proof of \cite[Claim 4.14]{Soergel1997KazhdanLusztigPA}) that this partial order $\preceq$ coincides with the usual Bruhat order $\leq$ on $\mathcal{A}^+$. The following result will be crucial in the sequel. We let $A$ be an alcove, $s\in S$ and $H$ be the hyperplane containing the wall separating $A$ and $As$.
	\begin{prop}\label{reverse order}
		Assume that $A,As\in\mathcal{A}^+$ and that there exists a special point $\mathbf{v}$ such that $w_\mathbf{v}A$ and $w_\mathbf{v}As$ belong to $\mathcal{A}^+$. Then $A<As$ if and only if $w_\mathbf{v}As<w_\mathbf{v}A$.
	\end{prop}
	\begin{proof} 
		Since the two orders $\leq$ and $\preceq$ coincide on $\mathcal{A}^+$ and $As=s_HA$, the fact that $A<As$ is equivalent to $As\subset E^+_H$.
		
		Because the actions of $W$ on the right and on the left on $\mathcal{A}$ commute, the hyperplane $H':=w_\mathbf{v}\car_1 H$ is the one separating $w_\mathbf{v}As$ from $w_\mathbf{v}A$. Using the fact that $w_\mathbf{v}A, w_\mathbf{v}As\in\mathcal{A}^+$, we see that $w_\mathbf{v}As<w_\mathbf{v}A$ if and only if $w_\mathbf{v}A$ is included in $E^+_{H'}$. Thus, to conclude it is enough to show that $E^+_{H'}=w_\mathbf{v}E^-_H$, or equivalently that $E^-_{H'}=w_\mathbf{v}E^+_H$. Notice that since $w_\mathbf{v}E^+_H$ is connected, it equals either $E^+_{H'}$ or $E^-_{H'}$. By Proposition \ref{conncomp}, we can find a special facet $\mathbf{u}$ such that $\mathscr{C}^-_\mathbf{u}\cap E^+_H=\emptyset$.  Then we have $w_\mathbf{v}\mathscr{C}^-_\mathbf{u}\cap w_\mathbf{v}E^+_H=\emptyset$. But 
		$$w_\mathbf{v}\mathscr{C}^-_\mathbf{u}=w_\mathbf{v}w_\mathbf{u}\mathscr{C}^+_\mathbf{u},$$
		and since $w_\mathbf{v}w_\mathbf{u}$ is a translation (thanks to Lemma \ref{translation}), $w_\mathbf{v}\mathscr{C}^-_\mathbf{u}$ is of the form $\mathscr{C}^+_\mathbf{w}$ for some special facet $\mathbf{w}$. So $\mathscr{C}^+_\mathbf{w}\cap w_\mathbf{v}E^+_H=\emptyset$, which implies that $E^-_{H'}=w_\mathbf{v}E^+_H$.
	\end{proof}
	
	Following \cite[§4, §5]{Soergel1997KazhdanLusztigPA}, we consider the bijection $\mathcal{A}\to \mathcal{A},~A\mapsto\check{A}$. This bijection sends an alcove $A$ inside of $\Pi_\mathbf{v}$ to the alcove $w_\mathbf{v}A$. We will denote by $A\mapsto\hat{A}$ the inverse of this bijection. For any alcove $B\subset\Pi:=\Pi_0$, we see that $\check{B}=w_0B$. For a general alcove $A$, write it uniquely as $A=\lambda+B$ for some $\lambda\in\mathbb{X}^\vee$, $B\subset\Pi$, and it follows that $\check{A}=\lambda+w_0B$.
	
	\begin{lem}\label{lemhat}
		Let $A\in\mathcal{A}$ and write $A=\lambda+B$, with $\lambda\in\mathbb{X}^\vee$, $B\subset\Pi$. We have $\hat{A}=\lambda+t_{2\rho^\vee}w_0B$.
	\end{lem}
	\begin{proof}
		For all $\mu\in\Pi$, we have that $w_0(\mu-\rho^\vee)\in\Pi$. Thus, $w_0(-\rho^\vee+B)\subset\Pi$ for any alcove $B\subset\Pi$, so if we set $C:=\rho^\vee+w_0(-\rho^\vee+B)$, we get $\check{C}=B$ and so 
		$$\hat{B}=C=2\rho^\vee+w_0B=t_{2\rho^\vee}w_0B.$$
		The general result follows easily, since any alcove $A$ can be written uniquely as  $A=\lambda+B$, with $\lambda\in\mathbb{X}^\vee$, $B\subset\Pi$, and $\hat{A}=\lambda+\hat{B}$.
	\end{proof}
	\begin{prop}\label{decalage rho}
		Let $A\in\mathcal{A}^+$. Then $\hat{A}\subset \rho^\vee+\mathscr{C}_0^+$.
	\end{prop}
	\begin{proof}
		Let $A\in\mathcal{A}^+$, and write $A=\lambda+B$, for $\lambda\in\mathbb{X}^\vee$, $B\subset\Pi$. Since $A$ is in $\mathcal{A}^+$, we must have $\lambda\in\overline{\mathscr{C}^+_{0}}\cap\mathbb{X}^\vee$. Indeed, if $\mu\in\Pi$ and $\lambda\in\mathbb{X}^\vee\backslash\overline{\mathscr{C}^+_{0}}$, then one sees using (\ref{box}) that $\lambda+\mu\notin\mathscr{C}^+_0$.
		
		Thus, it is enough to prove the proposition for $A\subset \Pi$. But in this case, we have by lemma \ref{lemhat}
		$$\hat{A}=t_{2\rho^\vee}w_0A=\rho^\vee+w_0(-\rho^\vee+A).$$
		Since $w_0(-\rho^\vee+A)\subset\Pi$, this concludes the proof.
	\end{proof}
	
	We also need to define a bijection $\mathcal{A}_\mathbf{g}\to \mathcal{A}_\mathbf{g},~\mathbf{h}\mapsto\hat{\mathbf{h}}$ for an arbitrary facet $\mathbf{g}\subset \overline{\mathbf{a}_1}$. One can do it as follows: let $\mathbf{h}\in \mathcal{A}_\mathbf{g}$ and $A\in\mathcal{A}$ be such that $\mathbf{h}\subset\overline{A}$ and $A\preceq Aw$ for all $w\in W_\mathbf{g}$ (such an alcove exists and is unique, thanks to the existence and unicity of a minimal element in  a coset of a parabolic subgroup in a Coxeter group), then we define $\hat{\mathbf{h}}$ as the element of $\mathcal{A}_\mathbf{g}$ which sits inside the closure of $\hat{A}$.
	\begin{coro}\label{shifting facets}
		For any $\mathbf{h}\in\mathcal{A}_\mathbf{g}^+$, we have $\hat{\mathbf{h}}\subset \rho^\vee+\overline{\mathscr{C}^+_0}$.
	\end{coro}
	\begin{proof}
		By definition, $\mathbf{h}$ sits inside the closure of some alcove $A$. Notice that, since $A$ contains $\mathbf{h}$ in its closure, we must have $A\in\mathcal{A}^+$ thanks to Proposition \ref{dominant facet}. Thus $\hat{A}\subset \rho^\vee+\mathscr{C}^+_0$ by Proposition \ref{decalage rho}, so $\hat{\mathbf{h}}\subset \rho^\vee+\overline{\mathscr{C}^+_0}$.
	\end{proof}
	Corollary \ref{shifting facets} implies in particular that the operation $\mathbf{h}\mapsto\hat{\mathbf{h}}$ preserves $\mathcal{A}^+_\mathbf{g}$.
	\begin{prop}\label{hat transf}
		Let $A\in\mathcal{A}^+$ be such that $Ar\geq A$  and $Ar\in\mathcal{A}^+$ for all $r\in W_\mathbf{g}$. Then, for all $r\in W_\mathbf{g}$, we have $\hat{A}r\leq\hat{A}$ and $\hat{A}r\in\mathcal{A}^+$.
	\end{prop}
	\begin{proof}
		Let $\mathbf{v}$ be the special point such that $\hat{A}\in\Pi_\mathbf{v}$. First notice that $\hat{A}$ is included in $\rho^\vee+\mathscr{C}^+_0$ thanks to Proposition \ref{decalage rho}. Thus, if we denote by $\mathbf{h}$ the facet of $\mathcal{A}_\mathbf{g}$ included in the closure of $\hat{A}$, we see that $\mathbf{h}\subset \rho^\vee+\overline{\mathscr{C}^+_0}$ and that $\mathbf{h}$ is included in the closure of all the alcoves of $\hat{A}W_\mathbf{g}$, so that they all belong to $\mathcal{A}^+$ thanks to Proposition \ref{dominant facet}. Since $A=w_\mathbf{v}\hat{A}$, we can apply Proposition \ref{reverse order} to get that $\hat{A}s<\hat{A}$ for all $s\in S_\mathbf{g}$. From the general theory of maximal elements in cosets of parabolic subgroups in Coxeter groups, this implies that $\hat{A}w\leq\hat{A}$ for all $w\in W_\mathbf{g}$.
	\end{proof}
	
	The anti-spherical $\ell$-Kazhdan-Lusztig polynomial ${^\ell n}_{x,y}$ will also be denoted by ${^\ell n}_{B,A}$ for the alcoves $B,A$ corresponding to $x,y\in {_\mathrm{f}W}$. The following result will be of great importance for us.
	\begin{lem}\label{lem askl1}
		For any dominant alcove $A\in\mathcal{A}^+$, we have $n_{A,\hat{A}}(1)=1$.
	\end{lem}
	\begin{proof}
		This is an immediate consequence of \cite[Theorem 5.1]{Soergel1997KazhdanLusztigPA} (see also \cite[§7]{Soergel1997KazhdanLusztigPA}).
	\end{proof}
	
	\begin{lem}\label{lem hat}
		Let $A\in\mathcal{A}^+$ and $\mathbf{g}\subset\overline{\mathbf{a}_1}$ be a facet such that $Ar\geq A$ and $Ar\in\mathcal{A}^+$ for all $r\in W_\mathbf{g}$. Then we have $n_{Ar,\hat{A}}(1)=1$ for all $r\in W_\mathbf{g}$.
	\end{lem}
	
	\begin{proof}
		
		By Lemma \ref{lem askl1}, we have that $n_{A,\hat{A}}(1)=1$ and, by Proposition \ref{prop max coset} (applied to the elements $w\in {_\mathrm{f}W}^\mathbf{g}$, $w'\in {_\mathrm{f}W}$ such that $w'\mathbf{a}_\ell=A$ and $w\mathbf{a}_\ell=\hat{A}$) combined with Proposition \ref{hat transf}, we deduce that $n_{Ar,\hat{A}}(1)=1$ for all $r\in W_\mathbf{g}$.
	\end{proof}
	The following Proposition will allow us to only consider facets which are far inside of the dominant cone.
	\begin{prop}\label{away}
		Let $\mathbf{g}\subset \overline{\mathbf{a}_1}$ be a facet  and $\mathbf{h}\in\mathcal{A}^+_\mathbf{g}$. Denote by $u$, resp. $\hat{u}$, the elements of ${_\mathrm{f}W}^\mathbf{g}$ corresponding to $\mathbf{h}$, resp. $\hat{\mathbf{h}}$ (i.e. such that $u\mathbf{g}=\mathbf{h}$, resp. $\hat{u}\mathbf{g}=\hat{\mathbf{h}}$). Then we have $n_{u,\hat{u}}(1)\neq 0$, so that
		$$\mathrm{Hom}_{D^b_{\mathcal{IW}}(\mathrm{Fl}^{\circ}_{\mathbf{g}},\mathbf{k})}^\bullet(\mathcal{E}_{u}^{\mathbf{g}},\mathcal{E}_{\hat{u}}^{\mathbf{g}})\neq 0.$$
	\end{prop}
	\begin{proof}
		
		Recall that $w_\mathbf{g}$ is the element of maximal length in $W_\mathbf{g}$, and set $A:=uw_\mathbf{g}\mathbf{a}_\ell$. Notice that $A\in\mathcal{A}^+$ thanks to Proposition \ref{prop min coset} and, by construction, $A$ is the alcove such that $\mathbf{h}\subset \overline{A}$, $Ar\in\mathcal{A}^+$ and $A\leq Ar$ for every $r\in W_\mathbf{g}$ (in particular $Ar\preceq A$ for all $r\in W_\mathbf{g}$), so that $\hat{\mathbf{h}}\subset\overline{\hat{A}}$. By Lemma \ref{lem hat}, we have that $n_{Aw_\mathbf{g},\hat{A}}(1)=1\neq 0$, and $\hat{A}w\leq \hat{A}$ for every $w\in W_\mathbf{g}$ thanks to Proposition \ref{hat transf}. In particular we deduce that $\hat{u}\mathbf{a}_\ell=\hat{A}$. Thanks to Proposition \ref{anti-spherical relations}, this implies that $$\mathrm{Hom}_{D^b_{\mathcal{IW}}(\mathrm{Fl}^{\circ}_{\mathbf{a}_1},\mathbf{k})}^\bullet(\mathcal{E}_{u},\mathcal{E}_{\hat{u}})\neq 0.$$  
		By Proposition \ref{proj form}, we get the desired result.
	\end{proof}
	
	\subsection{Facets which are not points}\label{Facets which are not points} The root system $\mathfrak{R}^\vee$ decomposes uniquely into a disjoint union $\mathfrak{R}^\vee=\mathfrak{R}^\vee_1\sqcup \cdots\sqcup\mathfrak{R}^\vee_t$ of irreducible root systems. This decomposition comes with a decomposition of the affine Weyl group $W= W_1\times\cdots \times W_t$ (where each $W_i$ is the affine Weyl group associated with $\mathfrak{R}_i^\vee$) and of the affine space $E= E_1\times \cdots\times E_t$ (where each $E_i$ is associated with $\mathfrak{R}_i^\vee$). Note that one also has dually a decomposition $\mathfrak{R}=\mathfrak{R}_1\sqcup \cdots\sqcup\mathfrak{R}_t$ into irreducible root systems, and an isomorphism $G_1\times\cdots\times G_r\xrightarrow{\sim} G$ induced by multiplication from the product of the minimal closed connected normal subgroups of $G$ of positive dimension over $\mathbb{F}$, each $G_i$ admitting $\mathfrak{R}_i$ as a root system (cf. \cite[Theorem and Corollary 27.5]{humphreys1975linear}). The point for us is that, for every $i$, the alcoves determined by the box action of $W_i$ on $E_i$ are open simplices, and that each alcove $A\in\mathcal{A}$ is a product of such alcoves $A_1\times\cdots\times A_t$. More generally, each facet $\mathbf{h}\subset \overline{A}$ decomposes into a product of facets $\mathbf{h}_1\times\cdots\times\mathbf{h}_t$, with $\mathbf{h}_i\subset\overline{A_i}$ for every $i$.
	
	In this subsection we prove that, for any facet $\mathbf{g}\subset\overline{\mathbf{a}_1}$ such that each $\mathbf{g}_i$ is not a point (where $\mathbf{g}=\mathbf{g}_1\times\cdots\times\mathbf{g}_t$), the set ${_\mathrm{f}W}^\mathbf{g}$ is single equivalence class. We will prove in the next subsection that ${_\mathrm{f}W}^\mathbf{g}$ is a single equivalence class exactly when each $\mathbf{g}_i$ is a non-special facet. Let us first explain how we can reduce ourselves to the case where the root system is irreducible; we start with a technical lemma, where $\mathbf{k}$ could be taken to be any field.  
	\begin{lem}\label{alg non com}
		Let $A$ be a local $\mathbf{k}$-algebra and $B$ be a connected $\mathbf{k}$-algebra (we do not assume $A$ and $B$ to be commutative). Assume that there exists a $\mathbf {k}$-algebra morphism $A\to \mathbf{k}$ and that $B$ is a finite dimensional $\mathbf{k}$-vector space. Then the ring $A\otimes_\mathbf{k}B$ is connected. 
	\end{lem}
	\begin{proof}
		Denote by $\mathfrak{m}$ the maximal ideal of $A$. Since we have a morphism $A\to \mathbf{k}$, we get an injection $A/\mathfrak{m}\hookrightarrow \mathbf{k}$, which is in fact an isomorphism $A/\mathfrak{m}\simeq \mathbf{k}$ as $A/\mathfrak{m}$ is a $\mathbf{k}$-vector space. Therefore we get an isomorphism 
		$$A/\mathfrak{m}\otimes_\mathbf{k}B\simeq B, $$
		from which we deduce that $A/\mathfrak{m}\otimes_\mathbf{k}B$ is a connected ring.

Thanks to Nakayama's lemma (cf. \cite[§2, Proposition 4.2.3]{knus2012quadratic}), a morphism of finitely generated $A$-modules $f:M\to N$ is surjective if and only if $$f\otimes \mathrm{id}_{A/\mathfrak{m}}:M/\mathfrak{m}M\to N/\mathfrak{m}N$$ is surjective. Since $A\otimes_\mathbf{k}B$ is finitely generated over $A$ (because $\mathrm{dim}_\mathbf{k}B<\infty$),  we can apply this last result to any endomorphism of the finitely generated $A$-module $A\otimes_\mathbf{k}B$, and deduce that an element $e\in A\otimes_\mathbf{k}B$ is invertible if and only if its image  $\overline{e}$ under the canonical projection $A\otimes_\mathbf{k}B\to A/\mathfrak{m}\otimes_\mathbf{k}B$ is invertible.

		Let $e\in A\otimes_\mathbf{k}B$ be an idempotent, i.e. an element satisfying $e(e-1)=0$. Then $\overline{e}\in A/\mathfrak{m}\otimes_\mathbf{k}B$ is an idempotent, from which we deduce that $\overline{e}\in\{1,0\}$ by connectedness of $A/\mathfrak{m}\otimes_\mathbf{k}B$. If $\overline{e}=1$, then $e\in (A\otimes_\mathbf{k}B)^\times$ thanks to the previous paragraph, so $e(e-1)=0\Rightarrow e=1$. If $\overline{e}=0$, then $$\overline{e-1}=\overline{e}-1=-1\in(A/\mathfrak{m}\otimes_\mathbf{k}B)^\times,$$
		from which we deduce once again that $e-1\in (A\otimes_\mathbf{k}B)^\times$ and thus $e=0$. Therefore any idempotent of $A\otimes_\mathbf{k}B$ is trivial, which means that this ring is connected. \end{proof}
	For the next result, notice that the decomposition $W= W_1\times\cdots \times W_t$ carries on at the level of parabolic subgroups, so that we have equalities ${W}^\mathbf{g}={W_1}^{\mathbf{g}_1}\times\cdots\times{W_t}^{\mathbf{g}_t}$ and ${_\mathrm{f}W}^\mathbf{g}={_\mathrm{f}W_1}^{\mathbf{g}_1}\times\cdots\times{_\mathrm{f}W_t}^{\mathbf{g}_t}$.
	\begin{prop}\label{reduction to irred}
		Let $\mathbf{g}=\mathbf{g}_1\times\cdots\times\mathbf{g}_t$ be a facet included in $\overline{\mathbf{a}_1}$, and $w=(w_1,\cdots,w_t)$, $w'=(w'_1,\cdots,w'_t)$ be elements of ${_\mathrm{f}}W^\mathbf{g}$. Then we have 
		\begin{equation}\label{351}
			w\mathscr{R}_\mathbf{g} w'\Longleftrightarrow w_i\mathscr{R}_{\mathbf{g}_i}w'_i~\forall i. 
		\end{equation}
		Moreover there is an equality
		\begin{equation}\label{352}
			{^\ell n}_{w,w'}(1)={^\ell n}_{w_1,w'_1}(1)\times\cdots\times {^\ell n}_{w_r,w'_r}(1).
		\end{equation}
	\end{prop}
	\begin{proof}
		 The isomorphism $G_1\times\cdots\times G_r\xrightarrow{\sim} G$ induces an  isomorphism from the product of the partial affine flag varieties $\mathrm{Fl}^{\circ}_{\mathbf{g}_i}$, each being relative to $G_i$, to the partial affine flag variety $\mathrm{Fl}^{\circ}_{\mathbf{g}}$ relative to $G$. This morphism is equivariant for the action of the product of the Iwahori subgroups $\mathrm{Iw}^+_1\times\cdots\times \mathrm{Iw}^+_t$ on the left (where $\mathrm{Iw}^+_i=L^+G_i\cap \mathrm{Iw}^+$ for each $i$) and of $\mathrm{Iw}^+$ on the right. Thus, external tensor product yields a functor
		\begin{equation*}
			F:\prod _{i=1}^tD^b_{\mathcal{IW}}(\mathrm{Fl}^{\circ}_{\mathbf{g}_i},\mathbf{k})  \to D^b_{\mathcal{IW}}(\mathrm{Fl}^{\circ}_{\mathbf{g}},\mathbf{k}),
		\end{equation*}
		and one can easily check with the Künneth formula that $F$ sends tuples of parity complexes to parity complexes. 
		
		Next, we claim that we have an isomorphism 
		\begin{equation*}
			F((\mathcal{E}_{w_1}^{\mathbf{g}_1},\cdots,\mathcal{E}_{w_t}^{\mathbf{g}_t}))\simeq \mathcal{E}_w^\mathbf{g}.
		\end{equation*}
		The two objects above are parity complexes and coincide when restricted to the stratum $\mathscr{X}^\mathbf{g}_w$, so we only need to check that the object $F((\mathcal{E}_{w_1}^{\mathbf{g}_1},\cdots,\mathcal{E}_{w_t}^{\mathbf{g}_t}))$ is indecomposable, which we will do by showing that its endomorphism ring is connected. 
		
		We claim that we have the following isomorphism of graded $\mathbf{k}$-vector spaces
		\begin{equation}
			\mathrm{Hom}_{D^b_{\mathcal{IW}}(\mathrm{Fl}^{\circ}_{\mathbf{g}},\mathbf{k}) }^\bullet(\mathcal{E}_{w_1}^{\mathbf{g}_1}\boxtimes\cdots \boxtimes\mathcal{E}_{w_t}^{\mathbf{g}_t},\mathcal{E}_{w'_1}^{\mathbf{g}_1}\boxtimes\cdots \boxtimes\mathcal{E}_{w'_t}^{\mathbf{g}_t})\simeq \bigotimes_{i=1}^t\mathrm{Hom}_{D^b_{\mathcal{IW}}(\mathrm{Fl}^{\circ}_{\mathbf{g}_i},\mathbf{k})}^\bullet(\mathcal{E}_{w_i}^{\mathbf{g}_i},\mathcal{E}_{w'_i}^{\mathbf{g}_i}).\label{eq locale 2}
		\end{equation}
		Let us briefly explain how to obtain (\ref{eq locale 2}). Thanks to \cite[Proposition 1.4.6]{achar2021perverse} (the result is stated with sheaves for the analytic topology there, but is easily translated for the \'etale topology), we know that for every $\mathcal{F},\mathcal{G}\in D^b_{\mathcal{IW}}(\mathrm{Fl}^{\circ}_{\mathbf{g}},\mathbf{k})$ we have a natural isomorphism of graded vector spaces
		\begin{equation}\label{Rhom}
			\mathrm{H}^\bullet(R\Gamma(R\underline{\mathrm{Hom}}(\mathcal{F},\mathcal{G})))\simeq \mathrm{H}^\bullet(R\mathrm{Hom}(\mathcal{F},\mathcal{G}))\simeq \mathrm{Hom}_{D^b_{\mathcal{IW}}(\mathrm{Fl}^{\circ}_{\mathbf{g}},\mathbf{k})}^\bullet(\mathcal{F},\mathcal{G}), 
		\end{equation}
		where $\underline{\mathrm{Hom}}$ is the internal hom functor, and $\mathrm{H}^\bullet(-)$ means the direct sum of cohomology groups. Thus, we only need to show that there exists an isomorphism $$R\underline{\mathrm{Hom}}(\mathcal{F}_1\boxtimes\cdots \boxtimes\mathcal{F}_t,\mathcal{G}_1\boxtimes\cdots \boxtimes\mathcal{G}_t)\simeq \mathlarger{\mathlarger{\boxtimes}}_{i=1}^tR\underline{\mathrm{Hom}}(\mathcal{F}_i,\mathcal{G}_i)$$
		for every $\mathcal{F}_i,\mathcal{G}_i\in D^b_{\mathcal{IW}}(\mathrm{Fl}^{\circ}_{\mathbf{g}},\mathbf{k})$. This fact is proved in \cite[§4.2.7(b)]{beilinson2018faisceaux} for constructible complexes. 
		
		Next, notice that the $\mathbf{k}$-algebra $\mathrm{Hom}_{D^b_{\mathcal{IW}}(\mathrm{Fl}^{\circ}_{\mathbf{g}_i},\mathbf{k})}(\mathcal{E}_{w_i}^{\mathbf{g}_i},\mathcal{E}_{w'_i}^{\mathbf{g}_i})$ is local for each $i$ (because $\mathcal{E}_{w_i}^{\mathbf{g}_i}$ is indecomposable), so that $A_i:=\mathrm{Hom}^\bullet_{D^b_{\mathcal{IW}}(\mathrm{Fl}^{\circ}_{\mathbf{g}_i},\mathbf{k})}(\mathcal{E}_{w_i}^{\mathbf{g}_i},\mathcal{E}_{w'_i}^{\mathbf{g}_i})$ is also a local $\mathbf{k}$-algebra as its degree zero part is local (here we apply \cite[Theorem 3.1]{gordon1982graded}); moreover, restriction to the stratum $\mathscr{X}^{\mathbf{g}_i}_{w_i}$ yields a $\mathbf{k}$-algebra morphism from $A_i$ to $\mathbf{k}$. Therefore we can apply Lemma \ref{alg non com} (recall that a local ring is connected) to deduce that the ring on the right-hand side of (\ref{eq locale 2}) is connected. Hence we deduce that the degree zero part of the left-hand side of (\ref{eq locale 2}) is connected (applying once again \cite[Theorem 3.1]{gordon1982graded}), proving the claim.
		
		Thus, (\ref{eq locale 2}) can be rewritten as  
		$$\mathrm{Hom}_{D^b_{\mathcal{IW}}(\mathrm{Fl}^{\circ}_{\mathbf{g}},\mathbf{k}) }^\bullet(\mathcal{E}_{w}^{\mathbf{g}},\mathcal{E}_{w'}^{\mathbf{g}})\simeq \bigotimes_{i=1}^t\mathrm{Hom}_{D^b_{\mathcal{IW}}(\mathrm{Fl}^{\circ}_{\mathbf{g}_i},\mathbf{k})}^\bullet(\mathcal{E}_{w_i}^{\mathbf{g}_i},\mathcal{E}_{w'_i}^{\mathbf{g}_i}). $$
		Since a finite tensor product of vector spaces is non-zero if and only if each vector space appearing in the product is non-zero, this isomorphism finishes implies the equivalence (\ref{351}).
		
		Now we prove (\ref{352}). First notice that, by compatibility of the $*$-pushforward with the external tensor product (cf. \cite[§4.2.7(a)]{beilinson2018faisceaux}), we get an isomorphism $\nabla_{w'}^{\mathbf{g}}\simeq \nabla_{w'_1}^{\mathbf{g}_1}\boxtimes\cdots\boxtimes\nabla_{w'_t}^{\mathbf{g}_1}$. Next, another use of (\ref{Rhom}) yields  the second isomorphism below:
		\begin{align*}
			\mathrm{Hom}_{D^b_{\mathcal{IW}}(\mathrm{Fl}^{\circ}_{\mathbf{g}},\mathbf{k}) }^\bullet(\mathcal{E}_{w}^{\mathbf{g}},\nabla_{w'}^\mathbf{g})&\simeq\mathrm{Hom}_{D^b_{\mathcal{IW}}(\mathrm{Fl}^{\circ}_{\mathbf{g}},\mathbf{k}) }^\bullet(\mathcal{E}_{w_1}^{\mathbf{g}_1}\boxtimes\cdots \boxtimes\mathcal{E}_{w_t}^{\mathbf{g}_t},\nabla_{w'_1}^{\mathbf{g}_1}\boxtimes\cdots\boxtimes\nabla_{w'_t}^{\mathbf{g}_1})\\&\simeq \bigotimes_{i=1}^t\mathrm{Hom}_{D^b_{\mathcal{IW}}(\mathrm{Fl}^{\circ}_{\mathbf{g}_i},\mathbf{k})}^\bullet(\mathcal{E}_{w_i}^{\mathbf{g}_i},\nabla_{w'_i}^{\mathbf{g}_i}). 
		\end{align*}
		Taking the dimensions in this isomorphism yields the desired equality.
	\end{proof}
	Now that we understand how to reduce the study to the case of an irreducible root system thanks to (\ref{351}), we treat this case in detail. We start with the following easy lemma (which does not require $\mathfrak{R}^\vee$ to be irreducible).
	\begin{lem}\label{descent}	
		Let $A$ be an alcove contained in  $\rho^\vee+\mathscr{C}^+_0$. Then there exist alcoves $$A_r,A_{r-1},\cdots,A_1,A_0$$ contained in $\rho^\vee+\mathscr{C}^+_0$ such that 
		$$\rho^\vee+\mathbf{a}_1=A_r\leq A_{r-1}\leq\cdots\leq A_1\leq A_0=A$$
		and $A_i$ is obtained by reflecting $A_{i-1}$ along one of its walls for each $i\in\llbracket1,r-1\rrbracket$.
	\end{lem}
	\begin{proof}
	 Since the order $\leq$ is invariant under translation by $\rho^\vee$ in $\mathcal{A}^+$ (because it coincides with the periodic order $\preceq$), it is equivalent to prove that for any alcove $A$ contained in $\mathscr{C}^+_0$, there exist alcoves $(A_i)_{0\leq i\leq r}$ contained in $\mathscr{C}^+_0$ such that $(A_r,A_0)=(\mathbf{a}_1,A)$, $A_{i}\leq A_{i-1}$ and $A_i$ is obtained by reflecting $A_{i-1}$ along one of its walls. This last condition is equivalent to requiring that $A_i=A_{i-1}s$ for some simple reflection $s$. Now if we let $w$ be the element of ${_\mathrm{f}W}$ such that $w\mathbf{a}_1=A$, we know by the proof of Lemma \ref{lem reg bl} that there exists some simple reflection $s$ such that $w_1:=ws<w$, with $w_1\in {_\mathrm{f}W}$, so that we can put $A_1:=w_1\mathbf{a}_1$. We then conclude by induction.
	\end{proof}
	
	See Figure \ref{Affine hyperplanes C2} for an illustration of Lemma \ref{descent} (but notice that one needs to dilate the affine plane on the picture by $\ell^{-1}$ to find back the same setting).
	\begin{lem}\label{desending facets}
		Let $\mathbf{h}\in\mathcal{A}_\mathbf{g}$. Pick a facet $\mathbf{p}\subset\overline{\mathbf{h}}$ and consider $\mathbf{h}':=v\mathbf{h}$ with $v\in W_\mathbf{p}$. If $\mathbf{p}\subset \mathscr{C}^+_0$, then $\mathbf{h}$ and $\mathbf{h}'$ belong to $\mathcal{A}^+_\mathbf{g}$, and we have $w\sim_{\mathbf{g}}w'$, where $w,w'\in {_{\mathrm{f}}W}^\mathbf{g}$ are such that $\mathbf{h}=w\mathbf{g}$ and $\mathbf{h}'=w'\mathbf{g}$. 
	\end{lem}
	\begin{proof}
		Since $\mathbf{p}$ is inside $\mathscr{C}^+_0$, the same is true for all the facets $v\mathbf{h}$, with $v\in W_\mathbf{p}$, since those contain $\mathbf{p}$ in their closure. In particular $\mathbf{h}$ and $\mathbf{h}'$ belong to $\mathcal{A}^+_\mathbf{g}$ thanks to Proposition \ref{dominant facet}.
		
		Let $w\in {_{\mathrm{f}}W}^\mathbf{g}$ be such that $w\mathbf{g}=\mathbf{h}$. By construction, there exists a facet $\mathbf{q}\subset\overline{\mathbf{g}}$ such that $w\mathbf{q}=\mathbf{p}$. In particular, we get that $W_\mathbf{p}=wW_\mathbf{q}w^{-1}$ and, for any $v\in  W_\mathbf{p}$, we have
		$$v\mathbf{h}=vw\mathbf{g}=wr\mathbf{g},$$
		where $r\in W_\mathbf{q}$ is such that $wrw^{-1}=v$. Let $u\in W_\mathbf{g}$ be such that $wru$ is maximal in $wrW_\mathbf{g}$. We still have $wru\mathbf{g}=\mathbf{h}'$, and since $\mathbf{h}'\in\mathcal{A}^+_\mathbf{g}$, we must have $wru\in  {_{\mathrm{f}}W}^\mathbf{g}$, so $w'=wru$, with $ru\in W_\mathbf{q}$. Thanks to the second point of Proposition \ref{prop max coset}, we conclude that $w\sim_{\mathbf{g}}w'$ (more precisely, if we let $w''\in {_{\mathrm{f}}W}^\mathbf{g}$ be the maximal element in $wW_\mathbf{q}$ (see Figure \ref{Lemma C2} below), then we have $n_{w,w''}(1)\neq 0$ and $n_{w',w''}(1)\neq0$).
	\end{proof}
	  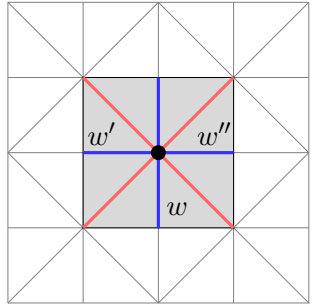
\begin{figure}[!h]	\begin{tikzpicture}[>=stealth]
		\foreach \x in {1,...,5}{
			\draw[ultra thin, color=gray] (\x,1)--(\x,5);
		}
		\foreach \x in {1,...,5}{
			\draw[ultra thin, color=gray] (1,\x)--(5,\x);
		}
		
		\foreach \x in {1,...,3}{
			\draw[ultra thin, color=gray] (1,2*\x-1)--(2*\x-1,1);
			\draw[ultra thin, color=gray] (1,2*\x-1)--(7-2*\x,5);
		}
		\draw[ultra thin, color=gray] (3,1)--(5,3);
		\draw[ultra thin, color=gray] (5,1)--(5,1);
		\draw[ultra thin, color=gray] (3,5)--(5,3);
		\draw[very thin, fill=gray!30] (3,3)--(4,2)--(3,2)--(3,3);
		\draw[very thin, fill=gray!30] (3,3)--(3,2)--(2,2)--(3,3);
		\draw[very thin, fill=gray!30] (3,3)--(2,2)--(2,3)--(3,3);
		\draw[very thin, fill=gray!30] (3,3)--(2,3)--(2,4)--(3,3);
		\draw[very thin, fill=gray!30] (3,3)--(2,4)--(3,4)--(3,3);
		\draw[very thin, fill=gray!30] (3,3)--(3,4)--(4,4)--(3,3);
		\draw[very thin, fill=gray!30] (3,3)--(4,4)--(4,3)--(3,3);
		\draw[very thin, fill=gray!30] (3,3)--(4,3)--(4,2)--(3,3);
	   \draw[very thick, color=red!60, fill opacity=0] (3,3)--(4,2);
	   \draw[very thick, color=red!60, fill opacity=0] (3,3)--(4,4);
	   \draw[very thick, color=red!60, fill opacity=0] (3,3)--(2,2);
	    \draw[very thick, color=red!60, fill opacity=0] (3,3)--(2,4);
	    \draw[very thick, color=blue!80, fill opacity=0] (3,3)--(3,2);
	    \draw[very thick, color=blue!80, fill opacity=0] (3,3)--(2,3);
	    \draw[very thick, color=blue!80, fill opacity=0] (3,3)--(3,4);
	    \draw[very thick, color=blue!80, fill opacity=0] (3,3)--(4,3);
       \draw (3,3) node[circle,inner sep=2, fill=black] {};
       \draw (3.25,2.25) node[scale=1,rotate=0]{$w$};
       \draw (2.25,3.25) node[scale=1,rotate=0]{$w'$};
       \draw (3.75,3.25) node[scale=1,rotate=0]{$w''$};

	\end{tikzpicture}
	\caption{\textit{Lemma \ref{desending facets} in type $C_2$}. The black dot is $\mathbf{p}$, while the $W_\mathbf{p}$-conjugates of $\mathbf{h}$ are represented by blue lines, and the red lines are the other walls fixed by $W_\mathbf{p}$.}
	\label{Lemma C2}
\end{figure}

	We have now enough ingredients to conclude.
	\begin{thm}\label{thm not points}
		Assume that $\mathfrak{R}^\vee$ is irreducible and let $\mathbf{g}\subset\overline{\mathbf{a}_1}$ be a facet which is not a point. Then ${_\mathrm{f}W}^\mathbf{g}$ consists of a single class for the equivalence relation $\sim_{\mathbf{g}}$.
	\end{thm}
	\begin{proof}We introduce a new notation for this proof: if $\mathbf{h},\mathbf{h}'\in \mathcal{A}^+_\mathbf{g}$ are such that $$\mathbf{h}=w\mathbf{g},~\mathbf{h}=w'\mathbf{g}$$ for some $w,w'\in {_{\mathrm{f}}W}^\mathbf{g}$, then we write $\mathbf{h}\sim_{\mathbf{g}}\mathbf{h}'$ if $w\sim_{\mathbf{g}}w'$. We want to show that $\mathcal{A}^+_\mathbf{g}$ consists of a single equivalence class for $\sim_{\mathbf{g}}$.
		
		By Proposition \ref{away}, we have $\mathbf{h}\sim_{\mathbf{g}}\hat{\mathbf{h}}$ for any $\mathbf{h}\in\mathcal{A}^+_\mathbf{g}$ and, thanks to Corollary \ref{shifting facets}, we know that $\hat{\mathbf{h}}\subset \rho^\vee+ \overline{\mathscr{C}^+_0}$ . Thus, it is enough to show that all the elements of $\mathcal{A}^+_\mathbf{g}$ lying inside of $\rho^\vee+ \overline{\mathscr{C}^+_0}$ are in relation. For this, we will show that any facet satisfying this condition is in relation with the unique representative of its $W$-orbit in $\rho^\vee+\overline{\mathbf{a}_\ell}$.  	
		
		Pick $\mathbf{h}$ inside of $\rho^\vee+\overline{\mathscr{C}^+_0}$ and let $A\subset \rho^\vee+\mathscr{C}^+_0$ be an alcove containing $\mathbf{h}$ in its closure. Let also $A_r,A_{r-1},\cdots,A_1,A_0:=A$ be alcoves as in Lemma \ref{descent}, and denote by $s_i\in S$ the reflection such that $A_is_i=A_{i-1}$. 
		
		Since $\mathbf{g}$ is not a point and $\overline{\mathbf{a}_1}$ is a simplex (because $\mathfrak{R}^\vee$ is indecomposable), the face of the simplex fixed by any $s\in S$ (which is the closure of the facet fixed by $s$) must have a nonzero intersection\footnote{Recall that, for a positive integer $n$, an $n$-simplex $\Delta$ inside of an affine space $E$ of dimension $n$ is defined as the convex hull of $n+1$ vertices (those vertices being $n+1$ points not lying on a same hyperplane). For $0\leq k\leq n$, a $k$-face of $\Delta$ is a subset consisting of the convex hull of $k+1$ vertices of $\Delta$. From these definitions, it is straightforward to show that an $(n-1)$-face of $\Delta$ has a non-empty intersection with any $k$-face, for $k\geq 1$.} with $\overline{\mathbf{g}}$ inside $\overline{\mathbf{a}_1}$. This means that for any $s\in S$, there exists a facet $\mathbf{q}\subset\overline{\mathbf{g}}$ such that $s\in W_\mathbf{q}$. 	
		
		For $i\in\llbracket 1,r\rrbracket$, denote by $\mathbf{q}_i$ a facet such that $\mathbf{q}_i\subset\overline{\mathbf{g}}$ and $s_i\in W_{\mathbf{q}_i}$; for $i\in\{0,\cdots,r\}$, denote by $w_i$ the element of $_{\mathrm{f}}W$ such that $w_i\mathbf{a}_1=A_i$ and put $\mathbf{h}_i:=w_i\mathbf{g}$, $\mathbf{p}_i:=w_i\mathbf{q}_i$. Notice that $A_is_i=w_is_iw_i^{-1}A_i$ and that $\mathbf{h}_i\subset \overline{A_i}$.  By construction, we have the following data
		\begin{align*}
			&\mathbf{p}_i\subset\overline{\mathbf{h}_i},\quad \text{with}~ \mathbf{h}_0=\mathbf{h}~\text{and}~\mathbf{h}_r\subset\rho^\vee+\overline{\mathbf{a}_1},\\
			&\mathbf{h}_{i-1}=w_is_iw_i^{-1} \mathbf{h}_{i}~\forall i\in\llbracket 1,r\rrbracket\quad\text{with}~w_is_iw_i^{-1}\in W_{\mathbf{p}_i}.
		\end{align*} 
		Since the alcoves $A_i$ are inside of $\rho^\vee+ \mathscr{C}^+_0$, the facets $\mathbf{p}_i$ are inside of $\mathscr{C}^+_0$ for all $i$. Thus we can apply Lemma \ref{desending facets} to get that $\mathbf{h}_{i}\sim_{\mathbf{g}}\mathbf{h}_{i-1}$ for every $i$, and by transitivity $\mathbf{h}_r\sim_{\mathbf{g}}\mathbf{h}$. This concludes the proof.
	\end{proof}
	
	\subsection{Non-special facets which are points}\label{Facets which are points}
	In order to get the statement of Theorem \ref{thm not points} for non-special facets which are points, we need to be able to say more about the anti-spherical Kazhdan-Lusztig polynomials $n_{B,A}$. For that, we will make use of the combinatorial data linking the ``periodic" polynomial $p_{B,A}$ of \cite[Remark 4.4]{Soergel1997KazhdanLusztigPA} (which is very closely related to Lusztig's polynomial $Q_{B,A}$ from \cite{LUSZTIG1980121}) and the anti-spherical polynomials. 
	
	Recall the partial order $\preceq$ introduced in subsection \ref{geometry}. The following lemma will be helpful throughout all of this subsection.
	\begin{lem}\label{lem max coset}
		Let $\mathbf{v}$ be a special facet and $C$ be an alcove contained in $\mathscr{C}^+_{\mathbf{v}}$. Then $C$ is maximal in $W_\mathbf{v}C$ for $\preceq$.
	\end{lem}
	\begin{proof}
		If $D\in W_\mathbf{v}C$ is different from $C$, then there exists a hyperplane $H$ passing through $\mathbf{v}$ which separates $D$ from $\mathscr{C}^+_\mathbf{v}$ (one may take a well chosen wall of $\mathscr{C}^+_\mathbf{v}$), so that $D\subset E^-_H$ and $D\prec s_HD$, with $s_HD\in W_\mathbf{v}D$. Assume that $s_HD\neq C$, then $s_HD$ is not included in $\mathscr{C}^+_{\mathbf{v}}$ (because $C$ is the only alcove of $W_\mathbf{v}D$ contained in $\mathscr{C}^+_{\mathbf{v}}$), so that we can once again find some alcove $D_2\in W_\mathbf{v}D$ satisfying $s_HD\prec D_2$. Since $W_\mathbf{v}D$ is a finite set, we can repeat this process a finite number of times until finding an alcove $D_n\in W_\mathbf{v}D$ which is contained in $\mathscr{C}^+_{\mathbf{v}}$ and such that $D\prec D_n$, so that $D_n=C$. Therefore $C$ is the maximal element in $W_\mathbf{v}C$ for $\preceq$.
	\end{proof}
	We now recall the definition of the ``periodic'' module $\mathcal{P}$, which is the free left $\mathbb{Z}[v^{\pm1}]$-module with basis $\mathcal{A}$, equipped with a structure of right $\mathcal{H}$-module satisfying
	\begin{equation}\label{action hecke alcoves}
		\forall s\in S,~AC_s=
		\begin{cases}
			As+vA, & \text{if}\ A\prec As \\
			As+v^{-1}A, & \text{if}\ As\prec A,
		\end{cases}
	\end{equation}
	where $C_s:=H_s+v$ (these data do define a right action of $\mathcal{H}$ thanks to \cite[Lemma 4.1]{Soergel1997KazhdanLusztigPA}). One then defines the submodule $\mathcal{P}^\circ\subset \mathcal{P}$ as the right $\mathcal{H}$-submodule generated by the elements of the form
	$$E_\lambda:=\sum_{z\in W_0}v^{l(z)}(\lambda+z\mathbf{a}_1),~\lambda\in \mathbb{X}^\vee.$$
	Notice that this definition parallels Lusztig's definition of
	$$e_{\{\lambda\}}:=\sum_{A\in\mathcal{A},~\{\lambda\}\subset \overline{A}}A$$ 
	from \cite[§1.7]{LUSZTIG1980121}, because the set of alcoves $\{\lambda+z\mathbf{a}_1,~z\in W_0\}$ is equal to the set $\{A\in\mathcal{A},~\{\lambda\}\subset \overline{A}\}$.
	
	Recall that a morphism $f:M\to N$ of right $\mathcal{H}$-modules is called \textit{skew linear} if it satisfies $f(xH_w)=f(x)(H_{w^{-1}})^{-1}$, $f(xv)=f(x)v^{-1}$  for every $x\in M,~w\in W$. It can be shown (cf. \cite[Theorem 4.3]{Soergel1997KazhdanLusztigPA}) that $\mathcal{P}^\circ$ admits a unique $\mathcal{H}$-skew linear involution $\overline{(\cdot)}:\mathcal{P}^\circ\to \mathcal{P}^\circ$ such that $\overline{E_\lambda}=E_\lambda$ for all $\lambda\in\mathbb{X}^\vee$, and that for all $A\in \mathcal{A}$ there exists a unique $\underline{P}_A\in\mathcal{P}^\circ$ which is self dual with respect to this involution, with $\underline{P}_A\in A+\sum_Bv\mathbb{Z}[v]B$. The $\underline{P}_A$ form a $\mathbb{Z}[v^{\pm1}]$-basis of $\mathcal{P}^\circ$, and the periodic polynomials are defined via the formula
	$$\underline{P}_A=\sum_B p_{B,A}B.$$
	The following result follows quite directly from the constructions, but will be of great importance for us.
	\begin{lem}\label{lem inv}
		Let $A\in\mathcal{A}$ and $\mathbf{v}$ be the special facet such that $A\subset\Pi_\mathbf{v}$. Then we have 
		$$p_{wC,A}(1)=p_{C,A}(1)\qquad \text{for all}~C\in\mathcal{A}~\text{and}~w\in W_\mathbf{v}. $$
	\end{lem}
	\begin{proof}
		Denote by $\mathcal{P}_1$ the free left $\mathbb{Z}$-module $\mathbb{Z}\otimes_{\mathbb{Z}[v^{\pm1}]}\mathcal{P}$, where $\mathbb{Z}$ is seen as a  $\mathbb{Z}[v^{\pm1}]$-module through the map $\mathbb{Z}[v^{\pm1}]\to \mathbb{Z},~v\mapsto 1$, and let $\varphi:\mathcal{P}\to \mathcal{P}_1$ be the morphism of $\mathbb{Z}$-modules induced by sending $v$ to $1$. In particular we have
		$$\varphi(\underline{P}_A)=\sum_{B}p_{B,A}(1)B. $$
		The left and right actions of $W$ on $\mathcal{A}$ endow $\mathcal{P}_1$ with a structure of left and right $\mathbb{Z}[W]$-module, and one can check with (\ref{action hecke alcoves}) that $\varphi(PH_w)=\varphi(P)w$ for every $w\in W$ and $P\in\mathcal{P}$. In particular, $\varphi(PH_w)$ belongs to the right $\mathbb{Z}[W]$-submodule generated by $\varphi(P)$, which we denote by $\varphi(P)\mathbb{Z}[W]$. From \cite[Remark 4.4]{Soergel1997KazhdanLusztigPA}, we know that $p_{B,A}(1)=Q_{B,A}(1)$ for every alcoves $A,B$, where $Q_{B,A}$ is Lusztig's polynomial from \cite{LUSZTIG1980121}. But $\underline{D}_A:=\sum_{B}Q_{B,A}B$ belongs to the right $\mathcal{H}$-submodule generated by $e_\mathbf{v}$ thanks to \cite[Theorem 2.15]{LUSZTIG1980121}, from which we deduce that $\varphi(\underline{P}_A)$ belongs to $e_\mathbf{v}\mathbb{Z}[W]$. Now $e_\mathbf{v}$ is invariant under the left action of $W_\mathbf{v}$ by construction, but since the left and right actions of $W$ on $\mathcal{P}_1$ commute, we get that $wP=P$ for every $P\in e_\mathbf{v}\mathbb{Z}[W]$, concluding the proof.
	\end{proof}
	For any $x\in W$ and $A\in\mathcal{A}$, write $A=\lambda+B$ for a unique $\lambda\in\mathbb{X}^\vee$ and $B\subset\Pi$, and put $x\ast A:= x\lambda+B$. We will denote by $N_A$ the element that we denoted by $N_x$ in subsection \ref{section partial affine flag}, where $x\mathbf{a}_1=A$, and we define (following \cite[Proposition 5.2]{Soergel1997KazhdanLusztigPA}) the $\mathbb{Z}[v^{\pm1}]$-linear application $$\mathrm{res}:\mathcal{P}\to \mathcal{N}, $$
	which sends an alcove $A$ to $N_A$ if $A\in\mathcal{A}^+$, and to $0$ otherwise. Finally, for any $A\in\mathcal{A}$, set 
	$$\mathrm{alt}~\underline{P}_A:=\sum_{x\in W_0}(-1)^{l(x)}\underline{P}_{x\ast A}. $$
	
	The link between periodic and antispherical polynomials is made explicit by the following result, which is \cite[Theorem 5.3(1)]{Soergel1997KazhdanLusztigPA}.
	\begin{prop}\label{period and antisph}
		For any alcove $A\subset \rho^\vee+\mathscr{C}^+_0$, we have 
		$$\underline{N}_A=\mathrm{res}~\mathrm{alt}~\underline{P}_A.$$
	\end{prop}
	For any special facet $\mathbf{v}$ and alcove $A\in\mathcal{A}$ such that $A\subset \Pi_\mathbf{v}$, we define the set
	$$S_A=\{C~|~wC\preceq A~\forall w\in W_\mathbf{v}\}. $$ 
	The next few results will help us utilize Proposition \ref{period and antisph}. More precisely, our main goal in Proposition \ref{generic poly} will be to determine for which alcoves $C$ one has $n_{C,A}=p_{C,A}$, where the alcove $A\subset\rho^\vee+\mathscr{C}_0^+$ is fixed. As recalled in the lemma below, our interest for the set $S_A$ comes from the fact that it contains the support of $\underline{P}_A$.
	\begin{lem}\label{support poly perio}
		Let $A\in\mathcal{A}$ and $\mathbf{v}$ be the special facet such that $A\subset\Pi_\mathbf{v}$. 
		\begin{enumerate}
			\item We have the implication $p_{B,A}\neq0\Rightarrow B\in S_A$.
			\item We have 
			$$S_A=\{wC~|~w\in W_\mathbf{v},~C\preceq A~\text{and}~ C\subset \mathscr{C}^+_\mathbf{v}\}. $$
			\item Assume that there exists some facet $\mathbf{g}\subset\overline{\mathbf{a}_1}$ such that $A\in A^+_\mathbf{v}W_\mathbf{g}$, and that $\mathbf{v}\subset x\mathscr{C}^+_0$ for some $x\in W_0\backslash\{\mathrm{id}\}$ (this simply means that $\mathbf{v}$ is inside another quarter with vertex $\{0\}$ than $\mathscr{C}^+_0$). Then we have 
			$$D\in S_A\Rightarrow D\not\subset\rho^\vee+\mathscr{C}^+_0.$$
		\end{enumerate}
	\end{lem}
	\begin{proof}
		\begin{enumerate}
			\item This is \cite[Proposition 4.22]{Soergel1997KazhdanLusztigPA}.
			\item Let us write
			$$S'_A=\{wC~|~w\in W_\mathbf{v},~C\preceq A~\text{and}~ C\subset \mathscr{C}^+_\mathbf{v}\}. $$
			The inclusion $S_A\subset S'_A$ just follows from the fact that any alcove $C\in\mathcal{A}$ has a $W_{\mathbf{v}}$ conjugate $wC$ inside of $\mathscr{C}^+_{\mathbf{v}}$. Conversely, if $C$ is an alcove such that $C\subset\mathscr{C}^+_{\mathbf{v}}$ and $C\preceq A$, then we have $wC\preceq C$ for every $w\in W_{\mathbf{v}}$ thanks to Lemma \ref{lem max coset}, so $wC\preceq A$ for every $w\in W_{\mathbf{v}}$. This proves the inclusion $S'_A\subset S_A$, and concludes the proof of the first point.
			\item Let $C\subset\mathscr{C}^+_{\mathbf{v}}$ be an alcove such that $C\preceq A$.\begin{itemize}
				\item We claim that $C$ belongs to $A^+_{\mathbf{v}}W_\mathbf{g}$. Since $\preceq$ is invariant under translations, we can translate everything by $-\mu$, where $\mathbf{v}=\{\mu\}$, so that we are reduced to the case where $\mathbf{v}=\{0\}$. More precisely, if we denote by $w_\mu\in W$ the element such that $w_\mu\mathbf{a}_1=A^+_\mathbf{v}$, by $w\in W_\mathbf{g}$ the element such that $A=A^+_\mathbf{v}w$ and set $\omega_\mu:=t_{-\mu}w_\mu$, then we get that $\omega_\mu\mathbf{a}_1=\mathbf{a}_1$ and
				\begin{multline*}C\preceq A\Longleftrightarrow t_{-\mu}C\preceq t_{-\mu}A=t_{-\mu}(w_\mu\mathbf{a}_1w)=t_{-\mu}(w_\mu w\mathbf{a}_1)= \omega_\mu w\mathbf{a}_1\\=(\omega_\mu w\omega_\mu^{-1})\omega_\mu\mathbf{a}_1=(\omega_\mu w\omega_\mu^{-1})\mathbf{a}_1.  \end{multline*}
			 So now, if we set $\mathbf{g}':=\omega_\mu\mathbf{g}$, we get the inequality $t_{-\mu}C\preceq w'\mathbf{a}_1$, with $w'=\omega_\mu w\omega_\mu^{-1}\in W_{\mathbf{g}'}$. Since the alcoves $t_{-\mu}C$ and $w'\mathbf{a}_1$ belong to $\mathcal{A}^+$ (because they are translates by $-\mu$ of alcoves in $\mathscr{C}^+_\mathbf{v}$), we get that $t_{-\mu}C\leq w'\mathbf{a}_1$, so that $t_{-\mu}C=w''\mathbf{a}_1$ for some $w''\leq w'$, from which we easily deduce that $w''\in W_{\mathbf{g}'}$; this means that $w''=\omega_\mu u\omega_\mu^{-1}$ for some $u\in W_\mathbf{g}$, so 
				\begin{multline*}
					C=t_\mu(\omega_\mu u\omega_\mu^{-1})\mathbf{a}_1=t_\mu(\omega_\mu u\omega_\mu^{-1})t_{-\mu}t_\mu\mathbf{a}_1=(w_\mu u) (w_\mu^{-1}t_\mu)\mathbf{a}_1=(w_\mu u) \omega_\mu^{-1}\mathbf{a}_1\\=w_\mu u\mathbf{a}_1=A^+_\mathbf{v}u.
				\end{multline*}
				\item Denote by $\mathbf{h}$ the $W$-conjugate of $\mathbf{g}$ contained in $\overline{A^+_{\mathbf{v}}}$. In particular, we get that  $\mathbf{h}$ belongs to both $\overline{A^+_{\mathbf{v}}}$ and $\overline{C}$ (because $C\in A^+_\mathbf{v}W_\mathbf{g}$ thanks to the previous paragraph), so that $\overline{C}\cap\overline{A^+_{\mathbf{v}}}\neq\emptyset$. From this and the description of $S_A$ given in (2), we deduce that for any alcove $D\in S_A$, there exists some $w\in W_{\mathbf{v}}$ such that 
				$$\overline{D}\cap w\overline{A^+_{\mathbf{v}}}\neq\emptyset$$ 
				(just write $D=wC$ for some alcove $C$ as above). But $wA^+_{\mathbf{v}}\subset x\mathscr{C}^+_0$ for any $w\in W_{\mathbf{v}}$ (because $x^{-1}wA^+_{\mathbf{v}}$ is an alcove containing $x^{-1}\mathbf{v}\subset \mathscr{C}^+_0$ in its closure, so that $x^{-1}wA^+_{\mathbf{v}}\subset \mathscr{C}^+_0$ thanks to Proposition \ref{dominant facet}), so we get ultimately that if an alcove $D\in S_A$ is contained in $\mathscr{C}^+_0$, then $\overline{D}$ must have a non-empty intersection with a hyperplane separating $\mathscr{C}^+_0$ from $x\mathscr{C}^+_0$, so that $D$ is not contained in $\rho^\vee+\mathscr{C}^+_0$.
			\end{itemize} 
		\end{enumerate}
	\end{proof}

	\begin{lem}\label{lemme contained box}
		Assume that the root system $\mathfrak{R}^\vee$ is irreducible. Let $\lambda$ be a nonzero weight contained in $\overline{\mathscr{C}^+_0}\cap\mathbb{X}^\vee$, $\mathbf{v}$ be a special facet, $C$ be an alcove contained in $\Pi_\mathbf{v}$ and put $A:=\lambda+C$. If $\overline{A}\cap\overline{A^+_\mathbf{v}}\neq\emptyset$, we must have $C=A^+_\mathbf{v}$.
	\end{lem}
	\begin{proof}
		In order to simplify the notations, we may and will reduce ourselves to the case where $\mathbf{v}=\{0\}$. Let us write
		$$C=\{\mu\in E~|~n_\alpha-1<\langle\mu,\alpha\rangle<n_\alpha ~\forall\alpha\in\mathfrak{R}_+\}$$
		for some integers $n_\alpha$. First notice that if $\alpha',\alpha\in\mathfrak{R}_+$ are such that $\alpha\leq \alpha'$ (i.e.  $\alpha'-\alpha$ is a sum of positive roots), then we must have $n_\alpha\leq n_{\alpha'}$. Indeed, since $C\subset\mathscr{C}^+_0$, we know that $\langle\mu,\alpha\rangle\leq \langle\mu,\alpha'\rangle$ for every $\mu\in C$, so in particular we get that  
		$$n_\alpha-1< \langle\mu,\alpha\rangle\leq\langle\mu,\alpha'\rangle< n_{\alpha'},$$ proving the claim. 
		
		Since $C\subset\Pi$, we know that $n_\alpha=1$ for every $\alpha\in S_0$, and thus $n_\alpha\geq 1$ for every $\alpha\in\mathfrak{R}_+$. The assumption that $\overline{A}\cap\overline{\mathbf{a}_1}\neq\emptyset$ implies that we have 
		$$n_{\alpha}+\langle\lambda,{\alpha}\rangle\in\{1,2\}\quad\text{for every}~\alpha\in\mathfrak{R}_+. $$ 
		Therefore we must have $\langle\lambda,{\alpha}\rangle\in\{0,1\}$ for every $\alpha\in \mathfrak{R}_+$. Since $\lambda\neq 0$, there exists some $\alpha_0\in S_0$ such that $\langle\lambda,\alpha_0\rangle>0$. So let $\beta=\sum_{S_0}c_\alpha \alpha\in \mathfrak{R}_+$ be the longest root (cf. \cite[Chap. 6, Proposition 25]{cdi_springer_books_10_1007_978_3_540_34491_9}, this is where we need the assumption that $\mathfrak{R}$ is irreducible). In particular we have $c_\alpha\geq 1$ for every $\alpha\in S_0$, so that $\langle\lambda,\beta\rangle>0$, and thus $n_\beta=1$. But for every $\alpha\in \mathfrak{R}_+$, we have $\alpha\leq\beta$, so we get that $n_{\alpha}\leq n_\beta$, and finally $n_\alpha=1$. This means that $C=\mathbf{a}_1$.
	\end{proof}
	The following corollary can be visualized on Figure \ref{Type C2} below.
	\begin{coro}\label{lem boite}
		Let $\mathbf{v}$ be a special facet, $\mathbf{g}\subset\overline{\mathbf{a}_1}$ be a non-special facet and $A\in A^+_\mathbf{v}W_\mathbf{g}$ be the alcove which is maximal in $A^+_\mathbf{v}W_\mathbf{g}$ for $\preceq$. Then $A$ belongs to $\Pi_\mathbf{v}$.
	\end{coro}
	\begin{proof}
		Let us first explain why $A\subset \mathscr{C}^+_\mathbf{v}$. Denote by $\mathbf{h}'$ the element of $\mathcal{A}_\mathbf{g}$ which belongs to the closure of $A^+_\mathbf{v}$. We have the inclusion $A^+_\mathbf{v}\subset \mathscr{C}^+_\mathbf{v}$, so we know that $\mathbf{h}'$ is included in $\overline{\mathscr{C}^+_\mathbf{v}}$. If $C\in A^+_\mathbf{v}W_\mathbf{g}$ is an alcove which is not included in $\mathscr{C}^+_\mathbf{v}$, then there exists a hyperplane $H$ of the boundary of $\mathscr{C}^+_\mathbf{v}$ which contains $\mathbf{h}'$ and separates $C$ from $\mathscr{C}^+_{\mathbf{v}}$, so we clearly have the inclusion $C\subset E^-_H$, and therefore $C\prec s_HC$. Since $s_H\in W_{\mathbf{h}'}$, the alcove $s_HC$ must belong\footnote{Indeed, write $C=w\mathbf{a}_1$ for some $w\in W$, then  $\mathbf{h}'=w\mathbf{g}$ and $w^{-1}s_Hw\in W_\mathbf{g}$, so that $s_HC=C(w^{-1}s_Hw)\in CW_\mathbf{g}=A^+_\mathbf{v}W_\mathbf{g}$.} to $A^+_\mathbf{v}W_\mathbf{g}$. This shows that the maximal element in  $A^+_\mathbf{v}W_\mathbf{g}$ for $\preceq$ must be contained in $\mathscr{C}^+_\mathbf{v}$, whence $A\subset \mathscr{C}^+_\mathbf{v}$.
		
		Now let us show that $A\subset\Pi_\mathbf{v}$. For this we may and will assume until the end of this proof that $\mathfrak{R}$ is irreducible. Assume that this inclusion does not hold, then because $A\subset \mathscr{C}^+_\mathbf{v}$, we can write $A=\lambda+C$ for some nonzero weight $\lambda\in\overline{\mathscr{C}^+_0}\cap\mathbb{X}^\vee$ and some alcove $C\subset\Pi_\mathbf{v}$. But we also have $\overline{A}\cap \overline{A^+_\mathbf{v}}\neq\emptyset$ (both of those sets contain $\mathbf{h}'$), therefore we can apply Lemma \ref{lemme contained box} (here we use the assumption that $\mathfrak{R}$ is irreducible) to see that $C=A^+_\mathbf{v}$. So we are in the following situation:
		\begin{equation}\label{before trans}
			\lambda+A^+_\mathbf{v}\in A^+_\mathbf{v}W_\mathbf{g}. 
		\end{equation}
		Let $\mu\in\mathbb{X}^\vee$ be such that $\{\mu\}=\mathbf{v}$, and denote by $w_\mu\in W$ the element such that $w_\mu\mathbf{a}_1=A^+_\mathbf{v}$. Putting $\omega_\mu:=t_{-\mu}w_\mu$ and applying $t_{-\mu}$ to (\ref{before trans}), we get that
		$$\lambda+\omega_\mu\mathbf{a}_1\in \omega_\mu\mathbf{a}_1 W_\mathbf{g}=(\omega_\mu W_\mathbf{g}\omega_\mu^{-1})\omega_\mu\mathbf{a}_1, $$
		and, using the fact that $\omega_\mu$ fixes the fundamental alcove $\mathbf{a}_1$, we obtain
		\begin{equation}\label{after trans}
			\lambda+\mathbf{a}_1\in W_{\mathbf{g}'}\mathbf{a}_1,~ \text{with}~\mathbf{g}':=\omega_\mu\mathbf{g}.
		\end{equation}
		Since $\omega_\mu\in\Omega$ (because $\omega_\mu$ fixes $\mathbf{a}_1$) and $\mathbf{g}$ is a non-special facet, $\mathbf{g}'$ is also a non-special facet\footnote{Indeed, the action of $W$ (and hence of $\Omega$) on $E$ obviously permutes the set of hyperplanes $\mathcal{H}$, so it also permutes the set of special facets by construction.} included in the closure of the fundamental alcove. From (\ref{after trans}) we deduce that $\mathbf{g}'\subset \lambda+\overline{\mathbf{a}_1}$. But $\lambda\in\overline{\mathscr{C}^+_0}\cap\mathbb{X}^\vee$ is nonzero, so $\langle\lambda,\beta\rangle\geq 1$, where $\beta\in\mathfrak{R}_+$ is the longest root. Since $\mathbf{g}'\subset \overline{\mathbf{a}_1}$, we have that $\langle\mu,\beta\rangle\leq 0$ for every  $\mu\in-\lambda+\mathbf{g}'$ (because $\langle\mu',\beta\rangle\leq 1$ for all $\mu'\in\overline{\mathbf{a}_1}$); therefore the inclusion  $-\lambda+\mathbf{g}'\subset \overline{\mathbf{a}_1}$ is possible only when $\langle\mu,\beta\rangle=0$ for all $\mu\in-\lambda+\mathbf{g}'$. Since $\alpha\leq \beta$ for all $\alpha\in\mathfrak{R}_+$ we get that 
		$$0\leq \langle\mu,\alpha\rangle\leq \langle\mu,\beta\rangle=0~\forall\alpha\in\mathfrak{R}_+,\mu\in-\lambda+\mathbf{g}'$$ which means that $\mu=0$ and so $\mathbf{g}'=\{\lambda\}$, contradicting our assumption that $\mathbf{g}'$ is not a special facet. This concludes the proof of the lemma.
	\end{proof}
	We are now ready to prove our central result.
	\begin{prop}\label{generic poly} 
		Let $\mathbf{g}\subset\overline{\mathbf{a}_1}$ be a non-special facet, $\mathbf{h}\in\mathcal{A}_\mathbf{g}^+$ be such that $\mathbf{h}\subset\rho^\vee+\overline{\mathscr{C}^+_0}$ and $B\subset\rho^\vee+\mathscr{C}^+_0$ be an alcove containing $\mathbf{h}$ in its closure. Let also $\mathbf{v}$ be a special facet such that $\mathbf{v}\subset\overline{B}$, and $A$ be the alcove which is maximal in $A^+_\mathbf{v}W_\mathbf{g}$ for $\leq$. Then we have $A\subset \Pi_\mathbf{v}$ and  
		$$n_{Br,A}(1)=1\qquad \forall r\in W_\mathbf{g}.$$
	\end{prop}
	\begin{proof}
		Let $\mathbf{h}'$ be the element of $\mathcal{A}_\mathbf{g}$ which is included in $\overline{A^+_\mathbf{v}}$. By construction (i.e. because $\mathbf{v}\subset\rho^\vee+\overline{\mathscr{C}_0^+}$), we have that $A^+_\mathbf{v}\subset\rho^\vee+\mathscr{C}^+_0$, so that in particular $\mathbf{h}'\subset \mathscr{C}^+_0$. Thus, applying Proposition \ref{dominant facet}, we know that all the alcoves of $A^+_\mathbf{v}W_\mathbf{g}$ are contained in $\mathcal{A}^+$ (because they contain $\mathbf{h}'$ in their closure), so that the orders $\preceq$ and $\leq$ coincide in $A^+_\mathbf{v}W_\mathbf{g}$. Therefore $A$ is maximal in $A^+_\mathbf{v}W_\mathbf{g}$ for $\preceq$, so is included in $\Pi_\mathbf{v}$ thanks to Corollary \ref{lem boite}.
		
		By Lemma \ref{period and antisph} and because $A\subset \rho^\vee+\mathscr{C}^+_0$, we have that
		\begin{equation}\label{periodic}
			\sum_{C}n_{C,A}N_C=\sum_{x\in W_0}(-1)^{l(x)}\sum_{C}p_{C,x\ast A}~\mathrm{res}~C. 
		\end{equation}
		Our next goal is to use this formula to show that
		$$ n_{C,A}=p_{C,A}\qquad\text{for all}~C\subset  \rho^\vee+\mathscr{C}^+_0. $$
		For that we will prove that when the alcove $C\subset  \rho^\vee+\mathscr{C}^+_0$ is fixed and $x\in W_0\backslash\{\mathrm{id}\}$, we have $p_{C,x\ast A}=0$; this will follow from a careful study of the periodic polynomials. 
		
		Fix $x\in W_0\backslash\{\mathrm{id}\}$, and notice that $x\ast A\subset \Pi_{x\mathbf{v}}$ since $A\subset \Pi_\mathbf{v}$. We can then apply the first point of Lemma \ref{support poly perio}:
		$$p_{D,x\ast A}\neq 0 \Longrightarrow D\in S_{x\ast A}.$$
		But we have that $x\ast A\in A^+_{x\mathbf{v}}W_\mathbf{g}$ (because $x\ast A$ and $A^+_{x\mathbf{v}}$ are just translates of $A$ and $A^+_{\mathbf{v}}$ respectively, by the same translation of $\mathbb{Z}\mathfrak{R}^\vee$) and $x\mathbf{v}\subset x\mathscr{C}^+_0$, so the third point of Lemma \ref{support poly perio} tells us that 
		$$D\in S_{x\ast A}\Rightarrow D\not\subset  \rho^\vee+\mathscr{C}^+_0.$$
		This fact, combined with (\ref{periodic}), implies that we have 
		\begin{equation}\label{periodic vs antisph}
			n_{C,A}=p_{C,A}\qquad\text{for all}~C\subset  \rho^\vee+\mathscr{C}^+_0.
		\end{equation}
		
		On the other hand, recall the result of Lemma \ref{lem inv}:
		\begin{equation}\label{inva peri}
			p_{wC,A}(1)=p_{C,A}(1) ~\forall w\in W_{\mathbf{v}}.
		\end{equation}
		Thus, for any alcove $C$  included in $\rho^\vee+\mathscr{C}^+_0$, (\ref{periodic vs antisph}) and (\ref{inva peri}) yield
		\begin{equation}\label{inva antisph}
			n_{wC,A}(1)=n_{C,A}(1) \qquad\forall w\in W_{\mathbf{v}}~\text{such that}~wC\subset \rho^\vee+\mathscr{C}^+_0.
		\end{equation}
		
		By construction, there exist elements $w\in W_\mathbf{v},~r_0\in W_\mathbf{g}$ such that $B=wA^+_\mathbf{v}$ and $A^+_\mathbf{v}=Ar_0$, so that $B=wAr_0$. Since $B\subset\rho^\vee+\mathscr{C}^+_0$, we can apply (\ref{inva antisph}) to get that
		$$n_{B,A}(1)=n_{Ar_0,A}(1). $$
		Finally, since $A$ is maximal in $AW_\mathbf{g}$ for $\leq$, we can apply successively the first point of Proposition \ref{prop max coset} (where the facet $\mathbf{q}$ from this proposition is our current $\mathbf{g}$, and where the elements $w\in{_\mathrm{f}W}^\mathbf{g}$, $w'\in{_\mathrm{f}W}$ are such that $w\mathbf{a}_1=A$ and $w'\mathbf{a}_1=B$) to obtain
		$$n_{Br,A}(1)=n_{B,A}(1)=n_{Ar_0,A}(1)=n_{A,A}(1)=1~\forall r\in W_\mathbf{g}. $$
	\end{proof}
	See Figure \ref{Type C2} for an illustration of the setting of the previous Proposition.
	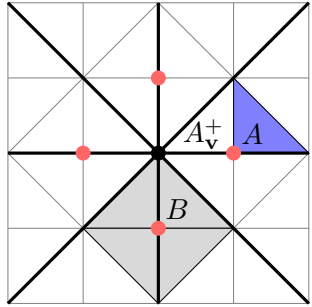
\begin{figure}[!h]	\begin{tikzpicture}[>=stealth]
		\foreach \x in {1,...,5}{
			\draw[ultra thin, color=gray] (\x,1)--(\x,5);
		}
		\foreach \x in {1,...,5}{
			\draw[ultra thin, color=gray] (1,\x)--(5,\x);
		}
		
		\foreach \x in {1,...,3}{
			\draw[ultra thin, color=gray] (1,2*\x-1)--(2*\x-1,1);
			\draw[ultra thin, color=gray] (1,2*\x-1)--(7-2*\x,5);
		}
		\draw[ultra thin, color=gray] (3,1)--(5,3);
		\draw[ultra thin, color=gray] (5,1)--(5,1);
		\draw[ultra thin, color=gray] (3,5)--(5,3);
		\draw[very thin, fill=gray!30] (3,3)--(4,2)--(3,2)--(3,3);
		\draw[very thin, fill=gray!30] (3,3)--(3,2)--(2,2)--(3,3);
		\draw[very thin, fill=gray!30] (4,2)--(3,1)--(3,2)--(4,2);
		\draw[very thin, fill=gray!30] (2,2)--(3,1)--(3,2)--(2,2);
		\draw[very thin, fill=blue!50] (4,3)--(5,3)--(4,4)--(4,3);
		\draw[very thick] (3,3)--(3,5);
		\draw[very thick] (3,3)--(5,5);
		\draw[very thick] (3,3)--(5,3);
		\draw[very thick] (3,3)--(5,1);
		\draw[very thick] (3,3)--(3,1);
		\draw[very thick] (3,3)--(1,1);
		\draw[very thick] (3,3)--(1,3);
		\draw[very thick] (3,3)--(1,5);
		
		\draw (3,3) node[circle,inner sep=2, fill=black] {};
		\draw (4,3) node[circle,inner sep=2, fill=red!60] {};
		\draw (2,3) node[circle,inner sep=2, fill=red!60] {};
		\draw (3,2) node[circle,inner sep=2, fill=red!60] {};
		\draw (3,4) node[circle,inner sep=2, fill=red!60] {};
		\draw (3.25,2.25) node[scale=1,rotate=0]{$B$};
		\draw (3.6,3.25) node[scale=1,rotate=0]{$A^+_\mathbf{v}$};
		\draw (4.25,3.25) node[scale=1,rotate=0]{$A$};

	\end{tikzpicture}
		\caption{\textit{Proposition \ref{generic poly}  in type $C_2$}. The black dot is $\mathbf{v}$ and the thick black lines represent the affine hyperplanes fixed by the reflections in $W_\mathbf{v}$. The red dots are the $W_\mathbf{v}$-conjugates of $\mathbf{h}$, and the set of gray alcoves represents $\{Br,~r\in W_\mathbf{g}\}$.}  \label{Type C2}
	\end{figure}
	\begin{rem}
	
		Keep the notations of the previous proposition. By positivity of the coefficients of the anti-spherical Kazdhan-Lusztig polynomials, the fact that $n_{Br,A}(1)=1$ for all $r\in W_\mathbf{g}$ means that the polynomial $n_{Br,A}$ is always a non-zero monomial.
	\end{rem}
	\begin{coro}\label{coro periodic}	
		Let $\mathbf{g}\subset\overline{\mathbf{a}_1}$ be a non-special facet, $\mathbf{h}\in\mathcal{A}^+_\mathbf{g}$ be such that $\mathbf{h}\subset\rho^\vee+\overline{\mathscr{C}^+_0}$, $B\subset\rho^\vee+\mathscr{C}^+_0$ be an alcove such that $\mathbf{h}\subset\overline{B}$, and let $\mathbf{v}\subset\rho^\vee+\overline{\mathscr{C}^+_0}$ be a special facet such that $\mathbf{v}\subset\overline{B}$. If we let $u\in W_\mathbf{v}$ be such that $B':=uB$ is included in $\rho^\vee+\overline{\mathscr{C}^+_0}$ and denote by $\mathbf{h}'$ the element of $\mathcal{A}^+_\mathbf{g}$ included in $\overline{B'}$, then we have $w'\sim_\mathbf{g}w$, where $w',w\in {_\mathrm{f}W}^\mathbf{g}$ are such that $\mathbf{h}=w\mathbf{g}$ and $\mathbf{h}'=w'\mathbf{g}$.
	\end{coro}
	\begin{proof}
		Denote by $A$ the alcove which is maximal in $A^+_\mathbf{v}W_\mathbf{g}$ for $\leq$, and by $\mathbf{h}''$ the element of $\mathcal{A}_\mathbf{g}$ included in $\overline{A}$. Recall that $A\subset \Pi_\mathbf{v}$ thanks to Proposition \ref{generic poly}, so that in particular $\mathbf{h}''\in\mathcal{A}^+_\mathbf{g}$. So if we let $w''$ be the element of ${_\mathrm{f}W}^\mathbf{g}$ which satisfies $w''\mathbf{g}=\mathbf{h}''$ and $r\in W_\mathbf{g}$ be the element such that $Br$ is maximal in $BW_\mathbf{g}$ for $\leq$ (so that $w\mathbf{a}_1=Br$ and $w''\mathbf{a}_1=A$), then Proposition \ref{generic poly} tells us that 
		$$n_{Br,A}(1)=1,$$
		from which we deduce that $w\sim_\mathbf{g} w''$ thanks to Corollary \ref{anti spherical implies R}. The same reasoning (considering this time the alcove $B'$  containing both $\mathbf{h}'$ and $\mathbf{v}$ in its closure) shows that $w'\sim_\mathbf{g} w''$, so finally $w\sim_\mathbf{g}w'$ by transitivity.
	\end{proof}
	We can finally complete the proof of Theorem \ref{thm not points}.
	\begin{thm}\label{cas general}
		Assume that $\mathfrak{R}^\vee$ is irreducible and let $\mathbf{g}\subset\overline{\mathbf{a}_1}$ be a non-special facet. Then ${_\mathrm{f}W}^\mathbf{g}$ consists of a single class for the equivalence relation $\sim_{\mathbf{g}}$.
	\end{thm}
	\begin{proof}
		Thanks to Theorem \ref{thm not points}, we may and will assume that $\mathbf{g}$ is a point. With the notation of the proof of Theorem \ref{thm not points}, recall that we want to show that $\mathcal{A}^+_\mathbf{g}$ consists of a single equivalence class for $\sim_{\mathbf{g}}$. Also recall from the proof of Theorem \ref{thm not points} that it is enough to show that all the elements of $\mathcal{A}^+_\mathbf{g}$ lying inside of $\rho^\vee+ \overline{\mathscr{C}^+_0}$ are in relation.
		
		Pick $\mathbf{h}$ inside of $\rho^\vee+\overline{\mathscr{C}^+_0}$ and let $A\subset \rho^\vee+\mathscr{C}^+_0$ be an alcove containing $\mathbf{h}$ in its closure. Let also $A_r,A_{r-1},\cdots,A_1,A_0:=A$ be alcoves as in Lemma \ref{descent}, and denote by $s_i\in S$ the reflection such that $A_is_i=A_{i-1}$. 
		
		Since $\mathfrak{R}$ is irreducible, there exists a unique simple reflection $\sigma\in S$ which does not belong to the finite Weyl group $W_0$. Moreover, because $\overline{\mathbf{a}_1}$ is a simplex, there is only one facet inside $\overline{\mathbf{a}_1}$ which is a point and which is not included in the hyperplane fixed by $\sigma$: this is the special facet $\{0\}$. Thus, since $\mathbf{g}$ is not special, it must be included in the wall fixed by $\sigma$, or in other words we have $\sigma\in W_\mathbf{g}$. 
		
		For $i\in\llbracket 0,\cdots,r\rrbracket$, denote by $w_i$ the element of $_{\mathrm{f}}W$ such that $w_i\mathbf{a}_1=A_i$ and put $\mathbf{h}_i:=w_i\mathbf{g}$. Notice that $A_is_i=w_is_iw_i^{-1}A_i$ and that $\mathbf{h}_i\subset \overline{A_i}$. By construction, we also have
		$$\mathbf{h}_{i-1}=w_is_iw_i^{-1} \mathbf{h}_{i}~\forall i\in\llbracket 1,\cdots,r\rrbracket~\text{and}~\mathbf{h}_0=\mathbf{h}.$$
		
		Fix $i\in\llbracket 1,\cdots,r\rrbracket$. Now, if $s_i=\sigma$, then $w_is_iw_i^{-1}$ fixes $\mathbf{h}_{i}$, so that $\mathbf{h}_{i}=\mathbf{h}_{i-1}$. So assume that $s_i\neq \sigma$, which implies that $s_i\in W_0$, and put $\mathbf{v}_i:=w_i\{0\}$. Then we see that $\mathbf{v}_i$ is a special facet included in $\overline{A_i}$, which thus satisfies the inclusion $\mathbf{v}_i\subset\rho^\vee+\overline{\mathscr{C}^+_0}$, with $w_is_iw_i^{-1}\in W_{\mathbf{v}_i}$. Therefore we may apply Corollary \ref{coro periodic} to see that $\mathbf{h}_{i}\sim_\mathbf{g}\mathbf{h}_{i-1}$. By transitivity of the relation we get $\mathbf{h}\sim_{\mathbf{g}}\mathbf{h}_r$, concluding the proof.
	\end{proof}
	\begin{rem}\label{precise link}
	    Take back the notations of the previous proof. Corollary \ref{coro periodic} (when $\mathbf{g}$ is a point) and Lemma \ref{desending facets} (when $\mathbf{g}$ is not a point) actually allow us to be a bit more precise: for all $i\in\llbracket1,\cdots,r\rrbracket$, there exists an element $u_i\in{_\mathrm{f}W}^\mathbf{g}$ (which is the maximal element in $W_{\mathbf{v}_i}w_i$ in Corollary \ref{coro periodic}, and the maximal element in $w_iW_{\mathbf{q}_i}$ in  Lemma \ref{desending facets}) such that $n_{w_i,u_i}(1)\neq0$ and $n_{w_{i-1},u_i}(1)\neq0$.
	\end{rem}
	One of the advantages of our proof is that it gives an explicit way of linking an element $w\in{_\mathrm{f}W}^\mathbf{g}$ to the unique element $v$ of ${_\mathrm{f}W}^\mathbf{g}$ satisfying $v\mathbf{g}\subset \rho^\vee+\overline{\mathbf{a}_1}$. This allows us to give the following result.
	\begin{coro}\label{bound}
		Assume that $\mathfrak{R}^\vee$ is irreducible. Let $\mathbf{g}\subset\overline{\mathbf{a}_1}$ be a non-special facet and $w,w'$ be elements of $ {_\mathrm{f}W}^\mathbf{g}$ which are in the same equivalence class for $\sim_\mathbf{g}$. Denote by $A_w$ (resp. $A_{w'}$) the alcove which contains $w\mathbf{g}$ (resp. $w'\mathbf{g}$) in its closure and which is minimal in $A_wW_\mathbf{g}$ (resp. $A_{w'}W_\mathbf{g}$) for $\preceq$. Then there exists a positive integer $s$ and a chain of elements of ${_\mathrm{f}W}^\mathbf{g}$
		$$w_s=w,w_{s-1},\cdots,w_0=w' $$
		such that, for all $i\in\llbracket 0,s-1\rrbracket$, there exists an element $u_i\in {_\mathrm{f}W}^\mathbf{g}$ satisfying
		$${ n}_{w_i,u_i}(1)\neq 0~\text{and}~{ n}_{w_{i+1},u_i}(1)\neq 0, $$
		and such that
		$$s\leq 2+d(\hat{A}_w-\rho^\vee)+d(\hat{A}_{w'}-\rho^\vee). $$
	\end{coro}
	\begin{proof}
		Denote by $v$ the unique element of ${_\mathrm{f}W}^\mathbf{g}$ such that $v\mathbf{g}\subset\rho^\vee+\overline{\mathbf{a}_1}$. Let us also denote by $\hat{w}$ the element of ${_\mathrm{f}W}^\mathbf{g}$ such that $\hat{w}\mathbf{g}\subset \overline{\hat{A}_w}$ and $A_0:=\hat{A}_w,A_{r-1},\cdots,A_r:=\rho^\vee+\mathbf{a}_1$ be the alcoves of Lemma \ref{descent}. By Remark \ref{precise link}, we know that there exist elements $w_i,u_i\in {_\mathrm{f}W}^\mathbf{g}$ such that $(w_r,w_0)=(v,\hat{w})$ and
		$$w_i\mathbf{g}\subset \overline{A_i},~n_{w_i,u_i}(1)\neq0~\text{and}~n_{w_{i+1},u_i}(1)\neq0,~\forall i\in \llbracket 0,r-1\rrbracket .$$
		Therefore, the chain $w,\hat{w},w_1,\cdots,w_r=v$ linking $w$ to $v$ is of length equal to $1$ plus the number of hyperplanes separating $\hat{A}_w$ from $\mathbf{a}_1+\rho^\vee$, and this last number is equal to the number of hyperplanes separating $\hat{A}_w-\rho^\vee$ from $\mathbf{a}_1$, i.e. to $d(\hat{A}_w-\rho^\vee)$. Repeating the same process for $w'$, we find a chain linking $w'$ to $v$ of length equal to $1+d(\hat{A}_{w'}-\rho^\vee)$, so that the concatenation of these two chains yields the desired result.
	\end{proof}
	
	\section{Smith-Treumann theory and consequences}\label{smith section}
	\subsection{Geometric Satake equivalence}\label{geosat}
	The affine Grassmannian $\mathrm{Gr}$ associated with $G$ can be defined as the fppf-quotient $LG/L^+G$, which can be shown to be an ind-projective ind-scheme over $\mathbb{F}$. In particular, $\mathrm{Gr}$ coincides with the partial affine flag variety $\mathrm{Fl}_{\{0\}}$ (cf. subsection \ref{The affine Weyl group and some Bruhat-Tits theory}). For any $\lambda\in\mathbb{X}^\vee$, we denote by $z^\lambda\in T(\mathcal{O})$ the image of $z$ by $\lambda:\mathbb{G}_\mathrm{m}(\mathcal{K})\to T(\mathcal{K})$, and define $[\lambda]:=z^\lambda L^+G$. The orbit $L^+G\cdot [\lambda]$ will be denoted $X_\lambda$. Using the Cartan decomposition, one can show that the action of $L^+G$ on $\mathrm{Gr}$ induces the following equality
	$$(\mathrm{Gr})_{\mathrm{red}}=\bigsqcup_{\lambda\in \mathbb{X}^\vee_+}X_\lambda,$$
	where $(\mathrm{Gr})_{\mathrm{red}}$ is the reduced ind-scheme associated with $\mathrm{Gr}$. Moreover, each orbit $X_\lambda$ is a smooth $\mathbb{F}$-scheme of finite type and
	$$\overline{X_\lambda}=\bigsqcup_{\substack{\mu\leq\lambda\\\mu\in \mathbb{X}^\vee_+}}X_\mu$$
	is a projective $\mathbb{F}$-scheme, where $\mu\leq \lambda$ means as usual that $\lambda-\mu$ is a sum of positive coroots.
	
	For any $X$ which is a locally closed finite union of $L^+G$-orbits, one can show that the $L^+G$-action on $X$ factors through a quotient group of finite type $J$; this quotient may and will be chosen so that the kernel of the map $L^+G\to J$ is contained in $\mathrm{ker}(L^+G\to G)$. More concretely, one can take $J$ to be the group representing the functor $R\mapsto G(R[z]/z^n)$, for a large enough integer $n$. We can then consider the $J$-equivariant derived category of constructible \'etale $\mathbf{k}$-sheaves $D^b_{J}(X,\mathbf{k})$, which does not depend on our choice of $J$ (see \cite[§3.3]{ciappara2021hecke}). We are now allowed to define the category $D^b_{L^+G}(\mathrm{Gr},\mathbf{k})$ as the direct limit of the categories $D^b_{J}(X,\mathbf{k})$, indexed by finite and closed unions of orbits $X$ which are ordered by inclusion and where the transition maps are given by direct images. This category admits a canonical perverse $t$-structure, and we denote by $\mathrm{Perv}_{L^+G}(\mathrm{Gr},\mathbf{k})$ its heart.  
	
	For $\lambda\in\mathbb{X}^\vee_+$, let $i^\lambda:X_\lambda\hookrightarrow\mathrm{Gr}$ denote the inclusion and define
	$$\Delta^{\mathrm{sph}}_\lambda:={^p\mathrm{H}}^0(i^\lambda_{ !}\underline{\mathbf{k}}_{X_\lambda}[\mathrm{dim}(X_\lambda)]),\qquad\nabla^{\mathrm{sph}}_\lambda:={^p\mathrm{H}}^0(i^\lambda_{ *}\underline{\mathbf{k}}_{X_\lambda}[\mathrm{dim}(X_\lambda)]).$$
	The complex $\Delta^{\mathrm{sph}}_\lambda$, resp. $\nabla^{\mathrm{sph}}_\lambda$, is called the standard perverse sheaf, resp the costandard perverse sheaf attached to $\lambda$. By construction, there is a canonical map $\Delta^{\mathrm{sph}}_\lambda\to\nabla^{\mathrm{sph}}_\lambda$, and we define the intersection cohomology complex $\mathrm{IC}_\lambda^{\mathrm{sph}}$ associated with $\lambda$ as the image of this map. Intersection cohomology complexes cover (up to isomorphism) all of the simple objects of the category $\mathrm{Perv}_{L^+G}(\mathrm{Gr},\mathbf{k})$ when $\lambda$ runs through $\mathbb{X}^\vee_+$. 
	\begin{prop}[Proposition 12.4, \cite{baumann:hal-01491529}]
		The category $\mathrm{Perv}_{L^+G}(\mathrm{Gr},\mathbf{k})$ is a highest weight category, with weight poset $\mathbb{X}^\vee_+$, standard objects $\{\Delta^{\mathrm{sph}}_\lambda,~\lambda\in \mathbb{X}^\vee_+\}$ and costandard objects $\{\nabla^{\mathrm{sph}}_\lambda,~\lambda\in \mathbb{X}^\vee_+\}$.
	\end{prop}

	The category $D^b_{L^+G}(\mathrm{Gr},\mathbf{k})$ admits a monoidal product $\star$, which restricts to a monoidal product on $\mathrm{Perv}_{L^+G}(\mathrm{Gr},\mathbf{k})$. In the following crucial result, we denote by $G^\vee$ the Langlands dual group of $G$ over $\mathbf{k}$, and by $\mathrm{Rep}_\mathbf{k}(G^\vee)$ the category of algebraic finite dimensional $\mathbf{k}$-representations of $G^\vee$.
	\begin{thm}[Theorem 14.1, \cite{Mirkovic2004GeometricLD}]\label{satake equivalence}
		There is an equivalence of monoidal categories 
		$$(\mathrm{Perv}_{L^+G}(\mathrm{Gr},\mathbf{k}),\star)\xrightarrow{\sim} (\mathrm{Rep}_\mathbf{k}(G^\vee),\otimes_\mathbf{k}).$$
	\end{thm}
	\begin{rem}
		Let $\lambda\in\mathbb{X}^\vee_+$. It is well known that the category $\mathrm{Rep}_\mathbf{k}(G^\vee)$ also admits a highest weight structure, and we denote by $\Delta_\lambda$, resp. $\nabla_\lambda$, resp. $L_\lambda$, the standard object, resp. the costandard object, resp. the simple module, associated with $\lambda$. Then, the previous equivalence of categories sends $\nabla_\lambda^{\mathrm{sph}}$, resp. $\Delta_\lambda^{\mathrm{sph}}$, resp. $\mathrm{IC}_\lambda^{\mathrm{sph}}$, on $\nabla_\lambda$, resp. $\Delta_\lambda$, resp. $L_\lambda$ (cf. \cite[Proposition 13.1]{Mirkovic2004GeometricLD}). We will say that this equivalence is an \textit{equivalence of highest weight categories}.
	\end{rem}

	\subsection{Iwahori-Whittaker variant}\label{IW variant} Recall the construction of $D^b_{\mathcal{IW}}(\mathrm{Gr},\mathbf{k})$ from subsection \ref{section partial affine flag}. If we denote by $Y_\lambda$ the orbit of $[\lambda]$ under the action of $\mathrm{Iw}_u^+$, then we have a decomposition
	$$(\mathrm{Gr})_{\mathrm{red}}=\bigsqcup_{\lambda\in \mathbb{X}^\vee}Y_\lambda,$$
	where each $Y_\lambda$ is a finite dimensional affine space over $\mathbb{F}$. One can show that an orbit $Y_\lambda$ supports a non-zero Iwahori-Whittaker local system iff $\lambda$ is strictly dominant, and that in this case there exists exactly one (up to isomorphism) such local system of rank one on $Y_\lambda$, which we will denote by $\mathcal{L}^\lambda_{\mathrm{AS}}$ (here $\mathrm{AS}$ stands for Artin-Schreier).

	Once again, the category $D^b_{\mathcal{IW}}(\mathrm{Gr},\mathbf{k})$ admits a canonical perverse $t$-structure, and we will denote by $\mathrm{Perv}_{\mathcal{IW}}(\mathrm{Gr},\mathbf{k})$ its heart. Thanks to the fact that the $\mathrm{Iw}_u^+$-orbits are affine spaces over $\mathbb{F}$, this category of perverse sheaves admits a transparent highest weight structure (cf. \cite[Corollary 3.6]{JEP_2019__6__707_0}). Namely, the weight poset is given by $\mathbb{X}^\vee_{++}$, and the standard, costandard and simple objects associated with some $\lambda\in\mathbb{X}^\vee_{++}$ are respectively given by
	$$\Delta^{\mathcal{IW}}_\lambda:=j^\lambda_{ !}\mathcal{L}^\lambda_{\mathrm{AS}}[\mathrm{dim}(Y_\lambda)],\qquad \nabla^{\mathcal{IW}}_\lambda:=j^\lambda_{ *}\mathcal{L}^\lambda_{\mathrm{AS}}[\mathrm{dim}(Y_\lambda)],\qquad \mathrm{IC}_\lambda^{\mathcal{IW}},$$
	where $j^\lambda:Y_\lambda\hookrightarrow\mathrm{Gr}$ is the inclusion, and $\mathrm{IC}_\lambda^{\mathcal{IW}}$ is obtained as the image of the canonical morphism $\Delta^{\mathcal{IW}}_\lambda\to \nabla^{\mathcal{IW}}_\lambda$. In the sequel, we will denote by $\mathrm{Tilt}_{\mathcal{IW}}(\mathrm{Gr},\mathbf{k})$ the category of tilting objects associated with this highest weight category, and by $\mathscr{T}^{\mathcal{IW}}_{\lambda}$ the indecomposable tilting object of highest weight $\lambda\in\mathbb{X}^\vee_{++}$.
	
	Although the category $D^b_{\mathcal{IW}}(\mathrm{Gr},\mathbf{k})$ is not endowed with a canonical convolution product making it a monoidal category, it admits a right action of the monoidal category $D^b_{L^+G}(\mathrm{Gr},\mathbf{k})$. The following result was found by the authors of \cite{JEP_2019__6__707_0}, and gives another incarnation of the category $\mathrm{Rep}_\mathbf{k}(G^\vee)$. Note that by our assumptions on $G$, the element $\rho^\vee$, defined as the half-sum of positive coroots, belongs to $\mathbb{X}^\vee$.
	\begin{thm}\label{IW equivalence}
		The functor  $\mathcal{F}\mapsto \nabla^{\mathcal{IW}}_{\rho^\vee}\star \mathcal{F}$ induces an equivalence of highest weight categories 
		$$\mathrm{Perv}_{L^+G}(\mathrm{Gr},\mathbf{k})\xrightarrow{\sim}\mathrm{Perv}_{\mathcal{IW}}(\mathrm{Gr},\mathbf{k}).$$
	\end{thm}
	\begin{rem}\label{rem1}
		At this stage, we can note that the action of $\mathbb{G}_\mathrm{m}$ on $\mathrm{Gr}$ by rescaling $z$ stabilizes each $\mathrm{Iw}_u^+$-orbit. This allows us to consider the loop rotation equivariant Iwahori-Whittaker derived category of $\mathbf{k}$-sheaves
		$$D^b_{\mathcal{IW},\mathbb{G}_\mathrm{m}}(\mathrm{Gr},\mathbf{k}),$$
		which comes with a natural $t$-exact forgetful functor to $D^b_{\mathcal{IW}}(\mathrm{Gr},\mathbf{k})$ (cf. \cite[§5.2]{RW22}). It is then not difficult to show (cf. \cite[Lemma 5.2]{RW22}) that the forgetful functor 
		\begin{equation}\label{eq3}
			\mathrm{Perv}_{\mathcal{IW},\mathbb{G}_\mathrm{m}}(\mathrm{Gr},\mathbf{k})\to \mathrm{Perv}_{\mathcal{IW}}(\mathrm{Gr},\mathbf{k})
		\end{equation}
		is an equivalence of categories. These considerations will become useful in section \ref{smith section}.
	\end{rem}
	\subsection{Fixed points of the affine Grassmannian and connected components}\label{Fixed points of the affine Grassmannian and connected components} As in subsection \ref{Notations}, we fix an integer $n\geq1$. We denote by $\mathbf{\mu}_n$ the finite group scheme of $n$-th roots of unity, which acts on $L^+G$ and $LG$ by rescaling the indeterminate $z$; in particular, $\mu_n$ acts on $\mathrm{Gr}$. The following fact, which will only be used with $n=\ell$, is one of the fundamental tools used by the authors of \cite{RW22} (cf. \cite[Proposition 4.7]{RW22}). For any $\lambda\in\overline{\mathbf{a}_n}\cap\mathbb{X}^\vee$, we will denote by $\mathbf{g}_\lambda\subset\overline{\mathbf{a}_n}$ the facet which contains $\lambda$.
	\begin{prop}\label{connex}
		For any $\lambda\in\overline{\mathbf{a}_n}\cap\mathbb{X}^\vee$, the map $g\mapsto g\cdot [\lambda]$ factors through an open and closed embedding 
		$$\mathrm{Fl}^{n,\circ}_{\mathbf{g}_\lambda}\hookrightarrow (\mathrm{Gr})^{\mu_n}$$ and the induced map
		$$\bigsqcup_{\lambda\in \overline{\mathbf{a}_n}\cap\mathbb{X}^\vee}\mathrm{Fl}^{n,\circ}_{\mathbf{g}_\lambda}\to (\mathrm{Gr})^{\mu_n}$$
		is an isomorphism of ind-schemes.\end{prop}
	
	If we denote by $\mathrm{Iw}_{u,\ell}^+$ the inverse image of $U^+$ under the evaluation map $L_\ell^+G\to G,~z^\ell\mapsto 0$, then we get $(\mathrm{Iw}_u^+)^{\mu_\ell}=\mathrm{Iw}_{u,\ell}^+$, and the orbits of $(\mathrm{Gr})^{\mu_\ell}$ under the action of $\mathrm{Iw}_{u,\ell}^+$ are still parametrized by $\mathbb{X}^\vee$ (cf. \cite[Lemma 4.8]{RW22}). Thus, using the morphism $\mathrm{Iw}_{u,\ell}^+\to \mathbb{G}_\mathrm{a}$ induced by $\chi$, we can take back the constructions of subsection \ref{IW variant} to define the category $D^b_{\mathcal{IW}_\ell}(Y,\mathbf{k})$, where $Y\subset (\mathrm{Gr})^{\mu_\ell}$ is a finite locally closed union of $\mathrm{Iw}_{u,\ell}^+$-orbits. The theory of parity complexes also adapts to the present context: a complex $\mathcal{F}\in D^b_{\mathcal{IW}_\ell}(Y,\mathbf{k})$ is called $*$-even, resp. $!$-even, if for any\footnote{Notice that, if $\lambda\notin\mathbb{X}^\vee_{++}$, the restriction and co-restriction of $\mathcal{F}$ to $(Y_\lambda)^{\mu_\ell}$ is zero (because this orbit supports an Iwahori-Whittaker local system only when $\lambda\in\mathbb{X}^\vee_{++}$), so that we can restrict ourselves to the case where $\lambda\in\mathbb{X}^\vee_{++}$.} $\lambda\in\mathbb{X}^\vee_{++}$ such that $(Y_\lambda)^{\mu_\ell}\subset Y$ the complex $(j_{\mu_\ell}^\lambda)^*\mathcal{F}$, resp. $(j_{\mu_\ell}^\lambda)^!\mathcal{F}$, is concentrated in even degrees, where $j_{\mu_\ell}^\lambda:(Y_\lambda)^{\mu_\ell}\hookrightarrow Y$ is the inclusion. The definition for $*$-odd and $!$-odd complexes is similar and one says that a complex is even, resp. odd, if it is both $*$-even and $!$-even, resp. $*$-odd and $!$-odd. As usual, a complex is called parity if it is a direct sum of even and odd complexes. 
	
	Defining the category $D^b_{\mathcal{IW}_\ell}((\mathrm{Gr})^{\mu_\ell},\mathbf{k})$ (resp. $D^b_{\mathcal{IW}_\ell}(\mathrm{Fl}^{\ell,\circ}_{\mathbf{g}},\mathbf{k})$ for a facet $\mathbf{g}\subset\overline{\mathbf{a}_\ell}$) as a direct limit, it then makes sense to talk about even and odd objects in this category, and we will denote by $\mathrm{Par}_{\mathcal{IW}_\ell}((\mathrm{Gr})^{\mu_\ell},\mathbf{k})$ (resp. $\mathrm{Par}_{\mathcal{IW}_\ell}(\mathrm{Fl}^{\ell,\circ}_{\mathbf{g}},\mathbf{k})$) the additive full subcategory consisting of parity objects. By Proposition \ref{connex}, it is clear that the additive category $\mathrm{Par}_{\mathcal{IW}_\ell}((\mathrm{Gr})^{\mu_\ell},\mathbf{k})$ splits into a direct sum of subcategories of the form $\mathrm{Par}_{\mathcal{IW}_\ell}(\mathrm{Fl}^{\ell,\circ}_{\mathbf{g}},\mathbf{k})$, where $\mathbf{g}$ runs through the facets inside $\overline{\mathbf{a}_\ell}$. 
	
	It is important to note that we have a bijection $\mathbf{g}\mapsto \ell\cdot\mathbf{g}$ between facets inside $\overline{\mathbf{a}_1}$ associated with the action $\car_1$ of $W$ on $E$ and facets inside $\overline{\mathbf{a}_\ell}$ associated with the action $\car_\ell$ of $W$ on $E$, and that a simple change of variable (namely, replacing $z$ with $z^\ell$) induces canonical isomorphisms of ind-schemes $\mathrm{Fl}^{\circ}_{\mathbf{g}}:=\mathrm{Fl}^{1,\circ}_{\mathbf{g}}\simeq \mathrm{Fl}^{\ell,\circ}_{\ell\cdot\mathbf{g}}$, $\mathrm{Iw}_{u}^+\simeq \mathrm{Iw}_{u,\ell}^+$, sending $\mathrm{Iw}_{u}^+$-orbits in $\mathrm{Fl}^{\circ}_{\mathbf{g}}$ to $\mathrm{Iw}_{u,\ell}^+$-orbits in $\mathrm{Fl}^{\ell,\circ}_{\ell\cdot\mathbf{g}}$. For any facet $\mathbf{g}\subset \overline{\mathbf{a}_1}$, we thus have a canonical equivalence of categories
	\begin{equation*}
		D^b_{\mathcal{IW}}(\mathrm{Fl}^{\circ}_{\mathbf{g}},\mathbf{k})\simeq D^b_{\mathcal{IW}_\ell}(\mathrm{Fl}^{\ell,\circ}_{\ell\cdot\mathbf{g}},\mathbf{k}),
	\end{equation*}
	restricting to an equivalence between the corresponding categories of parity complexes:
	\begin{equation}\label{parity l vs 1}
		\mathrm{Par}_{\mathcal{IW}}(\mathrm{Fl}^{\circ}_{\mathbf{g}},\mathbf{k})\simeq \mathrm{Par}_{\mathcal{IW}_\ell}(\mathrm{Fl}^{\ell,\circ}_{\ell\cdot\mathbf{g}},\mathbf{k}).
	\end{equation}
	In the sequel, we will denote by $\mathcal{E}_{\ell,w}^{\ell\cdot\mathbf{g}}$ the indecomposable parity complex in $\mathrm{Par}_{\mathcal{IW}_\ell}(\mathrm{Fl}^{\ell,\circ}_{\ell\cdot\mathbf{g}},\mathbf{k})$ corresponding to $\mathcal{E}^{\mathbf{g}}_w$ via the equivalence (\ref{parity l vs 1}). It is clear that, up to a shift, all of the indecomposable parity complexes in $\mathrm{Par}_{\mathcal{IW}_\ell}(\mathrm{Fl}^{\ell,\circ}_{\ell\cdot\mathbf{g}},\mathbf{k})$ (coming from Proposition \ref{par prop}) arise in this way.
	
	\subsection{Smith category and the linkage principle}
	Since the action of $\mathbb{G}_\mathrm{m}$ on $\mathrm{Gr}$ stabilizes the fixed points $(\mathrm{Gr})^{\mu_\ell}$, we can take back the construction recalled in Remark \ref{rem1} to define the category $D^b_{\mathcal{IW}_\ell,\mathbb{G}_\mathrm{m}}(Y,\mathbf{k})$, where $Y\subset (\mathrm{Gr})^{\mu_\ell}$ is a locally closed finite union of $\mathrm{Iw}_{u,\ell}^+$-orbits. We will denote by $D^b_{\mathcal{IW}_\ell,\mathbb{G}_\mathrm{m}}(Y,\mathbf{k})_{\mu_\ell-\text{perf}}$ the full subcategory whose objects are the $\mathcal{F}$ for which the object $\mathrm{Res}^{\mathbb{G}_\mathrm{m}}_{\mu_\ell}(\mathcal{F})$ has perfect geometric stalks in the sense of \cite[§3.3]{RW22}. The \textit{Smith category} $\mathrm{Sm}_{\mathcal{IW}}(Y,\mathbf{k})$ on $Y$ is by definition the Verdier quotient
	$$D^b_{\mathcal{IW}_\ell,\mathbb{G}_\mathrm{m}}(Y,\mathbf{k})/D^b_{\mathcal{IW}_\ell,\mathbb{G}_\mathrm{m}}(Y,\mathbf{k})_{\mu_\ell-\text{perf}}. $$ 
	
	Let $X\subset\mathrm{Gr}$ be a locally closed finite union of $\mathrm{Iw}_u^+$-orbits, and $i_X:(X)^{\mu_\ell}\hookrightarrow X$ denote the inclusion. We define the functor $$i_X^{!*}:\mathrm{Perv}_{\mathcal{IW},\mathbb{G}_\mathrm{m}}(X,\mathbf{k})\to \mathrm{Sm}_{\mathcal{IW}}((X)^{\mu_\ell},\mathbf{k})$$ as the composition of the inverse image $i_X^*$ with the canonical quotient map $$Q:D^b_{\mathcal{IW}_\ell,\mathbb{G}_\mathrm{m}}((X)^{\mu_\ell},\mathbf{k})\to \mathrm{Sm}_{\mathcal{IW}}((X)^{\mu_\ell},\mathbf{k}).$$ A crucial result says that taking $i_X^!$ instead of $i_X^*$ in the previous construction gives isomorphic functors (cf. \cite[§6.2]{RW22}), whence the notation. It then takes a bit more work (cf. \cite[Lemma 6.1]{RW22}) to prove that, for two locally closed finite unions of $\mathrm{Iw}_{u,\ell}^+$-orbits $Z\subset Y$, the canonical functor $f^{\mathrm{Sm}}_*:\mathrm{Sm}_{\mathcal{IW}}(Z,\mathbf{k})\to \mathrm{Sm}_{\mathcal{IW}}(Y,\mathbf{k})$ induced by the direct image in the corresponding derived categories is fully faithful, and fits into the following commutative diagram
	$$\xymatrix{
		D^b_{\mathcal{IW}_\ell,\mathbb{G}_\mathrm{m}}(Z,\mathbf{k})\ar[r]^{f_*} \ar[d] &  D^b_{\mathcal{IW}_\ell,\mathbb{G}_\mathrm{m}}(Y,\mathbf{k})  \ar[d] \\
		\mathrm{Sm}_{\mathcal{IW}}(Z,\mathbf{k}) \ar[r]^{f^{\mathrm{Sm}}_*} &  \mathrm{Sm}_{\mathcal{IW}}(Y,\mathbf{k}),
	}$$
	where the vertical arrows are the quotient maps. One can thus define the category $\mathrm{Sm}_{\mathcal{IW}}((\mathrm{Gr})^{\mu_\ell},\mathbf{k})$ as a direct limit indexed by finite closed unions of $\mathrm{Iw}_{u,\ell}^+$-orbits, and consider the functor
	$$i^{!*}_{\mathrm{Gr}}:\mathrm{Perv}_{\mathcal{IW},\mathbb{G}_\mathrm{m}}(\mathrm{Gr},\mathbf{k})\to \mathrm{Sm}_{\mathcal{IW}}((\mathrm{Gr})^{\mu_\ell},\mathbf{k}).$$
	
	The following statement is \cite[Theorem 7.4]{RW22}.
	\begin{thm}\label{equivalence1}
		The composition of functors
		$$\mathrm{Perv}_{\mathcal{IW}}(\mathrm{Gr},\mathbf{k})\xrightarrow{\text{Remark}~ \ref{rem1}}\mathrm{Perv}_{\mathcal{IW},\mathbb{G}_\mathrm{m}}(\mathrm{Gr},\mathbf{k})\xrightarrow{i^{!*}_{\mathrm{Gr}}}\mathrm{Sm}_{\mathcal{IW}}((\mathrm{Gr})^{\mu_\ell},\mathbf{k})$$
		restricts to a fully faithful functor $\Phi:\mathrm{Tilt}_{\mathcal{IW}}(\mathrm{Gr},\mathbf{k})\to \mathrm{Sm}_{\mathcal{IW}}((\mathrm{Gr})^{\mu_\ell},\mathbf{k})$.
	\end{thm}

	We can now reformulate the proof of the linkage principle from \cite[Theorem 8.5]{RW22}. Thanks to the decomposition of $(\mathrm{Gr})^{\mu_\ell}$ into its connected components from Proposition \ref{connex}, we deduce from Theorem \ref{equivalence1} above that two strictly dominant weights $\lambda$ and $\mu$ which are not in the same orbit for the box action cannot be in relation for $\mathscr{R}_2$ (seen as a relation on the weight poset $\mathbb{X}^\vee_{++}$ of the highest weight category $\mathrm{Perv}_{\mathcal{IW}}(\mathrm{Gr},\mathbf{k})$, cf. section \ref{relation section}). Indeed, let $\delta$, resp. $\gamma$, be the element of $W\car_\ell\lambda\cap\overline{\mathbf{a}_\ell}$, resp. of $W\car_\ell\mu\cap\overline{\mathbf{a}_\ell}$. We have
	$$\mathrm{Hom}_{\mathrm{Tilt}_{\mathcal{IW}}(\mathrm{Gr},\mathbf{k})}(\mathscr{T}_\lambda^{\mathcal{IW}},\mathscr{T}_\mu^{\mathcal{IW}})\simeq \mathrm{Hom}_{\mathrm{Sm}_{\mathcal{IW}}((\mathrm{Gr})^{\mu_\ell},\mathbf{k})}(\Phi(\mathscr{T}_\lambda^{\mathcal{IW}}),\Phi(\mathscr{T}_\mu^{\mathcal{IW}})),$$
	and since the object $\Phi(\mathscr{T}_\lambda^{\mathcal{IW}})$, resp. $\Phi(\mathscr{T}_\mu^{\mathcal{IW}})$, is indecomposable (thanks to the full faithfulness of $\Phi$), its support must be contained in a single connected component of $(\mathrm{Gr})^{\mu_\ell}$, which is easily seen to be the parametrized by $\delta$, resp.  $\gamma$, in the isomorphism of Proposition \ref{connex} applied to the case $n=\ell$.
	
	Thanks to Theorem \ref{IW equivalence}, this implies that two dominant weights $\lambda',\mu'$ which are not in the same orbit for the dot action cannot be in relation for $\mathscr{R}_2$ (where we now consider the highest weight category $\mathrm{Perv}_{L^+G}(\mathrm{Gr},\mathbf{k})$). Finally, the geometric Satake equivalence (Theorem \ref{satake equivalence}) and Theorem \ref{relations thm} allow to prove the linkage principle for $G^\vee$ (cf. \cite[Corollary 6.17, Part II]{jantzen2003representations}):
	\begin{equation}\label{linkage principle}
		\forall \lambda',\mu'\in \mathbb{X}^\vee_{+}~\text{such that}~W\bullet_{\ell}\lambda'\neq W\bullet_{\ell}\mu',~\text{we have}~\mathrm{Ext}_{G^\vee}^1(L_{\lambda'},L_{\mu'})=0. 
	\end{equation}

	\subsection{Consequences on equivalence relations}\label{Consequences on equivalence relations}
			Recall the definition of the ``dot'' action of $W$, which acts on $\mathbb{X}^\vee$ by $$w\bullet_{\ell}\mu:=w\car_\ell (\mu+\rho^\vee)-\rho^\vee,$$ for any $w\in W$, $\mu\in\mathbb{X}^\vee$. Let $\mu\in\mathbb{X}^\vee\cap\overline{\mathbf{a}_\ell}$ and $\mathbf{g}_\mu$ be the facet (for $\car_\ell$) containing $\mu$. Therefore, $\ell^{-1}\cdot\mathbf{g}_\mu\subset\overline{\mathbf{a}_1}$ is a facet for $\car_1$. The assignment $w\mapsto w\car_\ell\mu$, resp. $w\mapsto w\bullet_\ell(\mu-\rho^\vee)$, induces a bijection 
	\begin{equation}\label{param dom weights}
	    {_\mathrm{f}W}^{\ell^{-1}\cdot\mathbf{g}_{\mu}}\xrightarrow{\sim}W\car_\ell\mu\cap\mathbb{X}^\vee_{++}, ~\text{resp.}~{_\mathrm{f}W}^{\ell^{-1}\cdot\mathbf{g}_\mu}\xrightarrow{\sim}W\bullet_\ell(\mu-\rho^\vee)\cap\mathbb{X}^\vee_{+}
	\end{equation} 
	(this follows from the fact that, if we denote by $W_{\mathbf{g}_\mu,\car_\ell}\subset W$ the stabilizer of $\mathbf{g}_\mu$ for $\car_\ell$, then we have an equality $W_{\mathbf{g}_\mu,\car_\ell}=W_{\ell^{-1}\cdot\mathbf{g}_\mu}$).
	
	In the sequel, we will want to determine an exact description of the blocks in $\mathbb{X}^\vee_{++}$ (seen as the weight poset of $\mathrm{Perv}_{\mathcal{IW}}(\mathrm{Gr},\mathbf{k})$) for $\mathscr{R}_2$ (cf. subsection \ref{Equivalence relations on Lambda}). The following results will help us to do so (cf. the proof of \cite[Theorem 8.9]{RW22}). 
	\begin{prop}\label{iso parity} Let $\mu\in\mathbb{X}^\vee\cap\overline{\mathbf{a}_\ell}$, $\mathbf{g}_\mu$ be the facet (for $\car_\ell$) containing $\mu$, and  $w,w'$ be elements of ${_\mathrm{f}W}^{\ell^{-1}\cdot\mathbf{g}_\mu}$. We have an isomorphism
		$$\mathrm{Hom}_{\mathrm{Tilt}_\mathcal{IW}(\mathrm{Gr},\mathbf{k})}(\mathscr{T}^\mathcal{IW}_{w\car_\ell\mu},\mathscr{T}^\mathcal{IW}_{w'\car_\ell\mu})\simeq \mathrm{Hom}^\bullet_{\mathrm{Par}_{\mathcal{IW}_\ell}(\mathrm{Fl}^{\ell,\circ}_{\mathbf{g}_\mu},\mathbf{k})}(\mathcal{E}_{\ell,w}^{\mathbf{g}_\mu},\mathcal{E}_{\ell,w'}^{\mathbf{g}_\mu}).$$
	\end{prop}
	\begin{proof}
		This is a direct consequence of \cite[Proposition 8.11]{RW22} and of the equivalence of highest weight categories $\mathrm{Rep}_{\mathbf{k}}(G^\vee)\simeq \mathrm{Perv}_{\mathcal{IW}}(\mathrm{Gr},\mathbf{k})$, which sends the indecomposable tilting module of highest weight $w\bullet_\ell (\mu-\rho^\vee)$, resp. $w'\bullet_\ell (\mu-\rho^\vee)$, to $\mathscr{T}^\mathcal{IW}_{w\car_\ell\mu}$, resp. $\mathscr{T}^\mathcal{IW}_{w'\car_\ell\mu}$.
	\end{proof}
	Let $\lambda\in\overline{\mathbf{a}_\ell}\cap\mathbb{X}^\vee$. Note that, thanks to the equivalence (\ref{parity l vs 1}), we have 
	$$\mathrm{Hom}^\bullet_{\mathrm{Par}_{\mathcal{IW}}(\mathrm{Fl}^{\circ}_{\ell^{-1}\cdot\mathbf{g}_\lambda},\mathbf{k})}(\mathcal{E}^{\ell^{-1}\cdot\mathbf{g}_\lambda}_{w},\mathcal{E}^{\ell^{-1}\cdot\mathbf{g}_\lambda}_{w'})\simeq \mathrm{Hom}^\bullet_{\mathrm{Par}_{\mathcal{IW}_\ell}(\mathrm{Fl}^{\ell,\circ}_{\mathbf{g}_\lambda},\mathbf{k})}(\mathcal{E}^{\mathbf{g}_\lambda}_{\ell,w},\mathcal{E}^{\mathbf{g}_\lambda}_{\ell,w'}). $$
	Therefore Proposition \ref{iso parity} yields
	\begin{equation}\label{eq relation W}
		\forall w,w'\in {_\mathrm{f}W}^{\ell^{-1}\cdot\mathbf{g}_\lambda},\qquad	w \mathscr{R}_{\ell^{-1}\cdot\mathbf{g}_\lambda} w'\Longleftrightarrow (w\car_\ell\lambda)\mathscr{R}_2 (w'\car_\ell\lambda).
	\end{equation}
	\subsection{Dilating weights by $\ell$}
	Smith-Treumann theory allows us to understand the effect of dilating dominant weights by $\ell$ on the equivalence relation $\sim$.
	\begin{prop}\label{red to non special}
		Let $\lambda,\mu\in\mathbb{X}^\vee_{++}$. We have an isomorphism
		$$\mathrm{Hom}_{\mathrm{Tilt}_\mathcal{IW}(\mathrm{Gr},\mathbf{k})}(\mathscr{T}^\mathcal{IW}_{\ell\cdot\lambda},\mathscr{T}^\mathcal{IW}_{\ell\cdot\mu})\simeq \mathrm{Hom}_{\mathrm{Tilt}_\mathcal{IW}(\mathrm{Gr},\mathbf{k})}(\mathscr{T}^\mathcal{IW}_{\lambda},\mathscr{T}^\mathcal{IW}_{\mu}).$$
	\end{prop}
	\begin{proof}
	    First recall that, by full-faithfulness of $\Phi$, we have an isomorphism
	    \begin{equation}\label{fulfaith1}
	        \mathrm{Hom}_{\mathrm{Tilt}_\mathcal{IW}(\mathrm{Gr},\mathbf{k})}(\mathscr{T}^\mathcal{IW}_{\ell\cdot\lambda},\mathscr{T}^\mathcal{IW}_{\ell\cdot\mu})\simeq\mathrm{Hom}_{\mathrm{Sm}_{\mathcal{IW}}((\mathrm{Gr})^{\mu_\ell},\mathbf{k})}(\Phi(\mathscr{T}^\mathcal{IW}_{\ell\cdot\lambda}),\Phi(\mathscr{T}^\mathcal{IW}_{\ell\cdot\mu})). 
	    \end{equation}
	    Next, notice that we have an embedding $$L_\ell G/L^+_\ell G\hookrightarrow (\mathrm{Gr})^{\mu_\ell}$$ thanks to \cite[Remark 4.8]{RW22}, and that the left-hand side identifies (via Proposition \ref{connex}) with the union of the connected components $\mathrm{Fl}^{\ell,\circ}_{\mathbf{g}_\lambda}$, with $\lambda$ running through $\overline{\mathbf{a}_\ell}\cap\ell\cdot\mathbb{X}^\vee$.  Moreover, the $L^+_\ell G$-orbits in $L_\ell G/L^+_\ell G$ are parametrized by $\ell\cdot\mathbb{X}_+^\vee$, and a simple change of variable $z\mapsto z^\ell$ together with (\ref{eq3}) induce an equivalence of categories $$R:\mathrm{Tilt}_{\mathcal{IW}}(\mathrm{Gr},\mathbf{k})\simeq \mathrm{Tilt}_{\mathcal{IW}_\ell,\mathbb{G}_\mathrm{m}}(L_\ell G/L^+_\ell G,\mathbf{k}),$$
	    which sends $\mathscr{T}^\mathcal{IW}_{\lambda}$ (resp. $\mathscr{T}^\mathcal{IW}_{\mu}$) to the indecomposable tilting object associated with $\ell\cdot\lambda$ (resp. $\ell\cdot\mu$).  Therefore, the indecomposable tilting objects in $\mathrm{Tilt}_{\mathcal{IW}_\ell,\mathbb{G}_\mathrm{m}}(L_\ell G/L^+_\ell G,\mathbf{k})$ are parity complexes (see \cite[Proposition 4.12]{JEP_2019__6__707_0}), so that we can take back the same arguments as in the proof of \cite[Theorem 7.4]{RW22} to show that restricting
	    the functor 
	    $$Q:D^b_{\mathcal{IW}_\ell,\mathbb{G}_\mathrm{m}}((\mathrm{Gr})^{\mu_\ell},\mathbf{k})\to \mathrm{Sm}_{\mathcal{IW}}((\mathrm{Gr})^{\mu_\ell},\mathbf{k})$$
	    to $L_\ell G/L^+_\ell G$ and composing it with $R$ yields a fully-faithful functor  $$Q\circ R:\mathrm{Tilt}_{\mathcal{IW}}(\mathrm{Gr},\mathbf{k})\to \mathrm{Sm}_{\mathcal{IW}}(L_\ell G/L^+_\ell G,\mathbf{k}), $$ which satisfies the following isomorphisms:
	    $$Q\circ R(\mathscr{T}^\mathcal{IW}_{\lambda})\simeq \Phi(\mathscr{T}^\mathcal{IW}_{\ell\cdot\lambda})~\text{and}~Q\circ R(\mathscr{T}^\mathcal{IW}_{\mu})\simeq \Phi(\mathscr{T}^\mathcal{IW}_{\ell\cdot\mu}).$$
	    The desired result thus follows from comparing (\ref{fulfaith1}) with the following isomorphism 
	    \begin{equation*}\label{fulfaith2}
	        \mathrm{Hom}(\mathscr{T}^\mathcal{IW}_{\lambda},\mathscr{T}^\mathcal{IW}_{\mu})\simeq\mathrm{Hom}(Q\circ R(\mathscr{T}^\mathcal{IW}_{\lambda}),Q\circ R(\mathscr{T}^\mathcal{IW}_{\mu})). 
	    \end{equation*}
	\end{proof}
 In the next corollary, we consider the equivalence relation $\sim$ on the set $\mathbb{X}^\vee_{++}$, seen as the weight poset of the highest weight category $\mathrm{Perv}_{\mathcal{IW}}(\mathrm{Gr},\mathbf{k})$ (cf. subsection \ref{Equivalence relations on Lambda}).
	\begin{coro}\label{inclusion classes}
	For any $\gamma\in\mathbb{X}^\vee_{++}$, denote by $\overline{\gamma}$ the equivalence class of $\gamma$ for the equivalence relation $\sim$. Then we have $$\forall \lambda\in \mathbb{X}^\vee_{++},\quad\ell\cdot\overline{\lambda}\subset\overline{\ell\cdot\lambda}.$$
	\end{coro}
	\begin{proof}
	 This is a direct consequence of Proposition \ref{red to non special} and of the fact that $\sim$ is generated by $\mathscr{R}_2$ (cf. section \ref{relation section}).
	\end{proof}
	\begin{coro}\label{red to non special2}
	Let $\lambda,\mu\in\mathbb{X}^\vee_{++}$. We have an equality 
	$$(\mathscr{T}^\mathcal{IW}_{\lambda}:\nabla^\mathcal{IW}_{\mu})=(\mathscr{T}^\mathcal{IW}_{\ell\cdot \lambda}:\nabla^\mathcal{IW}_{\ell\cdot\mu} ).$$
	\end{coro}
	\begin{proof}
	For any $\gamma\in\mathbb{X}^\vee_{++}$, let us denote by $n(\gamma)$ the number of elements $\nu\in\mathbb{X}^\vee_{++}$ such that $\nu< \gamma$.  We may and will assume that $\mu\leq \lambda$ to prove the desired equality, and proceed by induction on $n(\ell\cdot\mu)$. Standard arguments on highest weight categories (and the fact that Verdier duality sends $\Delta^\mathcal{IW}_\nu$, resp. $\mathscr{T}^\mathcal{IW}_{\nu}$, to $\nabla^\mathcal{IW}_{\nu}$, resp. $\mathscr{T}^\mathcal{IW}_{\nu}$, for all $\nu\in\mathbb{X}^\vee_{++}$; namely, we use the same arguments as the one used in \cite[§6.2]{achar2016modular}) imply that we have 
	\begin{multline}\label{induction multiplicities}
	    (\mathscr{T}^\mathcal{IW}_{\lambda}:\nabla^\mathcal{IW}_{\mu})=\mathrm{dim}_\mathbf{k}\mathrm{Hom}(\mathscr{T}^\mathcal{IW}_{\lambda},\mathscr{T}^\mathcal{IW}_{\mu})-\\ \sum_{\nu<\mu,~\nu\in\mathbb{X}^\vee_{++}}(\mathscr{T}^\mathcal{IW}_{\lambda}:\nabla^\mathcal{IW}_{\nu})\cdot (\mathscr{T}^\mathcal{IW}_{\mu}:\nabla^\mathcal{IW}_{\nu}).
	\end{multline}
	Now when $n(\ell\cdot\mu)=0$, we get that $n(\mu)=0$ and so
	\begin{align*}
	    (\mathscr{T}^\mathcal{IW}_{\lambda}:\nabla^\mathcal{IW}_{\mu})&=\mathrm{dim}_\mathbf{k}\mathrm{Hom}(\mathscr{T}^\mathcal{IW}_{\lambda},\mathscr{T}^\mathcal{IW}_{\mu})\\ &=\mathrm{dim}_\mathbf{k}\mathrm{Hom}(\mathscr{T}^\mathcal{IW}_{\ell\cdot\lambda},\mathscr{T}^\mathcal{IW}_{\ell\cdot\mu})\\ &=(\mathscr{T}^\mathcal{IW}_{\ell\cdot\lambda}:\nabla^\mathcal{IW}_{\ell\cdot\mu}), 
	\end{align*}
	where the second equality is due to Proposition \ref{red to non special}. This proves the first step of the induction. The induction step still follows from using (\ref{induction multiplicities}) and Proposition \ref{red to non special}, together with the fact that  for any $\nu\in\mathbb{X}^\vee_{++}$ such that $(\mathscr{T}^\mathcal{IW}_{\ell\cdot\mu}:\nabla^\mathcal{IW}_{\nu})\neq 0$, we have that $\nu\in\ell\cdot\mathbb{X}^\vee_{++}$ (this is because we must have $\nu\in W\car_\ell(\ell\cdot\mu)$ by the linkage principle).
	\end{proof}
	
	\subsection{The general case for equivalence relations on ${_\mathrm{f}W}^\mathbf{g}$}\label{section special}
	In this subsection, we apply the equivalence (\ref{eq relation W}) and Proposition \ref{red to non special} to finish the study of section \ref{Determination of the blocks} by treating the case of special facets. The arguments that we have used until now do not apply when $\mathbf{g}$ is a special facet (notice that Corollary \ref{lem boite} is false in this case). With good reason: we are going to see that the set ${_\mathrm{f}W}^\mathbf{g}$ splits into infinitely many classes when some component of $\mathbf{g}$ is a special facet. As in subsections \ref{Facets which are not points} and \ref{Facets which are points}, we are first going to deal with the case where the root system is irreducible before generalizing. 
	
	Let $\mathbf{g}\subset\overline{\mathbf{a}_1}$ be a facet, $w\in {_\mathrm{f}W}^\mathbf{g}$, and recall that if $\mathbf{g}$ special, then $w\mathbf{g}$ can be written as $w\mathbf{g}=\{\lambda\}$ for a unique coweight $\lambda$. We will denote by $\overline{w}$ the equivalence class of $w$ for $\sim_\mathbf{g}$, and define $r_\mathbf{g}(w)$ to be equal to $-1$ if $\mathbf{g}$ is \textit{not} special, and to be the unique (non-negative) integer such that $$\lambda\in \ell^{r_\mathbf{g}(w)}\mathbb{X}^\vee\backslash~ \ell^{r_\mathbf{g}(w)+1}\mathbb{X}^\vee,~\text{with}~w\mathbf{g}=\{\lambda\}$$ if $\mathbf{g}$ is special. For any positive integer $r$, we also define the group $W^{(r)}:=W_0\ltimes \ell^r\mathbb{Z}\mathfrak{R}^\vee$, which is a subgroup of $W$ isomorphic to it, with $W^{(0)}=W$. Finally, notice that for any $\lambda\in\mathbb{X^\vee}$ and positive integer $r$, we have $\ell^r\cdot W\car_\ell\lambda=W^{(r)}\car_\ell (\ell^r\cdot\lambda)$. We start with an easy lemma.
	\begin{lem}\label{param Wr}
	    Let $\nu\in\overline{\mathbf{a}_1}\cap\mathbb{X}^\vee$, $\mathbf{g}\subset\overline{\mathbf{a}_1}$ be the facet containing $\nu$, $w\in {_\mathrm{f}W}^\mathbf{g}$ and $r$ be a positive integer. If we set $\lambda:=w\car_1\nu$, then the application $W\to \mathbb{X}^\vee,~u\mapsto u\car_\ell(\ell\cdot\nu)$ induces a bijection  $$(W^{(r)}wW_\mathbf{g})\cap {_\mathrm{f}W}^\mathbf{g}\simeq W^{(r)}\car_\ell (\ell\cdot\lambda)\cap\mathbb{X}^{\vee}_{++}.$$ 
	\end{lem}
	\begin{proof}
	 Notice that $\ell\cdot\lambda=w\car_\ell(\ell\cdot\nu)$, so that $W^{(r)}\car_\ell (\ell\cdot\lambda)=W^{(r)}w\car_\ell (\ell\cdot\nu)$. The desired isomorphism then follows from the fact that $W_\mathbf{g}$ is the stabilizer of $\ell\cdot\nu$ for $\car_\ell$, and from the well-known fact that $u\mapsto u\car_\ell (\ell\cdot\nu)$ induces a bijection 
	 $${_\mathrm{f}W}^\mathbf{g}\simeq W\car_\ell (\ell\cdot\nu)\cap\mathbb{X}^{\vee}_{++}. $$
	\end{proof}
	\begin{prop}\label{facette speciale}
		Assume that $\mathfrak{R}^\vee$ is irreducible. For any facet $\mathbf{g}\subset\overline{\mathbf{a}_1}$ and $w\in {_\mathrm{f}W}^\mathbf{g}$, we have 
		\begin{equation*}
			\overline{w}=\begin{cases}
			\{w\} & \text{if $\mathbf{g}$ is special and}\ \mathrm{char}(\mathbf{k})=0\\
			(W^{(r_\mathbf{g}(w)+1)}wW_\mathbf{g})\cap {_\mathrm{f}W}^\mathbf{g} & \text{otherwise}.
			\end{cases}
		\end{equation*}
		
	\end{prop}
	\begin{proof}
		If $\mathbf{g}$ is not special, then we have $r_\mathbf{g}(w)=-1$ by definition and $$(W^{(r_\mathbf{g}(w)+1)}wW_\mathbf{g})\cap {_\mathrm{f}W}^\mathbf{g}=W\cap {_\mathrm{f}W}^\mathbf{g}={_\mathrm{f}W}^\mathbf{g},$$ which does coincide with $\overline{w}$ thanks to Theorem \ref{cas general}. So we assume from now on that $\mathbf{g}$ is a special facet. We will first treat the case where $\mathrm{char}(\mathbf{k})=\ell$.
		
		Let $w,v\in {_\mathrm{f}W}^\mathbf{g}$, $\nu\in \mathbb{X}^\vee\cap\overline{\mathbf{a}_1}$ be such that $\mathbf{g}=\{\nu\}$, and put $w\mathbf{g}=\{\lambda\}$, $v\mathbf{g}=\{\mu\}$. Notice that  $w\car_\ell(\ell\cdot\nu)=\ell\cdot\lambda$, $v\car_\ell(\ell\cdot\nu)=\ell\cdot\mu$, and recall that, thanks to Proposition \ref{iso parity} coupled with (\ref{parity l vs 1}), we have 
		$$\mathrm{Hom}_{\mathrm{Tilt}_\mathcal{IW}(\mathrm{Gr},\mathbf{k})}(\mathscr{T}^\mathcal{IW}_{\ell\cdot\lambda},\mathscr{T}^\mathcal{IW}_{\ell\cdot\mu})\simeq \mathrm{Hom}^\bullet_{\mathrm{Par}_{\mathcal{IW}}(\mathrm{Fl}^{\circ}_{\mathbf{g}},\mathbf{k})}(\mathcal{E}_w^{\mathbf{g}},\mathcal{E}_{v}^{\mathbf{g}}),$$
		which implies that $w\mathscr{R}_\mathbf{g}v\Leftrightarrow (\ell\cdot\lambda)\mathscr{R}_2 (\ell\cdot\mu)$. Moreover, recall that by the geometric linkage principle (\cite[Theorem 8.5]{RW22}) we have \begin{equation*}
			\forall\gamma\in\mathbb{X}^\vee_{++},~(\ell\cdot\lambda)\mathscr{R}_2 \gamma\Rightarrow\gamma\in W\car_\ell (\ell\cdot\lambda),
		\end{equation*}
		so that in particular \begin{equation}\label{orbit implique l}
			(\ell\cdot\lambda)\mathscr{R}_2 \gamma\Rightarrow\gamma\in\ell\cdot\mathbb{X}^\vee_{++}.
		\end{equation} Therefore, Lemma \ref{param Wr} (applied to $r=r_\mathbf{g}(w)+1$) implies that our claim on $\overline{w}$ is equivalent to showing that $$\overline{\ell\cdot\lambda}=W^{(r_\mathbf{g}(w)+1)}\car_\ell (\ell\cdot\lambda)\cap\mathbb{X}^{\vee}_{++},$$
		where $\overline{\ell\cdot\lambda}$ denotes the equivalence class of $\ell\cdot\lambda$ for the equivalence relation $\sim$ on $\mathbb{X}^\vee_{++}$.
		
		By definition of $r_\mathbf{g}(w)$ and $r_\mathbf{g}(v)$, we know that there exist $\lambda',\mu'\in \mathbb{X}^\vee_{++}$ such that $\ell^{r_\mathbf{g}(w)}\cdot\lambda'=\lambda$ and $\ell^{r_\mathbf{g}(v)}\cdot\mu'=\mu$. Assume (without any loss of generality, up to switching the roles of $\lambda$ and $\mu$) that $r_\mathbf{g}(v)\geq r_\mathbf{g}(w)$, so that Proposition \ref{red to non special} yields
		\begin{equation}\label{div par l}
			\mathrm{Hom}_{\mathrm{Tilt}_\mathcal{IW}(\mathrm{Gr},\mathbf{k})}(\mathscr{T}^\mathcal{IW}_{\ell\cdot\lambda},\mathscr{T}^\mathcal{IW}_{\ell\cdot\mu})\simeq \mathrm{Hom}_{\mathrm{Tilt}_\mathcal{IW}(\mathrm{Gr},\mathbf{k})}(\mathscr{T}^\mathcal{IW}_{\lambda'},\mathscr{T}^\mathcal{IW}_{\ell^{r_\mathbf{g}(v)-r_\mathbf{g}(w)}\cdot\mu'}).
		\end{equation}
		If we denote by $\mathbf{g}'\subset \overline{\mathbf{a}_\ell}$ the facet (for $\car_\ell$) containing the element of $W\car_\ell\lambda'\cap\overline{\mathbf{a}_\ell}$, then $\ell^{-1}\cdot\mathbf{g}'$ is a facet (for $\car_1$) contained in $\overline{\mathbf{a}_1}$ which is \textit{not} special (because otherwise we could write $\mathbf{g}'=\{\ell\cdot\lambda''\}$ for some weight $\lambda''$, and  $\lambda'= w''\car_\ell(\ell\cdot\lambda'')=\ell\cdot( w''\car_1\lambda'')$ for some $w''\in W$, contradicting the definition of $r_\mathbf{g}(w)$). So thanks to (\ref{eq relation W}) and Theorem \ref{cas general} we get 
		$$\overline{\lambda'}=W\car_\ell\lambda'\cap\mathbb{X}^\vee_{++}. $$
		In particular, if $r_\mathbf{g}(v)>r_\mathbf{g}(w)$, then we easily see that $\ell^{r_\mathbf{g}(v)-r_\mathbf{g}(w)}\cdot\mu'\notin W\car_\ell\lambda'$, so the right-hand side in (\ref{div par l}) is non-zero only if $r_\mathbf{g}(v)=r_\mathbf{g}(w)$. This observation, together with (\ref{div par l}) and (\ref{orbit implique l}), allows us to deduce that $\overline{\ell^{r_\mathbf{g}(w)+1}\cdot\lambda'}\subset\ell^{r_\mathbf{g}(w)+1}\cdot\overline{\lambda'}$. The reversed inclusion $\ell^{r_\mathbf{g}(w)+1}\cdot\overline{\lambda'}\subset \overline{\ell^{r_\mathbf{g}(w)+1}\cdot\lambda'}$ is obtained by applying Corollary \ref{inclusion classes}. We finally get the desired equality:
		\begin{align*}
		    \overline{\ell\cdot\lambda}=\overline{\ell^{r_\mathbf{g}(w)+1}\cdot\lambda'}&=\ell^{r_\mathbf{g}(w)+1}\cdot\overline{\lambda'}\\ &=\ell^{r_\mathbf{g}(w)+1}\cdot W\car_\ell\lambda'\cap\mathbb{X}^\vee_{++}\\ &=W^{(r_\mathbf{g}(w)+1)}\car_\ell (\ell\cdot\lambda)\cap\mathbb{X}^{\vee}_{++}.
		\end{align*}
		
		Now we pass to the case $\mathrm{char}(\mathbf{k})=0$. We want to show that for any $w'\in {_\mathrm{f}W}^\mathbf{g}$ such that $w'\neq w$, we have 
		\begin{equation}\label{eq0}
			\mathrm{Hom}^\bullet_{D^b_{\mathcal{IW}}(\mathrm{Fl}^\circ_{\mathbf{g}},\mathbf{k})}(\mathcal{E}^\mathbf{g}_{w'},\mathcal{E}^\mathbf{g}_w)=0.
		\end{equation}
		We claim that there exists $\omega\in \Omega$ such that $\omega\{0\}=\mathbf{g}$. Indeed, the map $u\mapsto u\{0\}$ is actually a bijection from $\Omega$ to $\overline{\mathbf{a}_1}\cap\mathbb{X}^\vee$ thanks to the first point of the remark in \cite[Ch. VI, §2.3]{cdi_springer_books_10_1007_978_3_540_34491_9}. We then have an isomorphism $L^+P_\mathbf{g}\simeq \dot{\omega}L^+G\dot{\omega}^{-1} $, and conjugation by $\dot{\omega}$ yields an isomorphism
		$$\mathrm{Fl}^\circ_{\{0\}}\simeq \mathrm{Fl}^\circ_{\mathbf{g}}. $$
		Since $\dot{\omega}$ belongs to $\Omega$, conjugation by $\dot{\omega}$ preserves $\mathrm{Iw}_u^+$, so that we get an equivalence of categories \begin{equation}\label{preuve1}
		    D^b_{\mathcal{IW}}(\mathrm{Fl}^\circ_{\mathbf{g}},\mathbf{k})\simeq D^b_{\mathcal{IW}}(\mathrm{Fl}^\circ_{\{0\}},\mathbf{k}).
		\end{equation} This equivalence implies that we only need to prove (\ref{eq0}) with $\mathbf{g}$ replaced by $\{0\}$. Now, recall that $\mathrm{Fl}_{\{0\}}\simeq\mathrm{Gr}$ and let $\lambda,\lambda'\in\mathbb{X}^\vee_{++}$ be such that $w\mathbf{g}=\{\lambda\}$, $w'\mathbf{g}=\{\lambda'\}$. By \cite[Remark 3.5]{JEP_2019__6__707_0}, the complex $\mathcal{E}_w$ is isomorphic to $\mathrm{IC}^{\mathcal{IW}}_{\lambda}$ when $\mathrm{char}(\mathbf{k})=0$ (which denotes the intersection cohomology complex on $\mathrm{Gr}$ associated with $\lambda$, cf. subsection \ref{IW variant}). Therefore in characteristic zero we have an isomorphism 
		$$\mathrm{Hom}^\bullet_{D^b_{\mathcal{IW}}(\mathrm{Fl}^\circ_{\{0\}},\mathbf{k})}(\mathcal{E}_{w'},\mathcal{E}_w)\simeq \mathrm{Hom}^\bullet_{D^b_{\mathcal{IW}}(\mathrm{Gr},\mathbf{k})}(\mathrm{IC}^{\mathcal{IW}}_{\lambda'},\mathrm{IC}^{\mathcal{IW}}_{\lambda}).$$
		But we know from \cite[§3.2]{JEP_2019__6__707_0} that we have an equivalence of triangulated categories $D^b_{\mathcal{IW}}(\mathrm{Gr},\mathbf{k})\simeq D^b\mathrm{Perv}_{\mathcal{IW}}(\mathrm{Gr},\mathbf{k})$, and that $\mathrm{Perv}_{\mathcal{IW}}(\mathrm{Gr},\mathbf{k})$ is semi-simple when $\mathrm{char}(\mathbf{k})=0$ thanks to \cite[Corollary 3.6]{JEP_2019__6__707_0}. Therefore the right-hand side above is zero when $\lambda\neq\lambda'$ and $\mathrm{char}(\mathbf{k})=0$, which is equivalent to $w\neq w'$. \footnote{Instead of using the equivalence (\ref{preuve1}), we could have argued by saying that $\mathrm{Fl}^\circ_{\mathbf{g}}$ is a connected component of $\mathrm{Gr}$.}\end{proof}
	This last result coupled with Proposition \ref{reduction to irred} now enables us to deal with the case where $\mathfrak{R}^\vee$ splits into irreducible root systems, inducing decompositions $W=W_1\times\cdots\times W_r$ and $\mathbf{g}=\mathbf{g}_1\times\cdots\times\mathbf{g}_r$ (see the beginning of subsection \ref{Facets which are not points}).
	\begin{thm}\label{cas final}
		Assume that $\mathrm{char}(\mathbf{k})=\ell$. For any facet $\mathbf{g}\subset\overline{\mathbf{a}_1}$ and $w=(w_1,\cdots,w_r)\in {_\mathrm{f}W}^\mathbf{g}$, we have
		\begin{equation*}
			\overline{w}=\prod_{i=1}^rW_i^{(r(w_i)+1)}w_iW_{\mathbf{g}_i}\cap {_\mathrm{f}W_i}^{\mathbf{g}_i}.
		\end{equation*}
		When $\mathrm{char}(\mathbf{k})=0$ and $r(w_i)\geq 0$ for some $i$, one replaces the $i$'th component in the above product by $\{w_i\}$.
	\end{thm}
	
	\section{Applications to representation theory}\label{Applications to representation theory}
	\subsection{Preliminaries on alcoves}
	Just as in subsection \ref{geometry}, the action $\bullet_\ell$ of $W$ on $E$ defines a hyperplane arrangement $\mathscr{H}'$, where the shift by $\rho^\vee$ induces a bijection between $\mathscr{H}'$ and $\mathscr{H}$. We will call a connected component of $E\backslash \mathscr{H}'$ an alcove for $\bullet_\ell$, and a facet contained in the closure of an alcove for $\bullet_\ell$ will be called a facet for $\bullet_\ell$; once again shifting by $\rho^\vee$ induces a bijection between facets for $\bullet_\ell$ and facets for $\car_\ell$, and composing this bijection with the dilation by $\ell^{-1}$ gives a bijection with facets for $\car_1$.  We put $C_\ell:=\mathbf{a}_\ell-\rho^\vee$. The bijection between sets of alcoves allows us to define right and left actions of $W$ on these sets, together with the Bruhat order $\leq$ and the periodic order $\preceq$, by transport of structure from the set of alcoves for $\car_1$ (or equivalently, by seeing every alcove for $\bullet_\ell$, resp. for $\car_\ell$, as a $W$-translate of $C_\ell$, resp. of $\mathbf{a}_\ell$). Note that if $A,B$ are two alcoves for $\bullet_\ell$, then we have $A\preceq B\Leftrightarrow A\uparrow B$, where the order $\uparrow$ is defined in \cite[§6]{jantzen2003representations}: this is an immediate consequence of  \cite[II, §6.6, (4)]{jantzen2003representations} (which says that $\uparrow$ is invariant under translation) together with \cite[II, Lemma 10.1]{achar2018reductive} (which implies that $\uparrow$ coincides with the Bruhat order inside of $\mathscr{C}^+_0-\rho^\vee$, and thus with $\preceq$). For any alcove $C$ for $\bullet_\ell$, we define the alcove $\hat{C}$ for $\bullet_\ell$, resp. the integer $d(C)$, by transport of structure using once again the bijection between alcoves for $\bullet_\ell$ and $\car_1$ (notice that  $d(C)$ is then the number of hyperplanes of $\mathscr{H}'$ separating $C$ from $C_\ell$).
	
	For any facet $\mathbf{h}$ for $\bullet_\ell$, we will denote by $W_{\mathbf{h},\bullet_\ell}\subset W$ the stabilizer of $\mathbf{h}$ for $\bullet_\ell$. One can easily check that we have 
	$$W_{\mathbf{h},\bullet_\ell}=W_{\mathbf{h}+\rho^\vee,\car_\ell}=W_{\ell^{-1}\cdot(\mathbf{h}+\rho^\vee)}, $$
	where the second and third sets denote stabilizers for $\car_\ell$ and $\car_1$ respectively.	
	
	Let $\lambda\in\mathbb{X}_{+}^\vee$ and $\mathbf{h}\subset E$ be the facet for $\bullet_\ell$ containing $\lambda$, which is of the form 
	\begin{multline*}
		\mathbf{h}=\{\mu\in E~|~\langle\mu+\rho^\vee,\alpha\rangle=\ell\cdot n_\alpha~\forall\alpha\in\mathfrak{R}^0_+(\mathbf{h}),\\
		\ell\cdot(n_\alpha-1)<\langle\mu+\rho^\vee,\alpha\rangle<\ell\cdot n_\alpha~\forall\alpha\in\mathfrak{R}^1_+(\mathbf{h})\} 
	\end{multline*} for suitable integers $n_\alpha$ and a partition $\mathfrak{R}_+=\mathfrak{R}^0_+(\mathbf{h})\sqcup\mathfrak{R}^1_+(\mathbf{h})$. If we let $C_\lambda$ be the alcove for $\bullet_\ell$ defined by the integers $(n_\alpha)_{\alpha\in\mathfrak{R}_+}$, then $C_\lambda$ is the only alcove satisfying
	$$\lambda\in \widetilde{C}_\lambda:=\{\mu\in E~|~
	\ell\cdot(n_\alpha-1)<\langle\mu+\rho^\vee,\alpha\rangle\leq\ell\cdot n_\alpha~\forall\alpha\in\mathfrak{R}_+\}.$$
	The set $\widetilde{C}_\lambda$ (which is denoted by $\widehat{C}_\lambda$ in \cite{jantzen2003representations}) is called the upper closure of $C_\lambda$. The alcove $C_\lambda$ can be characterized as the only alcove containing $\mathbf{h}$ in its closure which is minimal in $W_{\mathbf{h},\bullet_\ell}\bullet_\ell C_\lambda$ for the order $\preceq$ (cf. \cite[§6.11]{jantzen2003representations}). Now, let us denote by $\mu$ the element of $W\bullet_\ell\lambda$ contained in $\overline{C_\ell}$, by $\mathbf{g}'\subset\overline{\mathbf{a}_\ell}$ the facet containing $\mu+\rho^\vee$ and by $w$ the element of ${_\mathrm{f}W}^{\ell^{-1}\cdot\mathbf{g}'}$ such that $w\car_\ell(\mu+\rho^\vee)=\lambda+\rho^\vee$ (cf. the first isomorphism of (\ref{param dom weights})). Then the alcove $A:=\rho^\vee+C_\lambda$ is the only alcove for $\car_\ell$ containing $\lambda+\rho^\vee$ in its closure which is minimal in 
	$$W_{\mathbf{h}+\rho^\vee,\car_\ell} A=W_{w\mathbf{g}',\car_\ell}A=AW_{\mathbf{g}',\car_\ell}$$ 
	for $\preceq$. Likewise, the alcove $A_w:=\ell^{-1}\cdot A$ is the only alcove (for $\car_1$) containing $\ell^{-1}\cdot (\mathbf{h}+\rho^\vee)$ in its closure which is minimal in $A_wW_{\mathbf{g}}$ for $\preceq$, where $\mathbf{g}:=\ell^{-1}\cdot\mathbf{g}'$. By definition of the operation $\hat{}$ on the set of alcoves for $\bullet_\ell$, we have $\hat{C}_\lambda=\ell\cdot\hat{A}_w-\rho^\vee$, and so $\hat{C}_\lambda-\ell\cdot\rho^\vee=\ell\cdot(\hat{A}_w-\rho^\vee)-\rho^\vee$. In particular (once again by definition of the application $d(\cdot)$ on the set of alcoves for $\bullet_\ell$), we have
	\begin{equation}\label{distance}
		d(\hat{C}_\lambda-\ell\cdot\rho^\vee)=d(\hat{A}_w-\rho^\vee). 
	\end{equation}
	\subsection{A new proof of Donkin's Theorem}\label{A new proof of Donkin's Theorem}
	As an application of the study that was made in the previous sections, we can give a new proof of the description of blocks of $\mathrm{Rep}_\mathbf{k}(G^\vee)$. Unless specified otherwise, $G$ is assumed to be a semi-simple algebraic group of adjoint type over $\mathbb{F}$ (so $G^\vee$ is simply connected). For any $\lambda\in\mathbb{X}^\vee$, we define $r(\lambda)$ to be the unique non-negative integer such that $$\lambda\in\ell^{r(\lambda)}\cdot\mathbb{X}^\vee\backslash~ \ell^{r(\lambda)+1}\cdot\mathbb{X}^\vee.$$
	If we let $\lambda'$ be the $W$-conjugate (for the action $\car_\ell$) of $\lambda$ contained in $\overline{\mathbf{a}_\ell}$, $\mathbf{g}_{\lambda'}\subset \overline{\mathbf{a}_\ell}$ be the facet containing $\lambda'$, put $\mathbf{g}:=\ell^{-1}\cdot\mathbf{g}_{\lambda'}$ and let  $w$ be the unique element of  ${_\mathrm{f}W}^{\mathbf{g}}$ such that $w\car_\ell\lambda'=\lambda$ (cf. (\ref{param dom weights})), then one can easily check that $$r(\lambda)=r_\mathbf{g}(w)+1,$$ where $r_\mathbf{g}(w)$ is as it was defined in subsection \ref{section special}.
	
	We take back the setting of the beginning of subsection \ref{Facets which are not points} for the next statement: since $G^\vee$ is semi-simple and simply connected, we get a decomposition $G^\vee=G^\vee_1\times\cdots\times G_t^\vee$ into simply connected simple algebraic groups, each $G^\vee_i$ admitting $\mathfrak{R}^\vee_i$ as a root system; this decomposition induces a decomposition of the root system, affine Weyl group and dominant characters attached to $G^\vee$. Moreover, for every $i$, we will denote by $\rho_i^\vee\in\mathbb{X}^\vee_{i,+}$ the half sum of positive coroots relative to $G^\vee_i$ and, for any positive integer $r$, we will denote by $W_i^{(r)}$ the subgroup of $W_i$ whose translation part has been dilated by $\ell^r$ (cf. subsection \ref{section special} for the precise definition).
	\begin{thm}\label{thm Donkin}
	Let $\mu=(\mu_1,\cdots,\mu_r)\in \mathbb{X}^\vee_+=\prod_i \mathbb{X}^\vee_{i,+}$, and denote by $\overline{\mu}$ the equivalence class of $\mu$ for the equivalence relation $\sim$ (cf. section \ref{relation section}) on $\mathbb{X}^\vee_+$ (seen as the weight poset of $\mathrm{Rep}_\mathbf{k}(G^\vee)$). We have $$\overline{\mu}=\prod_{i=1}^r W_i^{(r(\mu_i+\rho^\vee_i))}\bullet_{\ell}\mu_i\cap\mathbb{X}^\vee_{i,+}.$$ 
	\end{thm}
	\begin{proof}
		Recall that thanks to the geometric Satake equivalence (Theorem \ref{satake equivalence}) and Theorem \ref{IW equivalence}, we have an equivalence of highest weight categories
		\begin{equation}\label{IW satake}
			\mathrm{Rep}_\mathbf{k}(G^\vee)\xrightarrow{\sim}\mathrm{Perv}_{\mathcal{IW}}(\mathrm{Gr},\mathbf{k})
		\end{equation}
		sending a tilting module $T_\mu$ to the tilting Iwahori-Whittaker perverse sheaf $\mathscr{T}^{\mathcal{IW}}_{\lambda}$, where $\lambda:=\mu+\rho^\vee$. Therefore, proving the claim is equivalent to showing that for any $\lambda\in\mathbb{X}^\vee_{++}=\mathbb{X}^\vee_{+}+\rho^{\vee}=\prod_i(\mathbb{X}^\vee_{i,+}+\rho^\vee_i)$, we have the equality
		\begin{equation}\label{equation equivalente}
			\overline{\lambda}=\prod_{i=1}^rW_i^{(r(\lambda_i))}\car_{\ell}\lambda_i\cap\mathbb{X}^\vee_{i,++},
		\end{equation}
		where $\overline{\lambda}$ now denotes the equivalence class of $\lambda=(\lambda_1,\cdots,\lambda_r)$ for the equivalence relation $\sim$ on $\mathbb{X}^\vee_{++}$, seen as the weight poset of $\mathrm{Perv}_{\mathcal{IW}}(\mathrm{Gr},\mathbf{k})$.
		
		Denote by $\lambda'$ the unique $W$-conjugate (for the dilated box action $\car_\ell$) of $\lambda$ which is contained in $\overline{\mathbf{a}_\ell}$, by $\mathbf{g}_{\lambda'}\subset \overline{\mathbf{a}_\ell}$ the facet containing $\lambda'$, put $\mathbf{g}:=\ell^{-1}\cdot \mathbf{g}_{\lambda'}\subset \overline{\mathbf{a}_1}$ and let $w$ the element of ${_\mathrm{f}W}^{\mathbf{g}}$ such that $w\car_\ell\lambda'=\lambda$. Recall that ${_\mathrm{f}W}^{\mathbf{g}}\car_\ell\lambda'=W\car_\ell\lambda'\cap\mathbb{X}^\vee_{++}$. By (\ref{eq relation W}), we have
		$$(w\car_\ell\lambda')\sim (w'\car_\ell\lambda')\Longleftrightarrow w\sim_\mathbf{g}w',~\forall w,w'\in  {_\mathrm{f}W}^{\mathbf{g}},$$
		and by Proposition \ref{reduction to irred} we have 
		$$w\sim_\mathbf{g}w'\Longleftrightarrow w_i\sim_{\mathbf{g}_i}w'_i~\forall i, $$
		where $w=(w_1,\cdots,w_r)$, $w'=(w'_1,\cdots,w'_r)$.
		
		On the other hand, the linkage principle tells us that 
		$$(w\car_\ell\lambda')\sim \mu\Rightarrow \mu\in W\car_\ell\lambda',$$
		from which we deduce (also using Lemma \ref{param Wr}) that proving (\ref{equation equivalente}) is equivalent to proving that 
		$$\overline{w}=\prod_{i=1}^rW_i^{(r(\lambda_i))}w_iW_{\mathbf{g}_i}\cap {_\mathrm{f}W_i}^{\mathbf{g}_i}.$$
		But this last equality was proved in Theorem \ref{cas final} (because $r(\lambda_i)=r_{\mathbf{g}_i}(w_i)+1$ for all $i$).
	\end{proof}
We can also apply our Corollary \ref{bound} to give a bound on the length of a minimum chain linking two weights in the same block. We start with a lemma.
	\begin{lem}\label{lem dilation}
		Let $\lambda,\mu\in \mathbb{X}_+^\vee$, $r\in\mathbb{Z}_{\geq0}$, and put $\lambda':=\ell^r\cdot(\lambda+\rho^\vee)-\rho^\vee$, $\mu':=\ell^r\cdot(\mu+\rho^\vee)-\rho^\vee$. We have 
		$$(T_{\lambda'}:\nabla_{\mu'})= (T_{\lambda}:\nabla_{\mu}).$$
	\end{lem}
	\begin{proof}
		By the geometric Satake equivalence (Theorem \ref{satake equivalence}) coupled with Theorem \ref{IW equivalence}, we have an equivalence of highest weight categories $$\mathrm{Rep}_\mathbf{k}(G^\vee)\xrightarrow{\sim}\mathrm{Perv}_{\mathcal{IW}}(\mathrm{Gr},\mathbf{k}),$$
		sending $T_\gamma$ to $\mathscr{T}^\mathcal{IW}_{\gamma+\rho^\vee}$ for every $\gamma\in\mathbb{X}^\vee_+$. We can thus apply Corollary \ref{red to non special2}:
		$$(T_{\lambda'}:\nabla_{\mu'})=(\mathscr{T}^\mathcal{IW}_{\ell^r\cdot(\lambda+\rho^\vee)}:\nabla^\mathcal{IW}_{\ell^r\cdot (\mu+\rho^\vee)})\overset{\ref{red to non special2}}{=}(\mathscr{T}^\mathcal{IW}_{\lambda+\rho^\vee},\nabla^\mathcal{IW}_{\mu+\rho^\vee} )= (T_{\lambda}:\nabla_{\mu}).$$
	\end{proof}
	\begin{prop}\label{bound group}
		Assume that $G^\vee$ is a simple and simply connected algebraic group over $\mathbf{k}$. Let $\lambda,\lambda'$ be two elements of $\mathbb{X}^\vee_+$ in the same equivalence class for $\sim$, and denote by $\tilde{\lambda}$ (resp. $\tilde{\lambda'}$) the unique element of $\mathbb{X}^\vee_+$ which satisfies $\lambda+\rho^\vee=\ell^r\cdot(\tilde{\lambda}+\rho^\vee)$ (resp. $\lambda'+\rho^\vee=\ell^r\cdot(\tilde{\lambda'}+\rho^\vee)$), where $r=r(\lambda+\rho^\vee)$ (notice that $\lambda=\tilde{\lambda}$ when $r=0$). Also denote by $C_{\tilde{\lambda}}$ (resp. $C_{\tilde{\lambda'}}$) the alcove containing ${\tilde{\lambda}}$ (resp. ${\tilde{\lambda'}}$) in its upper closure. Then there exists a chain of dominant characters 
		$$\lambda_s=\lambda,\lambda_{s-1},\cdots,\lambda_0=\lambda' $$
		such that, for all $i\in\llbracket 0,s-1\rrbracket$, there exists an indecomposable $G^\vee$-module $M_i$ satisfying
		$$[M_i:L_{\lambda_i}]\neq 0~\text{and}~[M_i:L_{\lambda_{i+1}}]\neq 0 $$
		and such that 
		$$s\leq 2+d(\hat{C}_{\tilde{\lambda}}-\ell\cdot\rho^\vee)+d(\hat{C}_{\tilde{\lambda'}}-\ell\cdot\rho^\vee). $$
		
	\end{prop}
	\begin{proof}
		Denote by $\mu$ the $W$-conjugate (for $\bullet_\ell$) of $\lambda$ which is included in $\overline{C_\ell}$. We will first deal with the case where $r=0$, which means that the facets (for $\car_\ell$) containing $\lambda+\rho^\vee$ and $\lambda'+\rho^\vee$ are not special (notice that, since $\lambda$ and $\lambda'$ are in the same equivalence class, we must have $r(\lambda+\rho^\vee)=r(\lambda'+\rho^\vee)$ thanks to Theorem \ref{thm Donkin}). 
		
		We denote by $\mathbf{g}_\mu\subset\overline{\mathbf{a}_\ell}$ the facet containing $\mu+\rho^\vee$, put $\mathbf{g}:=\ell^{-1}\cdot \mathbf{g}_\mu$ and let $w,w'$ be the elements of ${_\mathrm{f}W^{\mathbf{g}}}$ satisfying $w\car_\ell(\mu+\rho^\vee)=\lambda+\rho^\vee$, $w'\car_\ell(\mu+\rho^\vee)=\lambda'+\rho^\vee$. The facet $\mathbf{g}$ for $\car_1$ is a non-special facet included in the closure of $\overline{\mathbf{a}_1}$. Let us also denote by $A_w$, resp. $A_{w'}$, the alcove containing $w\mathbf{g}$, resp. $w'\mathbf{g}$, in its closure and which is minimal in $A_wW_\mathbf{g}$, resp. $A_{w'}W_\mathbf{g}$, for the order $\preceq$. Thus, we can apply Corollary \ref{bound}, and pick a chain of elements of ${_\mathrm{f}W}^\mathbf{g}$
		$$w_s=w,w_{s-1},\cdots,w_0=w' $$
		as in this corollary, with $s\leq 2+d(\hat{A}_w-\rho^\vee)+d(\hat{A}_{w'}-\rho^\vee)$. Thanks to the equation (\ref{distance}), we have that $d(\hat{A}_w-\rho^\vee)=d(\hat{C}_{\lambda}-\ell\cdot\rho^\vee)$ and $d(\hat{A}_{w'}-\rho^\vee)=d(\hat{C}_{\lambda'}-\ell\cdot\rho^\vee)$, so that $s$ is bounded by the desired integer.
		This same corollary tells us that for all $i\in\llbracket 0,s-1\rrbracket$, there exists an element $u_i\in {_\mathrm{f}W}^\mathbf{g}$ satisfying 
		$$n_{w_i,u_i}(1)\neq 0~\text{and}~n_{w_{i+1},u_i}(1)\neq 0. $$
		By Proposition \ref{car 0}, we can replace $n$ with ${^\ell n}$ in the above, so that the character formula of tilting modules given in \cite[Theorem 8.9]{RW22} yields 
		$$(T_{u_i\bullet_\ell\mu}:\nabla_{w_i\bullet_\ell\mu})\neq0~\text{and}~(T_{u_i\bullet_\ell\mu}:\nabla_{w_{i+1}\bullet_\ell\mu})\neq0, $$
		which implies that 
		$$[T_{u_i\bullet_\ell\mu}:L_{w_i\bullet_\ell\mu}]\neq0~\text{and}~[T_{u_i\bullet_\ell\mu}:L_{w_{i+1}\bullet_\ell\mu}]\neq0. $$
		So we get the result by putting $\lambda_i=w_i\bullet_\ell\mu$ and $M_i=T_{u_i\bullet_\ell\mu}$.
		
		Now we assume that $r>0$. Since $r(\tilde{\lambda}+\rho)=r(\tilde{\lambda'}+\rho)=0$, we can apply the previous step and find two sequences of dominant characters $(\tilde{\lambda}_i),~(\tilde{\nu}_i)$ of the desired length such that $\tilde{\lambda}_s=\tilde{\lambda},~\tilde{\lambda}_0=\tilde{\lambda'}$ and 
		$$(T_{\tilde{\nu}_i}:\nabla_{\tilde{\lambda}_i})\neq0~\text{and}~(T_{\tilde{\nu}_i}:\nabla_{\tilde{\lambda}_{i+1}})\neq0. $$
		By Lemma \ref{lem dilation}, we see that if we put $\nu_i:=\ell^r\cdot(\tilde{\nu}_i+\rho)-\rho$ and $\lambda_i:=\ell^r\cdot(\tilde{\lambda}_i+\rho)-\rho$, we get:
		$$(T_{\nu_i}:\nabla_{\lambda_i})=(T_{\tilde{\nu}_i}:\nabla_{\tilde{\lambda}_i})\neq0~\text{and}~(T_{\nu_i}:\nabla_{\lambda_{i+1}})=(T_{\tilde{\nu}_i}:\nabla_{\tilde{\lambda}_{i+1}})\neq0, $$
		from which we deduce that
		$$[T_{\nu_i}:L_{\lambda_i}]\neq0~\text{and}~[T_{\nu_i}:L_{\lambda_{i+1}}]\neq0. $$
		This concludes the proof.
	\end{proof}
We will now use the process described in  \cite[§II.7.3]{jantzen2003representations} to deduce from Theorem \ref{thm Donkin} the block decomposition of $\mathrm{Rep}_\mathbf{k}(G^\vee)$ in the general case where $G^\vee$ is a reductive group\footnote{We warn the reader that the statement (3) given in \cite[§II.7.3]{jantzen2003representations} is wrong, since one needs to assume the semi-simple group of \textit{loc. cit.} to be simple for it to be correct.}, starting with a lemma.
	\begin{lem}\label{central extension}
		Let $H_1,H_2$ be reductive algebraic groups over $\mathbf{k}$, with $T_1\subset B_1$, resp. $T_2\subset B_2$, a maximal torus and a Borel subgroup of $H_1$, resp. of $H_2$, and $\varphi:H_1\to H_2$ a central isogeny such that $\varphi(T_1)=T_2$, $\varphi(B_1)=B_2$. Denote by $\mathbb{X}(T_1)_+\subset\mathbb{X}(T_1)$ (resp. $\mathbb{X}(T_2)_+\subset\mathbb{X}(T_2)$) the dominant characters and characters associated with $T_1\subset B_1$ (resp. $T_2\subset B_2$). 
		
		Then $\varphi$ induces an injective morphism  $\mathbb{X}(T_2)_+\hookrightarrow\mathbb{X}(T_1)_+$, and the blocks of $H_2$ are the blocks of $H_1$ contained in $\mathbb{X}(T_2)_+$.
	\end{lem}
	\begin{proof}
		Since $\varphi$ is a central isogeny, the morphism induced on root data identifies the root systems of $H_1$ and $H_2$, (cf. \cite[§II.1.17]{jantzen2003representations}). Therefore, the injection $\mathbb{X}(T_2)\hookrightarrow\mathbb{X}(T_1)$ (induced by $\varphi|_{T_1}$) induces an injection $\mathbb{X}(T_2)_+\hookrightarrow\mathbb{X}(T_1)_+$. On the other hand, pulling back by $\varphi$ induces a fully-faithful functor $\varphi^*:\mathrm{Rep}_\mathbf{k}(H_2)\to\mathrm{Rep}_\mathbf{k}(H_1)$, which sends an indecomposable tilting module of highest weight $\lambda$ to the indecomposable tilting module of same highest weight (cf. \cite[§E.7]{jantzen2003representations}), and such that there exists no non-zero morphism between an indecomposable tilting module in the essential image of $\varphi^*$ and an indecomposable tilting module of $\mathrm{Rep}_\mathbf{k}(H_1)$ which is not in this essential image. The conclusion follows easily, thanks to the description of blocks via tilting modules (cf. Theorem \ref{relations thm}).
	\end{proof}

	\begin{coro}\label{coro 44}
Assume that $G$ is a general reductive group. Denote by $\mathcal{D}G^\vee$ the derived subgroup of $G^\vee$, by $T_2$ the reduced part of the neutral connected component of the centre of $G^\vee$ and let $H_1,\cdots,H_t$ be the simply connected covers of the minimal closed connected normal subgroups of positive dimension of $\mathcal{D}G^\vee$. For each $i\in \llbracket 1,t\rrbracket$, also denote by $T^\vee_i$ the split maximal torus of $H_i$ determined by the split maximal torus of $G^\vee$ which is Langlands dual to $T$, by $W_i$ (resp. by $\mathbb{X}^\vee_{i,+}$, resp. by $\rho^\vee_i$) the affine Weyl group (resp. the set of dominant characters of $T^\vee_i$ determined by $\mathfrak{R}_+^\vee$, resp. the half-sum of positive roots) associated with $H_i$, seen as a subgroup of the affine Weyl group associated with $\mathcal{D}G^\vee$, and by $\mathbb{X}(T_2)$ the group of characters of $T_2$. Then there is a central isogeny
		$$\varphi:\mathcal{D}G^\vee\times T_2\to G^\vee, $$
		which induces an injection $\mathbb{X}^\vee_+\hookrightarrow\prod_i\mathbb{X}^\vee_{i,+}\times \mathbb{X}(T_2)$ and such that, for every $$\lambda=(\lambda_1,\cdots,\lambda_t,\lambda_{t+1})\in\mathbb{X}^\vee_+ ,$$ we have
		$$\overline{\lambda}=\prod_{i=1}^tW_i^{(r(\lambda_i+\rho_i^\vee))}\bullet_\ell\lambda_i\cap \mathbb{X}^\vee_{i,+}\times\{\lambda_{t+1}\}.$$
	\end{coro}
	\begin{proof}
		The fact that $\varphi$ exists follows from \cite[II.1.18]{jantzen2003representations}. Moreover, we know that $\mathcal{D}G^\vee$ is a semi-simple algebraic group, so that we have a central isogeny $H'_1\times\cdots\times H'_t\to \mathcal{D}G^\vee$, where $H'_1,\cdots,H'_t$ are the minimal closed normal subgroups of positive dimension of $\mathcal{D}G^\vee$, which are simple algebraic groups. We then get a central isogeny $H_1\times\cdots\times H_t\to \mathcal{D}G^\vee$ by replacing each $H_i'$ with its simply connected cover. Therefore, Lemma \ref{central extension} and Theorem \ref{thm Donkin} tell us that for any dominant character $\lambda'$ of the maximal torus $T\cap\mathcal{D}G^\vee$ of $\mathcal{D}G^\vee$, the associated block of $\mathrm{Rep}_\mathbf{k}(\mathcal{D}G^\vee)$ is equal to 
		$$\prod_{i=1}^tW_i^{(r(\lambda'_i+\rho^\vee_i))}\bullet_\ell\lambda'_i\cap \mathbb{X}^\vee_{i,+}, $$
		where we see $\lambda'$ as the element $(\lambda'_1,\cdots,\lambda'_t)$ of $\prod_i\mathbb{X}^\vee_{i,+}$ (notice that the above set is included in the set of dominant characters of $T\cap\mathcal{D}G^\vee$, because each $W_i$ is a subgroup of the affine Weyl group associated with $\mathcal{D}G^\vee$). 
		
		Finally, since $T_2$ is a torus (in particular, the category $\mathrm{Rep}_\mathbf{k}(T_2)$ is semi-simple), one can easily check that the block of $\lambda$ in $\mathcal{D}G^\vee\times T_2$ is equal to the product of blocks associated to $(\lambda_1,\cdots,\lambda_t)$ and $\lambda_{t+1}$ in $\mathrm{Rep}_\mathbf{k}(\mathcal{D}G^\vee)$ and $\mathrm{Rep}_\mathbf{k}(T_2)$ respectively, i.e. to
		$$\prod_{i=1}^tW_i^{(r(\lambda_i+\rho^\vee_i))}\bullet_\ell\lambda_i\cap \mathbb{X}^\vee_{i,+}\times\{\lambda_{t+1}\}. $$
		This concludes the proof by Lemma \ref{central extension}, since the above set is included in $\mathbb{X}^\vee_+$.
	\end{proof}
	\begin{rem}
		We take back the context of Corollary \ref{coro 44}, and let $\lambda,\lambda'$ be dominant characters in the same equivalence class. The result of Proposition \ref{bound group} can also be generalized to the case where $G^\vee$ is a general reductive group. Indeed, using the central isogeny $\mathcal{D}G^\vee\times T_2\to G^\vee$, one is reduced to proving it for $\mathcal{D}G^\vee$, and since $\mathcal{D}G^\vee$ is a central isogeny of the product of $t$ simple subgroups (because $\mathcal{D}G^\vee$ is semi-simple), we can further reduce to the case where $G^\vee$ is a product of simply connected simple groups $H_1\times\cdots\times H_t$. Finally, the equality (\ref{352}) of Proposition \ref{reduction to irred} allows us to bound the length of a minimum chain linking two weights $(\lambda_1,\cdots,\lambda_t)$ and $(\lambda_1',\cdots,\lambda_t')$ in the same block by $\mathrm{max}\{s_i,~i\in \llbracket1,t\rrbracket\}$, where $s_i$ is the bound obtained in  Proposition \ref{bound group} for the length of a minimum chain linking $\lambda_i$ and $\lambda'_i$.
	\end{rem}

	\subsection{Block decomposition for a quantum group}\label{Block decomposition of a quantum group}
	In this subsection, $G$ is semi-simple of adjoint type. We assume that $\mathrm{char}(\mathbf{k})=0$, and that there exists a primitive $\ell$-th root of unity $q$ in $\mathbf{k}$. Moreover, we assume that $\ell$ is odd, greater than the Coxeter number of $\mathfrak{R}^\vee$ and not equal to $3$ if $\mathfrak{R}^\vee$ has a component of type $G_2$. We then denote by $U_{q,\mathbf{k}}$ Lusztig's quantized enveloping algebra specialized at $q$ and associated with $G^\vee$ (we take the conventions of \cite[§H]{jantzen2003representations}). The category of finite dimensional representations of $U_{q,\mathbf{k}}$, which we will denote by $\mathrm{Rep}(U_{q,\mathbf{k}})$, has many features in common with the category $\mathrm{Rep}(G^\vee_{\mathbf{k}'})$, where $G^\vee_{\mathbf{k}'}$ denotes the Langlands dual group of $G$ over some field $\mathbf{k}'$ of characteristic $\ell$. In particular, $\mathrm{Rep}(U_{q,\mathbf{k}})$ is a highest weight category with weight poset $\mathbb{X}^\vee_+$; for any $\lambda\in\mathbb{X}^\vee_+$, we will denote by $T_q(\lambda)$, resp. $\nabla_q(\lambda)$, the indecomposable tilting module, resp. the costandard object, with highest weight $\lambda$. 
	
	We fix a weight $\lambda\in\overline{C}_\ell\cap\mathbb{X}^\vee$, let $\mathbf{g}'$ be the facet for $\car_\ell$ containing $\lambda+\rho^\vee$ and $w,w'$ be elements of ${_\mathrm{f}W}^{\mathbf{g}}$, where $\mathbf{g}:=\ell^{-1}\cdot\mathbf{g}'$ (recall the bijections (\ref{param dom weights})). Most importantly for us, the multiplicity of costandard objects in tilting objects is known.
	\begin{prop}\label{character quantiq}
	We have
	$$(T_q(w\bullet_\ell\lambda):\nabla_q(w'\bullet_\ell\lambda))=n_{w',w}(1). $$
	Moreover, we have that $(T_q(w\bullet_\ell\lambda):\nabla_q(\lambda'))=0$ whenever $\lambda'\notin W\bullet_\ell\lambda$.
	\end{prop}
	\begin{proof}
	 The second claim is due to the linkage principle for quantum groups, see \cite[§8]{andersen1991representations}.
	 
	Let us denote by $W_{(0)}$ the subset of $W$ consisting of elements $w$ which are minimal in $wW_0$, and recall that the action of $\mathrm{Iw}$ by left multiplication on $\mathrm{Fl}_{\mathbf{a}_1}^\circ$, resp. on $\mathrm{Gr}^\circ$, yields stratifications 
	$$(\mathrm{Fl}_{\mathbf{a}_1}^\circ)_{\mathrm{red}}=\bigsqcup_{u\in W}X_u,\qquad(\mathrm{Gr}^\circ)_{\mathrm{red}}=\bigsqcup_{v\in W_{(0)}}Y_v, $$
	where $Y_u:=\mathrm{Iw}\cdot\dot{u}$, resp. $X_v:=\mathrm{Iw}\cdot\dot{v}$, is an affine $\mathbb{F}$-space of dimension $l(u)$, resp. $l(v)$. Moreover, the canonical projection $\pi:\mathrm{Fl}_{\mathbf{a}_1}^\circ\to \mathrm{Gr}^\circ$ is ind-proper, $\mathrm{Iw}$-equivariant and satisfies $\pi^{-1}(Y_w)=\bigsqcup_{z\in W_0}X_{wz}$ for every $w\in W_{(0)}$. We denote by $\mathrm{Perv}_{(\mathrm{Iw})}(\mathrm{Fl}_{\mathbf{a}_1}^\circ,\mathbf{k})$, resp. $\mathrm{Perv}_{(\mathrm{Iw})}(\mathrm{Gr}^\circ,\mathbf{k})$, the category of perverse sheaves on $\mathrm{Fl}_{\mathbf{a}_1}^\circ$, resp. $\mathrm{Gr}^\circ$, constant along the $\mathrm{Iw}$-orbits, and by $\mathrm{Rep}_{\overline{0}}(U_{q,\mathbf{k}})$ the Serre subcategory of $\mathrm{Rep}(U_{q,\mathbf{k}})$ generated by the simple objects whose highest weight belongs to $W\bullet_\ell0\cap\mathbb{X}^\vee_{+}$ (the fact that $\mathrm{Rep}_{\overline{0}}(U_{q,\mathbf{k}})$ is truly a block will follow from Proposition \ref{block quantique}). 
	
	Thanks to \cite{arkhipov2004quantum}, we have an equivalence of highest weight categories 
	$$\mathrm{Rep}_{\overline{0}}(U_{q,\mathbf{k}})\simeq \mathrm{Perv}_{(\mathrm{Iw})}(\mathrm{Gr}^\circ,\mathbf{k}),$$
	sending $T_q(x\bullet_\ell 0)$ (resp. $\nabla_q(y\bullet_\ell0)$) to the indecomposable tilting object $T^{(\mathrm{Iw})}_{x^{-1}}$ (resp. to the costandard object $\nabla^{(\mathrm{Iw})}_{y^{-1}}$) associated with $x$, for any $x\in{_\mathrm{f}W}$ (resp. $y\in{_\mathrm{f}W}$)\footnote{Notice that $w\mapsto w^{-1}$ induces a bijection ${_\mathrm{f}W}\simeq W_{(0)}$.}. 
	
	For any $u\in W$, let us denote by $\mathscr{T}_u\in \mathrm{Perv}_{(\mathrm{Iw})}(\mathrm{Fl}_{\mathbf{a}_1}^\circ,\mathbf{k})$ the indecomposable tilting object associated with $u$; by \cite[Proposition 3.4.1]{yun2009weights}, we have an isomorphism $\pi_!\mathscr{T}_w\simeq T^{(\mathrm{Iw})}_{w}$ for any $w\in W_{(0)}$. Therefore, combining \cite[(3.4.1)]{yun2009weights} with \cite[Theorem 5.3.1]{yun2009weights}, we get 
	$$(T^{(\mathrm{Iw})}_{u}: \nabla^{(\mathrm{Iw})}_{v})=\sum_{z\in W_0}(-1)^{l(z)}h_{vz,u}(1)~ \forall u,v\in W_{(0)}, $$
	where $(h_{x,y},~x,y\in W)$ denotes the ordinary Kazhdan-Lusztig polynomials associated with $W$, as described in \cite{Soergel1997KazhdanLusztigPA}. Using \cite[Proposition 3.4]{Soergel1997KazhdanLusztigPA} together with the fact that $h_{x,y}=h_{x^{-1},y^{-1}}$ for all $x,y\in W$ (which is a consequence of the fact that the anti-automorphism $i$ from the proof of \cite[Theorem 2.7]{Soergel1997KazhdanLusztigPA} commutes with the involution $d:\mathcal{H}\to \mathcal{H}$ defined in \textit{loc. cit.}) show that the right-hand side in the above equation coincides with $n_{v^{-1},u^{-1}}(1)$. So we get 
	$$(T_q(w\bullet_\ell0):\nabla_q(w'\bullet_\ell0))=(T^{(\mathrm{Iw})}_{w^{-1}}: \nabla^{(\mathrm{Iw})}_{w'^{-1}})=n_{w',w}(1),$$
	which proves our claim in the case where $\lambda=0$. Using the second point of \cite[Remark 7.2]{Soergel1997KazhdanLusztigPA}, one can deduce from this the general case. \end{proof}
	\begin{rem}
	The character formula of Proposition \ref{character quantiq} was originally stated as a conjecture in \cite[Conjecture 7.1]{Soergel1997KazhdanLusztigPA}, and proved in \cite{soergel1998character} for $\ell>33$.
	\end{rem}
	In particular, the second assertion of Proposition \ref{character quantiq} tells us that $\overline{\mu}\subset W\bullet_\ell\mu$ for any $\mu\in\mathbb{X}^\vee_+$.
	\begin{coro}
    We have an equality
	$$\mathrm{dim}_\mathbf{k}\mathrm{Hom}_{\mathrm{Rep}(U_{q,\mathbf{k}})}(T_q(w\bullet_\ell\lambda),T_q(w'\bullet_\ell\lambda))=\mathrm{dim}_\mathbf{k}\mathrm{Hom}^\bullet_{\mathrm{Par}_{\mathcal{IW}}(\mathrm{Fl}^{\circ}_{\mathbf{g}},\mathbf{k})}(\mathcal{E}_w^{\mathbf{g}},\mathcal{E}_{w'}^{\mathbf{g}}).$$
	\end{coro}
	\begin{proof}
	 Proposition \ref{character quantiq} and standard arguments (see \cite[§6.2]{achar2016modular}) show that we have
	\begin{equation}\label{last}
	    \mathrm{dim}_\mathbf{k}\mathrm{Hom}_{\mathrm{Rep}(U_{q,\mathbf{k}})}(T_q(w\bullet_\ell\lambda),T_q(w'\bullet_\ell\lambda))=\sum_{y\in {_\mathrm{f}W}^\mathbf{g}}n_{y,w}(1)\cdot n_{y,w'}(1). 
	\end{equation}
	But we also know that 
	$$\mathrm{dim}_\mathbf{k}~\mathrm{Hom}^\bullet_{D^b_{\mathcal{IW}}(\mathrm{Fl}^{\circ}_{\mathbf{g}},\mathbf{k})}(\mathcal{E}_x^\mathbf{g},\nabla^\mathbf{g}_y)=n_{y,x}(1)~\forall x,y\in{_\mathrm{f}W}^\mathbf{g}$$
	thanks to (\ref{l-KL}) together with Lemma \ref{lem1}. So by \cite[Proposition 2.6]{2014} the right-hand side in (\ref{last}) coincides with $$\mathrm{dim}_\mathbf{k}\mathrm{Hom}^\bullet_{\mathrm{Par}_{\mathcal{IW}}(\mathrm{Fl}^{\circ}_{\mathbf{g}},\mathbf{k})}(\mathcal{E}_w^{\mathbf{g}},\mathcal{E}_{w'}^{\mathbf{g}}).$$
	\end{proof}
Therefore, we get a very similar situation as the one in subsection \ref{Consequences on equivalence relations}:
$$w\bullet_\ell\lambda\sim w'\bullet_\ell\lambda \Longleftrightarrow w\sim_\mathbf{g} w'.$$
For any $\mu\in\mathbb{X}^\vee_+$, recall the definition of $r(\mu)$ from subsection \ref{A new proof of Donkin's Theorem}, and put $$\delta(\mu):=\begin{cases}
				1 & \text{if}\ r(\mu)=0 \\
				0 & \text{othewise.}\
			\end{cases}$$ 
			We define $W^\delta$ to be equal to $W$ when $\delta=1$, and to $\{\mathrm{id}\}$ otherwise. The next result is obtained from Theorem \ref{cas final} in the same way that Theorem \ref{thm Donkin} was, from which we take back the notations.
\begin{prop}\label{block quantique}
Let $\mu=(\mu_1,\cdots,\mu_r)\in \mathbb{X}^\vee_+=\prod_i \mathbb{X}^\vee_{i,+}$, and denote by $\overline{\mu}$ the equivalence class of $\mu$ in $\mathbb{X}^\vee_+$ (seen as the weight poset of $\mathrm{Rep}(U_{q,\mathbf{k}})$) for the equivalence relation $\sim$ from  section \ref{relation section}. We have $$\overline{\mu}=\prod_{i=1}^r W_i^{\delta(\mu_i+\rho^\vee_i)}\bullet_{\ell}\mu_i\cap\mathbb{X}^\vee_{i,+}.$$ 
\end{prop}
\begin{rems}\begin{enumerate}
    \item The proof of the previous proposition does not require Smith-Treumann theory (because the proof of Theorem \ref{cas final} does not require it when $\mathrm{char}(\mathbf{k})=0$).
    \item One can use Corollary \ref{bound} to give a bound for the length of a minimal chain linking two weights in the same block for $\mathrm{Rep}(U_{q,\mathbf{k}})$, of the same kind as Proposition \ref{bound group}. 
    \item Another proof of Proposition \ref{block quantique} was found in \cite{thams1994blocks} (under the same assumptions on $\ell$). However, the results of \textit{loc. cit.} use the proof of Donkin from \cite{donkin1980blocks}, so it does not allow to give a bound as in the previous point.
\end{enumerate}

\end{rems}

	\bibliographystyle{plain}
	\bibliography{biblio1}

@phdthesis{riche,
	TITLE = {{Geometric Representation Theory in positive characteristic}},
	AUTHOR = {Riche, Simon},
	URL = {https://tel.archives-ouvertes.fr/tel-01431526},
	SCHOOL = {{Universit{\'e} Blaise Pascal (Clermont Ferrand 2)}},
	YEAR = {2016},
	TYPE = {Habilitation {\`a} diriger des recherches},
	HAL_ID = {tel-01431526},
	HAL_VERSION = {v1},
}

@article{2014,
   title={Parity sheaves},
   volume={27},
   ISSN={1088-6834},
   url={http://dx.doi.org/10.1090/S0894-0347-2014-00804-3},
   DOI={10.1090/s0894-0347-2014-00804-3},
   number={4},
   journal={J. Amer. Math. Soc},
   publisher={American Mathematical Society (AMS)},
   author={Juteau, Daniel and Mautner, Carl and Williamson, Geordie},
   year={2014},
   month={May},
   pages={1169–1212}
}

@article{achar2021geometric,
      title={{A geometric Steinberg formula}}, 
      author={Pramod N. Achar and Simon Riche},
      year={2021},
      eprint={2109.11980},
      archivePrefix={arXiv},
      primaryClass={math.RT}
}

@misc{ciappara2021hecke,
      title={{Hecke category actions via Smith-Treumann theory}}, 
      author={Joshua Ciappara},
      year={2021},
      eprint={2103.07091},
      archivePrefix={arXiv},
      primaryClass={math.RT}
}

@article{PAPPAS2008118,
title = {Twisted loop groups and their affine flag varieties},
journal = {Adv. Math.},
volume = {219},
number = {1},
pages = {118-198},
year = {2008},
issn = {0001-8708},
doi = {https://doi.org/10.1016/j.aim.2008.04.006},
url={https://www.sciencedirect.com/science/article/pii/S0001870808001205},
author = {George Pappas and Michael Rapoport},
}

@article{JEP_2019__6__707_0,
     author = {Roman Bezrukavnikov and Dennis Gaitsgory and Ivan Mirkovi\'c and Simon Riche and Laura Rider},
     title = {An {Iwahori-Whittaker} model for the {Satake} category},
     journal = {J. \'E. polytech. {\textemdash} Math.},
     pages = {707--735},
     publisher = {\'Ecole polytechnique},
     volume = {6},
     year = {2019},
     doi = {10.5802/jep.104},
     zbl = {07114040},
     language = {en},
     howpublished = "\url{https://jep.centre-mersenne.org/articles/10.5802/jep.104/}"
}

@book{jantzen2003representations,
  title={{Representations of Algebraic Groups}},
  author={Jantzen, Jens Carsten},
  isbn={9780821835272},
  lccn={03058381},
  series={Math. Surveys Monogr.},
  year={2003},
  publisher={American Mathematical Society}
}

@article{Mirkovic2004GeometricLD,
  title={{Geometric Langlands duality and representations of algebraic groups over commutative rings}},
  author={Ivan Mirkovi\'c and Kari Vilonen},
  journal={Ann. of Math.},
  year={2004},
  volume={166},
  pages={95-143}
}

@article{epiga:4984,
  TITLE = {{The parabolic exotic t-structure}},
  AUTHOR = {Pramod N. Achar and Nicholas Cooney and Simon Riche},
  url = {https://epiga.episciences.org/4984},
  DOI = {10.46298/epiga.2018.volume2.4520},
  JOURNAL = {{Épijournal Géom. Algébrique}},
  VOLUME = {{Volume 2}},
  YEAR = {2018},
  MONTH = Nov,
}

@incollection{baumann:hal-01491529,
  TITLE = {{Notes on the geometric Satake equivalence}},
  AUTHOR = {Baumann, Pierre and Riche, Simon},
  URL = {https://hal.archives-ouvertes.fr/hal-01491529},
  BOOKTITLE = {{Relative aspects in Representation Theory, Langlands Functoriality and Automorphic Forms}},
  EDITOR = {V. Heiermann, D. Prasad},
  PUBLISHER = {{Springer}},
  SERIES = {Lecture Notes in Mathematics},
  VOLUME = {2221},
  PAGES = {1-134},
  YEAR = {2018},
  DOI = {10.1007/978-3-319-95231-4\_1},
  HAL_ID = {hal-01491529},
  HAL_VERSION = {v3},
}

@article{article,
	author = {De Visscher, Maud},
	year = {2008},
	month = {02},
	pages = {952-965},
	title = {On the blocks of semisimple algebraic groups and associated generalized {S}chur algebras},
	volume = {319},
	journal = {J. Algebra},
	doi = {10.1016/j.jalgebra.2007.11.015}
}

@article{Soergel1997KazhdanLusztigPA,
	title={Kazhdan-{L}usztig polynomials and a combinatoric for tilting modules},
	author={Wolfgang Soergel},
	journal={Represent. Theory},
	year={1997},
	volume={1},
	pages={83-114}
}

@article{LUSZTIG1980121,
	title = {Hecke algebras and {J}antzen's generic decomposition patterns},
	journal = {Adv. Math.},
	volume = {37},
	number = {2},
	pages = {121-164},
	year = {1980},
	issn = {0001-8708},
	doi = {https://doi.org/10.1016/0001-8708(80)90031-6},
	url = {https://www.sciencedirect.com/science/article/pii/0001870880900316},
	author = {George Lusztig}
}

@article{williamson2012modular,
  title={Modular intersection cohomology complexes on flag varieties},
  note={With an appendix by T.Braden},
  author={Williamson, Geordie},
  journal={Math. Z},
  volume={272},
  number={3},
  pages={697--727},
  year={2012},
  publisher={Springer}
}

@book{cdi_springer_books_10_1007_978_3_540_34491_9,
author = {Bourbaki, Nicolas},
address = {Berlin, Heidelberg},
booktitle = {Groupes et algèbres de Lie},
copyright = {Springer-Verlag Berlin Heidelberg 2007},
isbn = {354034490X},
language = {fre},
publisher = {Springer Berlin Heidelberg},
title = {Groupes et algèbres de Lie: Chapitres 4 à 6},
}

@book{knus2012quadratic,
  title={Quadratic and Hermitian forms over rings},
  note={With a foreword by I. Bertuccioni},
  author={Knus, Max-Albert},
  volume={294},
  year={1991},
  publisher={Springer-Verlag}
}

@article{gordon1982graded,
  title={Graded artin algebras},
  author={Gordon, Robert and Green, Edward L},
  journal={J. Algebra},
  volume={76},
  number={1},
  pages={111--137},
  year={1982},
  publisher={Academic Press}
}

@book{achar2021perverse,
  title={Perverse {S}heaves and {A}pplications to {R}epresentation {T}heory},
  author={Achar, Pramod N.},
  volume={258},
  year={2021},
  publisher={American Mathematical Soc.}
}

@book{humphreys1975linear,
  title={Linear {A}lgebraic {G}roups},
  author={Humphreys, James Edward},
  isbn={9780387901084},
  lccn={95154182},
  series={Graduate Texts in Mathematics},
  year={1975},
  publisher={Springer}
}

@inproceedings{achar2016modular,
  title={Modular perverse sheaves on flag varieties {I}: tilting and parity sheaves},
    note={With a joint appendix by Geordie Williamson},
  author={Achar, Pramod N. and Riche, Simon},
  booktitle={ Ann. {S}ci. {\'E}c. {N}orm. {S}up{\'e}r.},
  volume={49},
  pages={325--370},
  year={2016}
}

@article{beilinson2018faisceaux,
  title={Faisceaux pervers},
  author={Beilinson, Alexander and Bernstein, Joseph and Deligne, Pierre and Gabber, Ofer},
  journal={Ast{\'e}risque},
  volume={100},
  year={2018}
}

@article{achar2022geometric,
  title={A geometric model for blocks of {F}robenius kernels},
  author={Achar, Pramod N. and Riche, Simon},
  eprint={2203.03530},
  archivePrefix={arXiv},
  year={2022}
}

@misc{stacks-project,
  author       = {The {Stacks project authors}},
  title        = {The Stacks project},
  howpublished = {\url{https://stacks.math.columbia.edu}},
  year         = {2022},
}

@article{donkin1980blocks,
  title={The blocks of a semisimple algebraic group},
  author={Donkin, Stephen},
  journal={J. Algebra},
  volume={67},
  number={1},
  pages={36--53},
  year={1980},
  publisher={Academic Press}
}

@book{milne2017algebraic,
  title={Algebraic groups: the theory of group schemes of finite type over a field},
  author={Milne, James S.},
  volume={170},
  year={2017},
  publisher={Cambridge University Press}
}

@article{riche2018tilting,
  title={Tilting modules and the p-canonical basis},
  author={Riche, Simon and Williamson, Geordie},
  journal={Ast{\'e}risque},
  volume={397},
  year={2018}
}

@article{andersen1986inversion,
  title={An inversion formula for the {K}azhdan-{L}usztig polynomials for affine {W}eyl groups},
  author={Andersen, Henning Haahr},
  journal={Adv. Math.},
  volume={60},
  number={2},
  pages={125--153},
  year={1986},
  publisher={Academic Press}
}

@article{andersen1991representations,
  title={Representations of quantum algebras},
  author={Andersen, Henning Haahr and Polo, Patrick and Kexin, Wen},
  journal={ Invent. Math.},
  volume={104},
  number={1},
  pages={1--59},
  year={1991},
  publisher={Springer}
}

@article{arkhipov2004quantum,
  title={Quantum groups, the loop {G}rassmannian, and the {S}pringer resolution},
  author={Arkhipov, Sergey and Bezrukavnikov, Roman and Ginzburg, Victor},
  journal={J. Amer. Math. Soc.},
  volume={17},
  number={3},
  pages={595--678},
  year={2004}
}

@article{yun2009weights,
  title={Weights of mixed tilting sheaves and geometric {R}ingel duality},
  author={Yun, Zhiwei},
  journal={Selecta Math.},
  volume={14},
  number={2},
  pages={299--320},
  year={2009},
  publisher={Springer}
}

@article{humphreys1978blocks,
  title={Blocks and indecomposable modules for semisimple algebraic groups},
  author={Humphreys, James Edward and Jantzen, Jens Carsten},
  journal={J. Algebra},
  volume={54},
  number={2},
  pages={494--503},
  year={1978},
  publisher={Academic Press}
}

@misc{Lievis,
author = {Gibson, Joel},
title = {Lievis},
howpublished={\url{https://www.jgibson.id.au/lievis/}}
}

@article{soergel1998character,
  title={Character formulas for tilting modules over Kac-Moody algebras},
  author={Soergel, Wolfgang},
  journal={Represent. Theory 2},
  volume={2},
  number={13},
  pages={432--448},
  year={1998}
}

@article{achar2018reductive,
  title={Reductive groups, the loop {G}rassmannian, and the {S}pringer resolution},
  author={Achar, Pramod N. and Riche, Simon},
  journal={Invent. Math.},
  volume={214},
  number={1},
  pages={289--436},
  year={2018},
  publisher={Springer}
}

@article{thams1994blocks,
  title={The blocks of a quantum algebra},
  author={Thams, Lars},
  journal={Comm. Algebra},
  volume={22},
  number={5},
  pages={1617--1628},
  year={1994},
  publisher={Taylor \& Francis}
}

@article {Treu19,
    AUTHOR = {Treumann, David},
     TITLE = {Smith theory and geometric {H}ecke algebras},
   JOURNAL = {Math. Ann.},
  FJOURNAL = {Mathematische Annalen},
    VOLUME = {375},
      YEAR = {2019},
    NUMBER = {1-2},
     PAGES = {595--628},
      ISSN = {0025-5831,1432-1807},
   MRCLASS = {20C08 (14L30 32B20)},
  MRNUMBER = {4000251},
MRREVIEWER = {Shoumin\ Liu},
       DOI = {10.1007/s00208-019-01860-1},
       URL = {https://doi.org/10.1007/s00208-019-01860-1},
}

@article {Les21,
    AUTHOR = {Leslie, Spencer and Lonergan, Gus},
     TITLE = {Parity sheaves and {S}mith theory},
   JOURNAL = {J. Reine Angew. Math.},
  FJOURNAL = {Journal f\"{u}r die Reine und Angewandte Mathematik. [Crelle's
              Journal]},
    VOLUME = {777},
      YEAR = {2021},
     PAGES = {49--87},
      ISSN = {0075-4102,1435-5345},
   MRCLASS = {14F08 (13F35 14L30 14M15)},
  MRNUMBER = {4292864},
MRREVIEWER = {Wahei\ Hara},
       DOI = {10.1515/crelle-2021-0018},
       URL = {https://doi.org/10.1515/crelle-2021-0018},
}

@article {RW22,
    AUTHOR = {Riche, Simon and Williamson, Geordie},
     TITLE = {Smith-{T}reumann theory and the linkage principle},
   JOURNAL = {Publ. Math. Inst. Hautes \'{E}tudes Sci.},
  FJOURNAL = {Publications Math\'{e}matiques. Institut de Hautes \'{E}tudes
              Scientifiques},
    VOLUME = {136},
      YEAR = {2022},
     PAGES = {225--292},
      ISSN = {0073-8301,1618-1913},
   MRCLASS = {20G05 (14D24 14F08 22E57)},
  MRNUMBER = {4517647},
MRREVIEWER = {H.\ H.\ Andersen},
       DOI = {10.1007/s10240-022-00134-y},
       URL = {https://doi.org/10.1007/s10240-022-00134-y},
}

\end{document}